\documentclass[letterpapaer, 11pt]{amsart}
\usepackage{amsmath,amssymb,amsthm,pinlabel,tikz,hyperref,mathrsfs,color, thmtools}
\usepackage{versions}
\DeclareSymbolFontAlphabet{\amsmathbb}{AMSb}%
\usepackage{verbatim}
\usepackage{float}
\usepackage{caption}
\usepackage{subcaption}
\usepackage{enumitem}
\newcommand{\nc}{\newcommand}
\nc{\dmo}{\DeclareMathOperator}
\dmo{\ra}{\rightarrow}
\dmo{\Prob}{\mathbb{P}}
\dmo{\E}{\mathbb{E}}
\dmo{\N}{\mathbb{N}}
\dmo{\Z}{\mathbb{Z}}
\dmo{\Q}{\mathbb{Q}}
\dmo{\R}{\mathbb{R}}
\dmo{\C}{\mathcal{C}}
\dmo{\X}{\mathcal{X}}
\dmo{\U}{\mathcal{U}}
\dmo{\T}{\mathcal{T}}
\dmo{\F}{\mathcal{F}}
\dmo{\AC}{\mathcal{AC}}
\dmo{\w}{\omega}
\dmo{\MIN}{\mathcal{MIN}}
\dmo{\Mod}{Mod}
\dmo{\PMod}{PMod}
\dmo{\PMF}{\mathcal{PMF}}
\dmo{\Mat}{Mat}
\dmo{\supp}{supp}
\dmo{\UE}{\mathcal{UE}}
\dmo{\vol}{vol}
\dmo{\B}{B}
\dmo{\PB}{PB}
\dmo{\PR}{PSL(2,\mathbb{R})}
\dmo{\GL}{GL(k, \mathbb{C})}
\dmo{\SL}{SL(2, \mathbb{Z})}
\dmo{\Isom}{Isom}
\dmo{\RP}{\mathbb{R} \mathrm{P}}
\dmo{\I}{\mathcal{I}}
\dmo{\el}{\ell_{\C}}
\dmo{\NN}{\mathcal{N}}
\dmo{\rk}{rank}
\dmo{\tr}{tr}
\dmo{\llangle}{\langle\langle}
\dmo{\rrangle}{\rangle\rangle}
\dmo{\Unif}{Unif}
\dmo{\Out}{Out}
\dmo{\diam}{\operatorname{diam}}
\dmo{\Aut}{\operatorname{Aut}}
\dmo{\sumRho}{\mathscr{B}}
\dmo{\stopping}{\vartheta}
\dmo{\diffPivot}{\mathcal{P}}
\dmo{\diffEvenPivot}{\mathcal{Q}}
\dmo{\varGam}{\Upsilon}
\dmo{\prodSeq}{\Pi}
\dmo{\NSupp}{N_{supp}}
\dmo{\KSleep}{\mathnormal{K_{sleep}}}
\dmo{\Devi}{\upsilon}
\dmo{\DeviStr}{\iota}
\dmo{\DeviDisc}{\varrho}
\dmo{\sphere}{\mathcal{S}}
\dmo{\DeviUni}{\varsigma}
\dmo{\wVar}{\check{\w}}
\dmo{\muVar}{\eta}
\dmo{\sublinear}{\Delta}

\usetikzlibrary{decorations.markings, patterns}
\usetikzlibrary{decorations.pathmorphing}
\usetikzlibrary{calc, arrows.meta, positioning}
\tikzset{->-/.style={decoration={
  markings,
  mark=at position #1 with {\arrow{>}}},postaction={decorate}}}

\nc{\nt}{\newtheorem}

\nt{theorem}{Theorem}

\newtheorem{thm}{{\bf Theorem}}[section]
\newtheorem{conv}[thm]{{\bf Convention}}
\newtheorem{lem}[thm]{{\bf Lemma}}
\newtheorem{cor}[thm]{{\bf Corollary}}
\newtheorem{prop}[thm]{{\bf Proposition}}

\newtheorem{claim}[thm]{Claim}

\newtheorem{definition}[thm]{Definition}

\numberwithin{equation}{section}
\newtheorem{obs}[thm]{Observation}

\title[Random walks and contracting elements II]{Random walks and contracting elements II: Translation length and Quasi-isometric embedding}

\date{\today}
\author{Inhyeok Choi}

\address{%
		June E Huh Center for Mathematical Challenges, KIAS\\
		85 Hoegiro Dongdaemun-gu, Seoul 02455, Republic of Korea
}
\email{%
        inhyeokchoi48@gmail.com
        }

\begin{document}

\begin{abstract}
Continuing from \cite{choi2022random1}, we study random walks on metric spaces with contracting elements. We prove that random subgroups of the isometry group of a metric space is quasi-isometrically embedded into the space. We discuss this problem in two ways, namely, in the sense of random walks and counting problem. We also establish the genericity of contracting elements and the CLT and its converse for translation length.


\noindent{\bf MSC classes:} 20F67, 30F60, 57M60, 60G50
\end{abstract}

\maketitle

\section{Introduction}\label{section:intro}

This is the second in a series of articles concerning random walks on metric spaces with contracting elements. This series is a reformulation of the previous preprint \cite{choi2022limit} announced by the author, aiming for clearer and more concise exposition. 

Let $X$ be a geodesic metric space and let $o \in X$. We say that a subset $A$ of $X$ is \emph{contracting} if the closest point projection of a geodesic $\gamma$ onto $A$ is uniformly bounded whenever $\gamma$ is far away from $A$. An isometry $g$ of $X$ is \emph{contracting} if the orbit $\{g^{n} o\}_{n \in \Z}$ is a contracting quasigeodesic (see Definition \ref{dfn:contracting}). Two contracting isometries $g$ and $h$ of $X$ are \emph{independent} if there orbits have unbounded Hausdorff distance. 

\begin{conv}\label{conv:main}
Throughout, we assume that: \begin{itemize}
\item $(X, d)$ is a geodesic metric space;
\item $G$ is a countable group of isometries of $X$, and
\item $G$ contains two independent contracting isometries.
\end{itemize}
We also fix a basepoint $o \in X$.
\end{conv}

This type of contraction takes place in various non-positively curved spaces including (relatively) hyperbolic spaces, CAT(0) spaces, Teichm{\"u}ller space, Culler-Vogtmann Outer space and small cancellation groups. We refer the readers to \cite{sisto2013projections}, \cite{bestvina2009higher}, \cite{minsky1996quasi-projections}, \cite{algom-kfir2011strongly} and \cite{arzhantseva2019negative} for the treatment for each cases.

We wish to sample a random isometry $g$ in $G$. For this we fix i.i.d. random variables $g_{i}$'s distributed according to a probability measure $\mu$ and define $Z_{n} := g_{1} \cdots g_{n}$. The sequence of RVs $(Z_{n})_{n\ge 0}$ then constitute the random walk generated by $\mu$. We say that the random walk $(Z_{n})_{n\ge0}$ is \emph{non-elementary} if the support of $\mu$ generates a semigroup that contains two independent contracting isometries (see Subsection \ref{subsection:RW}).

Given a non-elementary random walk on $G$, one can ask the asymptotic behavior of the displacement $d(o, go)$ of a random isometry $g$. This was indeed pursued in \cite{choi2022random1}, discussing the strong law of large numbers (SLLN) with an exponential bound, central limit theorem (CLT) and the law of the iterated logarithm (LIL). In this article, we aim to investigate yet another quantity, the translation length $\tau(g):= \lim_{i} d_{X}(o, g^{i} o)/i$, of a random isometry $g$. The translation length of an element $g$ of $G$ often encodes interesting dynamical phenomena. For example, the translation length of a deck transformation of a negatively curved manifold is the length of the corresponding closed geodesic. The translation length of a pseudo-Anosov mapping class (irreducible outer automorphism, resp.) on Teichm{\"u}ller space (Outer space, resp.) equals its stretch factor as a self-map on the surface (expansion factor as a train-track map on the graph, resp.).

Our first main result is companion to the corresponding result for displacement (\cite[Theorem 6.4]{choi2022random1}).

\begin{theorem} \label{thm:expBd}
Let $(X, G, o)$ be as in Convention \ref{conv:main} and let $(Z_{n})_{n \ge 0}$ be the random walk generated by a non-elementary measure $\mu$ on $G$. Let $\lambda(\mu)$ be the escape rate of $\mu$, i.e., $\lambda(\mu) := \lim_{n \rightarrow \infty} \E_{\mu^{\ast n}} [d(o, go)] / n$. Then for each $0 < L < \lambda(\mu)$, there exists $K>0$ such that for each $n$ we have  \[
\Prob\Big(\textrm{$Z_{n}$ is contracting and $\tau(Z_{n}) \ge Ln$}\Big) \ge 1-K e^{-n/K}.
\]
\end{theorem}

This result has been observed by Sisto for simple random walks on various spaces in \cite{sisto2018contracting}. In the absence of moment conditions, Maher and Tiozzo observed in \cite{maher2018random} that non-elementary random walks on Gromov hyperbolic spaces favor loxodromic elements in probability. Their methods and Benoist-Quint's estimates in \cite{benoist2016central} also lead to the stronger SLLN for translation length under finite second moment condition, as noted by Dahmani and Horbez \cite{dahmani2018spectral}. Dahmani and Horbez also deduced the same SLLN on Teichm{\"u}ller space. Later, Baik, Choi and Kim obtained the same SLLN under finite first moment assumption using ergodic theorems and Maher-Tiozzo's notion of persistent joint \cite{baik2021linear}. We also note Le Bars' recent result \cite{le-bars2022central} that non-elementary random walks on a proper CAT(0) space favor contracting isometries in probability. Theorem \ref{thm:expBd} generalizes these results by obtaining an exponential bound from below without any moment condition. Finally, Goldsborough and Sisto recently presented a QI-invariant theory for exponential genericity of positive translation length \cite{goldsborough2021markov}. Their theory applies to (relatively) hyperbolic groups, acylindrically hyperbolic 3-manifold groups, right-angled Artin groups and many more.

Genericity of contracting elements is a recurring theme that has been investigated in various settings. For example, Yang describes genericity of contracting elements in counting problem for proper actions on a metric space \cite{yang2020genericity}. The novelty of Theorem \ref{thm:expBd} is that it discusses genericity of contracting elements for possibly non-WPD actions on spaces and for non-elementary random walks without moment condition.

We also provide a quantitative comparison between the displacement and the translation length of a random isometry.

\begin{theorem}\label{thm:discrepancy}
Let $(X, G, o)$ be as in Convention \ref{conv:main}, and $(Z_{n})_{n\ge 0}$ be the random walk generated by a non-elementary measure $\mu$ on $G$. \begin{enumerate}
\item If $\mu$ has finite $p$-th moment for some $p > 0$, then \[
\lim_{n \rightarrow \infty} \frac{1}{n^{1/2p}} [d(o, Z_{n} o)- \tau(Z_{n}) ]= 0 \quad \textrm{a.s.}
\]
\item If $\mu$ has finite first moment, then there exists $K>0$ such that \[
\limsup_{n \rightarrow \infty} \frac{1}{\log n} [d(o, Z_{n} o)- \tau(Z_{n}) ] \le K \quad \textrm{a.s.}
\]
\end{enumerate}
\end{theorem}

There have been many results that capture the sublinear discrepancy between the displacement and translation length for random walks with bounded support or with finite exponential moment: see \cite{maher2012exp}, \cite{maher2018random}. Moreover, using Benoist-Quint's strategy in \cite{benoist2016central}, \cite{benoist2016lin} and its application to other spaces (\cite{horbez2018clt}, \cite{dahmani2018spectral}), one can achieve sublinear discrepancy for random walks with finite first moment. We improve these observations by proving that random walks with finite $(1/2)$-th moment exhibit sublinear discrepancy between displacement and translation length.

Combining the CLT for displacement proved in \cite{choi2022random1} and Theorem \ref{thm:discrepancy}, we deduce the CLT for translation length:

\begin{theorem}[CLT and its converse]\label{thm:CLT}
Let $(X, G, o)$ be as in Convention \ref{conv:main}, and $(Z_{n})_{n\ge 0}$ be the random walk generated by a non-elementary measure $\mu$ on $G$. If $\mu$ has finite second moment, then $\frac{1}{\sqrt{n}} (d(o, Z_{n} o) - n \lambda)$ and $\frac{1}{\sqrt{n}} (\tau(Z_{n}) - n \lambda)$ converge to the same Gaussian distribution in law.

Conversely, if $\mu$ has infinite second moment, then for any sequence $(c_{n})_{n\ge 0}$, neither $\frac{1}{\sqrt{n}}(d(o, Z_{n}o) - c_{n})$ nor $\frac{1}{\sqrt{n}} (\tau(Z_{n}) - c_{n})$ converges in law.
\end{theorem}

CLT for translation length and the converses of CLT for Gromov hyperbolic spaces and Teichm{\"u}ller space have been observed in \cite{choi2021clt}. Here, we establish the same result for general spaces with contracting isometries.

Meanwhile, Taylor and Tiozzo proved in \cite{taylor2016random} that random subgroups of a weakly hyperbolic group is quasi-isometrically embedded into the ambient Gromov hyperbolic space, in the sense that such event happens for eventual probability 1. See also \cite{maher2019random} and \cite{maher2021random} for additional conclusions under geometric assumptions, e.g., acylindricity or WPD. These results are linked to a deeper understanding of convex-cocompact subgroup of mapping class groups and random extensions of surface groups and free groups.

The following theorem strengthens the conclusion of Taylor-Tiozzo's theorem and generalizes it to more general spaces.

\begin{theorem}\label{thm:qi} 
Let $(X, G, o)$ be as in Convention \ref{conv:main}, and $(Z_{n}^{(1)}, \ldots, Z_{n}^{(k)})_{n \ge 0}$ be $k$ independent random walks generated by a non-elementary measure $\mu$ on $G$. Then there exists $K>0$ such that \[
\Prob \left[ \langle Z_{n}^{(1)}, \ldots, Z_{n}^{(k)} \rangle \,\,\textrm{is q.i. embedded into a quasi-convex subset of $X$} \right] \ge 1- K e^{-n/K}.
\]
\end{theorem}

Thanks to concrete control of the decay rate, we can deduce the analogous conclusion for counting problems. 

\begin{theorem}\label{thm:qiCount}
Let $G$ be a finitely generated group acting on a metric space $X$ with at least two independent contracting elements. Then for each $k > 0$, there exists a finite generating set $S$ of $G$ such that \[
\frac{\#\left\{(g_{1}, \ldots, g_{k}) \in \big(B_{S}(n) \big)^{k} :\begin{array}{c}  \langle g_{1}, \ldots, g_{k} \rangle \,\, \textrm{is q.i. embedded into} \\ \textrm{a quasi-convex subset of $X$} \end{array} \right\} } {\big(\# B_{S}(n)\big)^{k}}
\]
converges to 1 exponentially fast.
\end{theorem}

Theorem \ref{thm:qiCount} has been previously observed for a vast number of group actions by Han and Yang. First, Yang describes the genericity of contracting elements for counting problem in \cite{yang2020genericity}, which corresponds to the case of $k=1$ in Theorem \ref{thm:qiCount}. Yang considered groups admitting a statistically convex-cocompact (SCC) actions on proper spaces, which include relatively hyperbolic group, groups with non-trivial Floyd boundary, non-splitting RAAGs and RACGs, small cancellation groups, mapping class groups acting on Teichm{\"u}ller spaces and many more. This was later generalized to arbitrary $k$ by Han and Yang in \cite{han2021generic}. 

Meanwhile, mapping class groups are not known to possess a contracting element with respect to the actions on their Cayley graphs. This necessitates a different approach for the counting problem in the Cayley graph. Theorem \ref{thm:qiCount} accomplishes this purpose by not requiring the action of $G$ on $X$ be proper or SCC, while deducing a result for the counting problem in the Cayley graph of $G$. 
For Gromov hyperbolic spaces and Teichm{\"u}ller space, Theorem \ref{thm:qiCount} for $k=1$ has been observed in \cite{choi2021generic} as an affirmative answer to a version of Farb's conjecture in \cite{farb2006problems}. 

\subsection{Structure of the paper}\label{subsection:structure}

In Section \ref{section:prelim}, we recall the basic notions and facts about contracting isometries, alignment of paths and Schottky sets that were discussed in \cite{choi2022random1}.

All the theorems of this paper eventually follow from Gou{\"e}zel's pivoting technique \cite{gouezel2022exponential}. Among them, Theorem \ref{thm:expBd}, \ref{thm:discrepancy} and the first half of Theorem \ref{thm:CLT} (CLT) use pivoting technique implicitly via referring to the results of \cite{choi2022random1}. This part is explained in Section \ref{section:firstMethod}, which only assumes the content of Section \ref{section:prelim} and does not require any knowledge about pivotal times.

The other theorems are deduced using pivoting technique that we explain in Section \ref{section:secondMethod}. In Subsection \ref{subsection:pivotDiscrete}, we recall the pivotal time construction for a discrete model. In Subsection \ref{subsection:pivotReduce}, we consider two lemmata to reduce a random walk into the discrete model. Given this preparation, we relate pivotal times with translation length and genericity of contracting isometries in Subsection \ref{subsection:pivot2}. In Subsection \ref{subsection:CLTConverse}, we combine these ingredients and deduce the converse of CLT. In Subsection \ref{subsection:pivotIndep}, we generalize the discussion in Subsection \ref{subsection:pivotReduce} to discuss the quasi-isometric embedding of $k$ independent random walks.

It remains to discuss counting problem. Our strategy is to compare the exponential bound for simple random walks on $G$ and a counting estimate on $G$. For this, we establish a quantitative version of Subsection \ref{subsection:pivotReduce} in Section \ref{section:effective}. We prove Theorem \ref{thm:qiCount} using this in Section \ref{section:counting}.

\subsection*{Acknowledgments}
The author thanks Hyungryul Baik, Talia Fern{\'o}s, Ilya Gekhtman, Thomas Haettel,  Joseph Maher, Hidetoshi Masai, Catherine Pfaff, Yulan Qing, Kasra Rafi,  Samuel Taylor and Giulio Tiozzo for helpful discussions. The author is also grateful to the American Institute of Mathematics and the organizers and the participants of the workshop ``Random walks beyond hyperbolic groups'' in April 2022 for helpful and inspiring discussions. Finally, the author thanks the anonymous referee for providing valuable comments that enabled huge improvement of the paper.

The author is supported by Samsung Science \& Technology Foundation (SSTF-BA1702-01 and SSTF-BA1301-51) and by a KIAS Individual Grant (SG091901) via the June E Huh Center for Mathematical Challenges at KIAS. This work constitutes part of the author’s PhD thesis. The revision was conducted while the author was participating in the program ``Randomness and Geometry" at the Fields Institute under the support of the Marsden Postdoctoral Fellowship.

\section{Preliminaries}\label{section:prelim}

In this section, we summarize the language and the setting of \cite{choi2022random1}. Throughout the paper, $X$ is a geodesic metric space with basepoint $o$ and $G$ is a countable isometry group of $X$.

\subsection{Contracting isometries and alignment}\label{subsection:contracting}

\begin{definition}[contracting sets]\label{dfn:contracting}
For a subset $A \subseteq X$ and $\epsilon > 0$, we define the \emph{closest point projection} of $x \in X$ to $A$ by \[
\pi_{A}(x) := \big\{a \in A : d_{X}(x, a)= d_{X}(x, A) \big\}.
\]
$A$ is said to be \emph{$K$-contracting} if: \begin{enumerate}
\item $\pi_{A}(z) \neq \emptyset$ for all $z\in X$ and
\item for all $x, y \in X$ such that $d_{X}(x, y) \le d_{X}(x, A)$ we have \[
\diam_{X}\big(\pi_{A}(x) \cup \pi_{A}(y)\big) \le K.
\]
\end{enumerate}

A $K$-contracting $K$-quasigeodesic is called a \emph{$K$-contracting axis}. An isometry $g \in G$ is said to be \emph{$K$-contracting} if its orbit $\{g^{n} o\}_{n \in \Z}$ is a $K$-contracting axis.
\end{definition}

\begin{definition}[Translation length]\label{dfn:trLength}
For $g \in G$, the \emph{(asymptotic) translation length} of $g$ is defined by \[
\tau(g) := \liminf_{n \rightarrow \infty} \frac{1}{n} d(o, g^{n} o).
\]
\end{definition}

Given an isometry $g \in G$, the orbit $\{g^{n} o\}_{n \in \Z}$ is a quasi-geodesic if and only if $g$ has strictly positive translation length. However, an isometry with positive translation length may not be contracting.

A \emph{path} is a map from either an interval or a set of consecutive integers to $X$. Given a path $\gamma : [a, b] \rightarrow X$, we define its \emph{reversal} $\bar{\gamma} : [a, b] \rightarrow X$ by the map $\bar{\gamma}(t) := \gamma(a+b - t)$.

A subset $A \subseteq X$ is said to be \emph{$K$-quasi-convex} if any geodesic $[x, y]$ connecting two points $x, y \in A$ is contained in the $K$-neighborhood of $A$. Two paths $\kappa$ and $\eta$ on $X$ are said to be \emph{$K$-fellow traveling} if their beginning points and ending points are pairwise $K$-near and $d_{Hauss}(\kappa, \eta) < K$.

Let us now recall some alignment lemmata established in \cite{choi2022random1}.

\begin{definition}[{\cite[Definition 3.5]{choi2022random1}}]\label{dfn:alignment}
For $i=1, \ldots, n$, let $\gamma_{i}$ be a path on $X$ whose beginning and ending points are $x_{i}$ and $y_{i}$, respectively. We say that $(\gamma_{1}, \ldots, \gamma_{n})$ is $C$-aligned if \[
\diam_{X}\big(y_{i} \cup \pi_{\gamma_{i}}(\gamma_{i+1})\big) < C, \quad \diam_{X}\big(x_{i+1} \cup \pi_{\gamma_{i+1}} (\gamma_{i}) \big) < C
\]
hold for $i = 1, \ldots, n-1$.
\end{definition}

Here, we regard points as degenerate paths. For example, for a point $x$ and a path $\gamma$, we say that $(x, \gamma)$ is $C$-aligned  if \[
\diam_{X}(\textrm{beginning point of $\gamma$} \cup \pi_{\gamma}(x)) < C
\] holds. The following observation is immediate: 

\begin{obs}\label{obs:align}
Let $g$ be an isometry of $X$. Let $n$ be a positive integer and let $k$ be an integer in $\{1, \ldots, n\}$. Let $K > 0$ and let $\gamma_{1}, \ldots, \gamma_{n}$ be paths on $X$. \begin{enumerate}
\item If $(\gamma_{1}, \ldots, \gamma_{k})$ and $(\gamma_{k}, \ldots, \gamma_{n})$ are $K$-aligned, then their concatenation $(\gamma_{1}, \ldots, \gamma_{n})$ is also $K$-aligned.
\item If $(\gamma_{1}, \ldots, \gamma_{n})$ is $K$-aligned, then $(\bar{\gamma}_{n}, \ldots, \bar{\gamma}_{1})$ is also $K$-aligned.
\item If $(\gamma_{1}, \ldots, \gamma_{n})$ is $K$-aligned, then $(g\gamma_{1}, \ldots, g\gamma_{n})$ is also $K$-aligned.
\end{enumerate}
\end{obs}

\begin{lem}[{\cite[Lemma 3.7]{choi2022random1}}]\label{lem:1segment}
For each $C>0$ and $K>1$, there exists $D= D(K, C)>K, C$ that satisfies the following.

Let $\gamma$ and $\gamma'$ be $K$-contracting axes whose ending points are $y$ and $y'$, respectively. If $(y, \gamma')$ and $(\gamma, y')$ are $C$-aligned, then $(\gamma, \gamma')$ is $D$-aligned.
\end{lem}

\begin{lem}[{\cite[Proposition 3.11]{choi2022random1}}]\label{lem:BGIPWitness}
For each $D>0$ and $K>1$, there exist $E = E(K, D)>K, D$ and $L = L(K,D)>K, D$ that satisfy the following. 

Let $x$ and $y$ be points in $X$ and let $\gamma_{1}, \ldots, \gamma_{n}$ be $K$-contracting axes whose domains are longer than $L$ and such that $(x, \gamma_{1}, \ldots, \gamma_{n}, y)$ is $D$-aligned. Then the geodesic $[x, y]$ has subsegments $\eta_{1}, \ldots, \eta_{n}$, in order from left to right, that are longer than $100E$ and such that $\eta_{i}$ and $\gamma_{i}$ are $0.1E$-fellow traveling for each $i$.  In particular, $(x, \gamma_{i}, y)$ are $E$-aligned for each $i$.
\end{lem}

\begin{lem}[{\cite[Proposition 2.9]{yang2019statistically}}]\label{lem:BGIPConcat}
For each $D, M>0$ and $K>1$, there exist $E = E(K, D, M)>D$ and $L = L(K,D)>D$ that satisfies the following.

Let $\gamma_{1}, \ldots, \gamma_{n}$ be $K$-contracting axes whose domains are  longer than $L$. Suppose that $(\gamma_{1}, \ldots, \gamma_{n})$ is $D$-aligned and $d(\gamma_{i}, \gamma_{i+1}) < M$ for each $i$. Then the concatenation $\gamma_{1} \cup \ldots \cup \gamma_{n}$ of $\gamma_{1}, \ldots, \gamma_{n}$ is an $E$-contracting axis.
\end{lem}

For the proofs of these lemmata, refer to Subsection 3.1 and 3.2 of \cite{choi2022random1}.

\subsection{Random walks} \label{subsection:RW}

Let $\mu$ be a probability measure on $G$. We denote by $\check{\mu}$ the \emph{reflected version of $\mu$}, which is defined by $\check{\mu}(g) := \mu(g^{-1})$. The \emph{random walk} generated by $\mu$ is the Markov chain on $G$ with the transition probability $p(g, h) := \mu(g^{-1} h)$.

Consider the \emph{step space} $(G^{\Z}, \mu^{\Z})$, the product space of $(G, \mu)$. Each element $(g_{n})_{n \in \Z} \in G^{\Z}$  is called a \emph{step path}, and there is a corresponding \emph{(bi-infinite) sample path} $(Z_{n})_{n \in \Z}$ under the correspondence \[
Z_{n} = \left\{ \begin{array}{cc} g_{1} \cdots g_{n} & n > 0 \\ id & n=0 \\ g_{0}^{-1} \cdots g_{n+1}^{-1} & n < 0. \end{array}\right.
\]
We also introduce the notation $\check{g}_{n} = g_{-n+1}^{-1}$ and $\check{Z}_{n} = Z_{-n}$. Note that we have an isomorphism $(G^{\Z}, \mu^{\Z}) \rightarrow  (G^{\Z_{>0}}, \check{\mu}^{\Z_{>0}})\times (G^{\Z_{>0}}, \mu^{\Z_{>0}}) $ by $(g_{n})_{n \in \Z} \mapsto ((\check{g}_{n})_{n > 0}, (g_{n})_{n >0})$. We will frequently use the latter parametrization: for each $(\w, \check{\w}) \in (G^{\Z_{>0}}, \check{\mu}^{\Z_{>0}}) \times (G^{\Z_{>0}}, \mu^{\Z_{>0}})$ we have \[\begin{aligned}
g_{n}(\check{\w}, \w) := g_{n}(\w), \quad Z_{n}(\check{\w}, \w) := Z_{n}(\w), \\
\check{g}_{n}(\check{\w}, \w) := g_{n}(\check{\w}), \quad\check{Z}_{n}(\check{\w}, \w) := Z_{n}(\check{\w})
\end{aligned}
\]

We define the \emph{support} of $\mu$, denoted by $\supp \mu$, as the set of elements in $G$ that are assigned nonzero values of $\mu$. We denote by $\mu^{N}$ the product measure of $N$ copies of $\mu$, and by $\mu^{\ast N}$ the $N$-th convolution measure of $\mu$.

Now suppose that $G$ is acting on the metric space $(X, d)$. A probability measure $\mu$ on $G$ is said to be \emph{non-elementary} if the semigroup generated by the support of $\mu$ contains two independent contracting isometries $g, h$ of $X$.

\subsection{Schottky set} \label{subsection:Schottky}

We now introduce the notion of Schottky set. Given a sequence $\alpha  = (a_{1}, \ldots, a_{n}) \in G^{n}$, we employ the following notations: \[\begin{aligned}
\Pi(\alpha) &:= a_{1} a_{2} \cdots a_{n}, \\
\Gamma(\alpha) &:= \big(o, a_{1} o, a_{1} a_{2} o, \ldots, \Pi(\alpha) o\big).
\end{aligned}
\]
The following definition is an adaptation of Gou{\"e}zel's notion of Schottky sets on hyperbolic groups in \cite{gouezel2022exponential}.

\begin{definition}[cf. {\cite[Definition 3.11]{gouezel2022exponential}}, {\cite[Definition 3.15]{choi2022random1}}]\label{dfn:Schottky}
Let $K_{0} > 0$ and define: \begin{itemize}
\item $D_{0}= D(K_{0}, K_{0})$ be as in Lemma \ref{lem:1segment}, 
\item $D_{1}= E(K_{0}, D_{0})$, $L = L(K_{0}, D_{0})$ be as in Lemma \ref{lem:BGIPWitness},
\item $E_{0}= E(K_{0}, D_{0})$, $L' = L(K_{0}, D_{0})$ be as in Lemma \ref{lem:BGIPWitness},
\item $L'' = L(K_{0}, D_{1})$ be as in Lemma \ref{lem:BGIPConcat}.
\end{itemize}
We say that a set of sequences $S \subseteq G^{n}$ is a \emph{fairly long $K_{0}$-Schottky set} if: \begin{enumerate}
\item $n > \max\{L, L', L''\}$;
\item $\Gamma(\alpha)$ is $K_{0}$-contracting axis for all $\alpha \in S$;
\item $d(o, \Pi(\alpha)o) \ge 10E_{0}$ for all $\alpha \in S$;
\item for each $x \in X$ we have \[
\#\Big\{\alpha \in S : \textrm{$\big(x, \Gamma(\alpha)\big)$ and $\big(\Gamma(\alpha), \Pi(s) x\big)$ are $K_{0}$-aligned}\Big\} \ge \# S - 1;
\]
\item for each $\alpha \in S$, $\big(\Gamma(\alpha), \Pi(\alpha)\Gamma(\alpha)\big)$ is $K_{0}$-aligned.
\end{enumerate}
When the Schottky parameter $K_{0}$ is understood, the constants $D_{0}, D_{1}, E_{0}$ \emph{always} denote the ones defined above. Further, we sometimes omit the Schottky parameter and just say that $S$ is a fairly long Schottky set.

Once a fairly long Schottky set $S$ is understood, its element $\alpha$ is called a \emph{Schottky sequence} and the translates of $\Gamma(\alpha)$ are called \emph{Schottky axes}. When a probability measure $\mu$ on $G$ is given in addition such that $S \subseteq (\supp \mu)^{n}$, we say that $S$ is a fairly long Schottky set for $\mu$.
\end{definition}

Some facts regarding Schottky sets are in order.

\begin{lem}[{\cite[Proposition 3.18]{choi2022random1}}]\label{lem:Schottky}
Let $\mu$ be a non-elementary probability measure on $G$. Then for each $N>0$ there exists a fairly long Schottky set for $\mu$ with cardinality $N$.
\end{lem}

\begin{definition}\label{dfn:semiAlign}
Let $S$ be a fairly long Schottky set and let $K>0$. We say that a sequence of Schottky axes is \emph{$K$-semi-aligned} if it is a subsequence of a $K$-aligned sequence of Schottky axes.

More precisely, for Schottky axes $\gamma_{1}, \ldots, \gamma_{n}$, we say that $(\gamma_{1}, \ldots, \gamma_{n})$ is \emph{$K$-semi-aligned} if there exist Schottky axes $\eta_{1}, \ldots, \eta_{m}$ such that $(\eta_{1}, \ldots, \eta_{m})$ is $K$-aligned and there exists a subsequence $\{i(1) < \ldots < i(n)\} \subseteq \{1, \ldots, m\}$ such that $\gamma_{l} = \eta_{i(l)}$ for $l=1, \ldots, n$.

Similarly, for points $x, y \in X$ and Schottky axes $\gamma_{1}, \ldots, \gamma_{n}$, we say that $(x, \gamma_{1}, \ldots, \gamma_{n}, y)$ is \emph{$K$-semi-aligned} if it is a subsequence of a $K$-aligned sequence $(x, \eta_{1}, \ldots, \eta_{m}, y)$ for some Schottky axes $\eta_{1}, \ldots, \eta_{n}$.
\end{definition}

Lemma \ref{lem:BGIPWitness} implies the following corollary.

\begin{cor}\label{cor:semiAlign}
Let $S$ be a fairly long $K_{0}$-Schottky set, with constants $D_{0}, D_{1}$ as in Definition \ref{dfn:Schottky}. Let $x, y \in X$ and $\gamma_{1}, \ldots, \gamma_{N}$ be Schottky axes. Then $(x, \gamma_{1}, \ldots, \gamma_{N}, y)$  is $D_{1}$-aligned whenever it is $D_{0}$-semi-aligned.
\end{cor}

\section{The first method: deviation inequalities}\label{section:firstMethod}

Our first approach to the limit laws rely on the pivoting technique implicitly via deviation inequalities. Throughout, we fix $(X, G, o)$ as in Convention \ref{conv:main}, fix a non-elementary probability measure $\mu$ on $G$ and fix a fairly long Schottky set $S \subseteq (\supp \mu)^{M_{0}}$ for $\mu$.

Employing the notations defined in Subsection \ref{subsection:RW}, let $(\check{\w}, \w) \in (G^{\Z_{>0}}, \check{\mu}^{\Z_{>0}}) \times (G^{\Z_{>0}}, \mu^{\Z_{>0}})$. We briefly recall the random variable $\Devi=\Devi(\check{\w}, \w)$ defined in \cite[Section 5]{choi2022random1}. For each $k\ge M_{0}$, we ask whether there exists $M_{0} \le i \le k$ such that: 
\begin{enumerate}
\item $\gamma := (Z_{i-M_{0}} o, Z_{i-M_{0}+1} o, \ldots, Z_{i} o)$ is a Schottky axis, and
\item $(\check{Z}_{m} o, \gamma, Z_{n}o )$ is $D_{0}$-semi-aligned for all $n \ge k$ and $m \ge 0$.
\end{enumerate}
Note that, by Corollary \ref{cor:semiAlign}, Item (2) forces that $(\check{Z}_{m} o, \gamma, Z_{n}o )$ is $D_{1}$-aligned for all $n \ge k$ and $m \ge 0$. We define $\Devi(\check{\w}, \w)$ as the minimal index $k$ that possesses such an auxiliary index $i \le k$.

Similarly, we defined $\check{\Devi} = \check{\Devi}(\check{\w}, \w)$ as the minimal index $k \ge M_{0}$ such that there exists $M_{0} \le i \le k$ for which: \begin{enumerate}
\item $\gamma :=  (\check{Z}_{i} o, \check{Z}_{i+1} o,  \ldots, \check{Z}_{i-M_{0}} o)$ is a Schottky axis, and
\item $(\check{Z}_{n} o, \gamma, Z_{m} o)$ is $D_{0}$-semi-aligned for all $n \ge k$ and $m \ge 0$.
\end{enumerate}

In \cite{choi2022random1} we proved the following results:

\begin{lem}[{\cite[Lemma 4.9]{choi2022random1}}]\label{lem:Devi}
Let $(X, G, o)$ be as in Convention \ref{conv:main} and let $\mu$ be a non-elementary probability measure on $G$. Then there exist $\kappa, K>0$ such that the following estimate holds for all $k$ and all choices of $g_{k+1}, \check{g}_{1}, \ldots, \check{g}_{k+1} \in G$:
\[\begin{aligned}
\Prob_{\check{\mu}^{\Z_{>0}} \times \mu^{\Z_{>0}}}\left(\Devi(\wVar, \w) \ge k \, \Big|\, g_{k+1}, \check{g}_{1}, \ldots, \check{g}_{k+1}\right) &\le K e^{-\kappa k},
\end{aligned}\]
and for all $k$ and for all choices of $g_{k+1}, \check{g}_{1}, \ldots, \check{g}_{k+1} \in G$:
\[\begin{aligned}
\Prob_{\check{\mu}^{\Z_{>0}} \times \mu^{\Z_{>0}}}\left(\check{\Devi}(\wVar, \w) \ge k \, \Big|\, \check{g}_{k+1}, g_{1}, \ldots, g_{k+1}\right) &\le K e^{-\kappa k}.
\end{aligned}\]
\end{lem}

By integrating for $g_{k+1}$ and $\check{g}_{1}, \ldots, \check{g}_{k+1}$ over $G$, we deduce that \[\begin{aligned}
&\Prob_{\check{\mu}^{\Z_{>0}} \times \mu^{\Z_{>0}}}\left(\Devi(\wVar, \w) \ge k \right) \\
&= \sum_{\substack{\check{g}_{1}, \ldots, \check{g}_{k+1}, \\ g_{k+1} \in G}} \Prob_{\check{\mu}^{\Z_{>0}} \times \mu^{\Z_{>0}}}\left(\Devi(\wVar, \w) \ge k \, \Big|\, g_{k+1}, \check{g}_{1}, \ldots, \check{g}_{k+1}\right) \cdot \mu(g_{k+1}) \check{\mu}(\check{g}_{1}) \cdots \check{\mu}(\check{g}_{k+1}) \le K e^{-\kappa k}.
\end{aligned}
\]
Similarly, we have $\Prob_{\check{\mu}^{\Z_{>0}} \times \mu^{\Z_{>0}}}\left(\check{\Devi}(\wVar, \w) \ge k \right) \le K e^{-\kappa k}$ for each $k$.

\begin{lem}[{\cite[Corollary 4.11]{choi2022random1}}]\label{lem:minDevi}
Let $p>0$ and suppose that $\mu$ has finite $p$-th moment. Then there exists $K>0$ such that \[
\E_{\check{\mu}^{\Z_{>0}} \times \mu^{\Z_{>0}}}\left[ \min \{ d(o, Z_{\Devi} o), d(o, \check{Z}_{\check{\Devi}} o)\}^{2p} \right] < K.
\]
\end{lem}

Let $F_{p}$ be the distribution of $\min\{d(o, Z_{\upsilon} o), d(o, \check{Z}_{\check{\upsilon}} o)\}^{2p}$, i.e., \[
F_{p}(u) = \Prob_{\check{\mu}^{\Z_{>0}} \times \mu^{\Z_{>0}}}\Big( \min \{ d(o, Z_{\Devi} o), d(o, \check{Z}_{\check{\Devi}} o)\}^{2p} \ge u\Big).
\]By Lemma \ref{lem:minDevi}, $\int_{0}^{\infty} F_{p}(u) du$ is finite whenever $\mu$ has finite $p$-th moment.

Using this, we can prove Theorem \ref{thm:discrepancy}.

\begin{proof}[Proof of Theorem \ref{thm:discrepancy}]
Recall that $\mu$ be a non-elementary probability measure on $G$. Let $\kappa, K >0$ be the constants for $\mu$ as in Lemma \ref{lem:Devi}.

Recall also our notations introduced in Subsection \ref{subsection:RW}: $(g_{i})_{i \in \Z}$'s are i.i.d. RVs distributed according to $\mu$ and $Z_{i} := g_{1} \cdots g_{i}$ for $i \ge 0$. Now for each $n>0$, we bring the $n$-step initial subpath $(g_{1}, \ldots, g_{n})$ of $(g_{i})_{i > 0}$ that  is distributed according to $\mu^{n}$, together with another RV $(h_{i})_{i \in \Z}$ distributed according to $\mu^{\Z}$ that is independent from $(g_{i})_{i \in \Z}$. We then define \[\begin{aligned}
&g_{i; 0} := \left\{ \begin{array}{cc} g_{i} & i = 1, \ldots, \lfloor n/2 \rfloor \\ h_{i} & i > \lfloor n/2  \rfloor, \end{array}\right. &\check{g}_{i; 0} := \left\{ \begin{array}{cc} g_{n - i + 1}^{-1} & i = 1, \ldots, n- \lfloor n/2 \rfloor \\ h_{-i}^{-1} & i >n-\lfloor n/2 \rfloor, \end{array}\right. \\
&g_{i; 1} := \left\{ \begin{array}{cc} g_{ \lfloor n/2 \rfloor +i} & i = 1, \ldots, n-\lfloor n/2 \rfloor \\ h_{i}  &i > n- \lfloor n/2  \rfloor, \end{array}\right.&\check{g}_{i; 1} := \left\{ \begin{array}{cc} g_{\lfloor n/2 \rfloor  - i + 1}^{-1} & i = 1, \ldots, \lfloor n/2 \rfloor \\ h_{-i}^{-1} & i > \lfloor n/2  \rfloor .\end{array}\right. 
\end{aligned}
\]

Then $ \big((\check{g}_{i; t})_{i > 0}, (g_{i; t})_{i>0} \big)$ is distributed according to $\check{\mu}^{\Z_{>0}} \times \mu^{\Z_{>0}}$ for $t=0, 1$. Using them, we similarly define other RVs such as  \[\begin{aligned}
Z_{i; t} &:= g_{1; t} \cdots g_{i; t}, &\check{Z}_{i; t} &:= \check{g}_{1; t} \cdots \check{g}_{i;t}, \\
\Devi_{t} &:= \Devi\big( (\check{g}_{i;t})_{i>0}, (g_{i; t})_{i>0}\big), &\check{\Devi}_{t} &:= \check{\Devi}\big(  (\check{g}_{i;t})_{i>0}, (g_{i; t})_{i>0}\big).
\end{aligned}
\]
Note that $\Devi_{t}$, $\check{\Devi}_{t}$ are RVs that depend on the choice of $n$. For this reason, we will also denote $\Devi_{0}$ by $v(n)$ and $\check{\Devi}_{0}$ by $v(n)$.

We now define \[
A_{n} := \Big\{ \max \{\Devi_{0}, \check{\Devi}_{0}, \Devi_{1}, \check{\Devi}_{1} \} \ge n/10\Big\}.
\]
Since $\big((\check{g}_{i; 0})_{i > 0}, (g_{i; 0})_{i>0} \big)$ and $\big((\check{g}_{i; 1})_{i > 0}, (g_{i; 1})_{i>0} \big)$ are both distributed according to $\check{\mu}^{\Z_{>0}} \times \mu^{\Z_{>0}}$, Lemma \ref{lem:Devi} implies \begin{equation}\label{eqn:exceptionalProbSmall}\begin{aligned}
\Prob(A_{n}) &\le2 \Prob_{\check{\mu}^{\Z_{>0}} \times \mu^{\Z_{>0}}} \big(\Devi(\check{\w}, \w) \ge n/10\big) +2 \Prob_{\check{\mu}^{\Z_{>0}} \times \mu^{\Z_{>0}}} \big(\check{\Devi}(\check{\w}, \w) \ge n/10\big)\\
&\le 4K e^{-\kappa n/10}.
\end{aligned}
\end{equation}
Furthermore, Lemma \ref{lem:minDevi} implies \[\begin{aligned}
\Prob\Big(\min \left\{d\left(o, Z_{\Devi_{t}; t} \,o\right), d \left(o, \check{Z}_{\check{\Devi}_{t}; t} \,o\right)\right\}^{2p} \ge u\Big) = F_{p}(u)
\end{aligned}
\]
for each $u \ge 0$ and $t=0, 1$. We aim to show:
\begin{claim}\label{claim:proofAClaim1} \begin{equation}\label{eqn:claimProof1}
[d(o, Z_{n} o) - \tau(Z_{n})] \le 2\min \left\{d\left(o, Z_{\Devi_{0}; 0} \,o\right), d \left(o, \check{Z}_{\check{\Devi}_{0}; 0} \,o\right)\right\}
\end{equation}
holds outside $A_{n}$. 
\end{claim}

For the proof of Claim \ref{claim:proofAClaim1}, we discuss everything outside $A_{n}$. That means, we assume from now on that \begin{equation}\label{eqn:assumeV}
\Devi_{0}, \check{\Devi}_{0}, \Devi_{1}, \check{\Devi}_{1} \le n/10.
\end{equation}
First, by the definition of $\Devi_{0}$, there exists an integer $i(0)$ such that 
\begin{enumerate}
\item $M_{0} \le i(0) \le \Devi_{0}$;
\item $\gamma_{0} := (Z_{i(0) -M_{0}; 0}\, o, Z_{i(0) - M_{0} + 1; 0} \,o, \ldots, Z_{i(0); 0} \,o)$ is a Schottky axis;
\item $\left(\check{Z}_{j; 0}\, o, \gamma_{0}, Z_{k; 0}\, o\right)$
is $D_{1}$-aligned for $j \ge 0$ and $k \ge \Devi_{0}$. 
\end{enumerate}
Recall our definition that: \begin{equation}\label{eqn:zConversion1}
\begin{aligned}
Z_{i; 0} &= g_{1; 0} \cdots g_{i; 0} = g_{1} \cdots g_{i} = Z_{i} & (i=0, 1, \ldots, \lfloor n/2 \rfloor), \\
\check{Z}_{i; 0}&= \check{g}_{1; 0} \cdots \check{g}_{i; 0} = g_{n}^{-1} \cdots g_{n-i+1}^{-1} = Z_{n}^{-1} Z_{n-i} &(i=0,1, \ldots, n- \lfloor n/2 \rfloor).
\end{aligned}
\end{equation}
Item (3) of the condition for $i(0)$, Display \ref{eqn:zConversion1} and Display \ref{eqn:assumeV} together imply: 
\begin{obs}\label{obs:proofAObs1}The sequence \[
\left( Z_{n}^{-1} Z_{n- j} o, \gamma_{0}, Z_{k} o\right) = 
\left( Z_{n}^{-1} Z_{n- j} o,\, (Z_{i(0) -M_{0}}o,  \ldots, Z_{i(0)}o),\, Z_{k} o\right)
\] is $D_{1}$-aligned for $0 \le j \le n-\lfloor n/2 \rfloor$ and $\Devi_{0} \le k \le \lfloor n/2 \rfloor$.
\end{obs}

Meanwhile, by the definition of $\check{\Devi}_{0}$, there exists an integer $j(0)$ such that 
\begin{enumerate}
\item $M_{0} \le j(0) \le \check{\Devi}_{0}$;
\item $(\check{Z}_{j(0); 0}\, o, \,\check{Z}_{j(0) -1; 0} \,o,\, \ldots, \,\check{Z}_{j(0)-M_{0}; 0} \,o)$ is a Schottky axis, and 
\item $\left(\check{Z}_{j;0} \,o, \, (\check{Z}_{j(0); 0} \, o,\, \ldots, \, \check{Z}_{j(0) - M_{0}} \, o),\, Z_{k;0} \, o \right)$
is $D_{1}$-aligned for $k\ge 0$ and $j \ge \Devi_{0}$. 
\end{enumerate}
We now define \[
\check{\gamma}_{0} := (Z_{n - j(0)} o, \, Z_{n - j(0)+1} o,\,\ldots,\, Z_{n-j(0) + M_{0}} o).
\]  Item (3) of the above, Display \ref{eqn:zConversion1} and Display \ref{eqn:assumeV} together imply: 
\begin{obs}\label{obs:proofAObs2} 
\[
\big(Z_{n}^{-1} Z_{n-j} o, \,Z_{n}^{-1} \check{\gamma}_{0} , \,Z_{k} o\big) = 
\big(Z_{n}^{-1} Z_{n-j} o, \,(Z_{n}^{-1} Z_{n - j(0)} o, \ldots,\, Z_{n}^{-1}Z_{n-j(0) + M_{0}} o) , \,Z_{k} o\big)
\] is $D_{1}$-aligned for $\check{\Devi}_{0} \le j \le n - \lfloor n/2 \rfloor$ and $0 \le k \le \lfloor n/2 \rfloor$.
\end{obs}

Similarly, we have indices $i(1), j(1)$ such that \begin{enumerate}
\item $M_{0} \le i(1) \le \Devi_{1}$ and $M_{0} \le j(1) \le \check{\Devi}_{1}$;
\item $(Z_{i(1) - M_{0}; 1} o, \ldots, Z_{i(1); 1} o)$ and $(\check{Z}_{j(1); 1} o, \ldots, \check{Z}_{j(1)-M_{0}; 1} o)$ are Schottky axes,
\item $\big(\check{Z}_{j; 1} o, (Z_{i(1) - M_{0}; 1} o, \ldots, Z_{i(1); 1} o), Z_{k; 1} o\big)$ is $D_{1}$-aligned for $j \ge 0$ and $k \ge \Devi_{1}$, and 
\item $\big(\check{Z}_{j; 1} o, (\check{Z}_{j(1); 1} o, \ldots, \check{Z}_{j(1)-M_{0}; 1} o), Z_{k; 1} o\big)$ is $D_{1}$-aligned for $k \ge 0$ and $j \ge \check{\Devi}_{1}$.
\end{enumerate}
This time, we define \[\begin{aligned}
\gamma_{1} := (Z_{\lfloor n/2 \rfloor + i(1) - M_{0}} o, \, \ldots, \, Z_{\lfloor n/2 \rfloor + i(1)} o),\quad \check{\gamma}_{1} := (Z_{\lfloor n/2 \rfloor - j(1)} o, \, \ldots, \, Z_{\lfloor n/2 \rfloor - j(1) + M_{0}} o).
\end{aligned}
\]
and recall the following:\begin{equation}\label{eqn:zConversion2}
\begin{aligned}
Z_{i; 1} &= g_{1; 1} \cdots g_{i; 1} = g_{\lfloor n/2 \rfloor + 1} \cdots g_{\lfloor n/2 \rfloor + i} = Z_{\lfloor n/2 \rfloor}^{-1} Z_{\lfloor n/2 \rfloor + i} & (i=0, 1, \ldots, n-\lfloor n/2 \rfloor), \\
\check{Z}_{i; 1}&= \check{g}_{1; 1} \cdots \check{g}_{i; 1} = g_{\lfloor n/2 \rfloor}^{-1} \cdots g_{\lfloor n /2 \rfloor-i+1}^{-1} = Z_{ \lfloor n/2 \rfloor}^{-1} Z_{\lfloor n/2 \rfloor-i} &(i=0,1, \ldots,  \lfloor n/2 \rfloor).
\end{aligned}
\end{equation}
Combining Item (3) and (4) of the conditions for $i(1)$ and $j(1)$, Display \ref{eqn:zConversion1} and Display \ref{eqn:assumeV}, we obtain:  \begin{obs}
\label{obs:proofAObs3} The sequence\[\begin{aligned}
\big(Z_{\lfloor n/2 \rfloor - j} o, \gamma_{1}, Z_{\lfloor n/2 \rfloor + k} o\big) &= \big(Z_{\lfloor n/2 \rfloor - j} o, (Z_{\lfloor n/2 \rfloor + i(1) - M_{0}} o, \, \ldots, \, Z_{\lfloor n/2 \rfloor + i(1)} o), Z_{\lfloor n/2 \rfloor + k} o\big)
\end{aligned}
\]
is $D_{0}$-aligned for $0 \le j \le \lfloor n/2 \rfloor$ and $\Devi_{1} \le k \le n - \lfloor n/2 \rfloor$. Moreover, \[\begin{aligned}
\big(Z_{\lfloor n/2 \rfloor - j} o, \check{\gamma}_{1}, Z_{\lfloor n/2 \rfloor + k} o\big) &= \big(Z_{\lfloor n/2 \rfloor - j} o, (Z_{\lfloor n/2 \rfloor - j(1)} o, \, \ldots, \, Z_{\lfloor n/2 \rfloor - j(1) + M_{0}} o), Z_{\lfloor n/2 \rfloor + k} o\big)
\end{aligned}
\]
is $D_{0}$-aligned for $\check{\Devi}_{1} \le j \le \lfloor n/2 \rfloor$ and $0 \le k \le n - \lfloor n/2 \rfloor$. 
\end{obs}

A special case of Observation \ref{obs:proofAObs1} is that $(o, \gamma_{0}, Z_{\Devi_{0}} o)$ is $D_{1}$-aligned. By Lemma \ref{lem:BGIPWitness} (cf. Definition \ref{dfn:Schottky}), there exists a subsegment $[p. q]$ of $[o, Z_{\Devi_{0}} o]$ that is $0.1E_{0}$-fellow traveling with $\gamma_{0}$. We deduce that \begin{equation}
\label{eqn:v0Apply}\begin{aligned}
d(o, Z_{i(0)- M_{0} } o) &\le d(o, p) - 0.1E_{0} = d(o, q) - d(p, q) - 0.1E_{0} \\
& \le d(o, Z_{\Devi_{0}} o) - d(Z_{i(0) - M_{0} + 1} o, Z_{i(0)} o) - 0.3E_{0} \\
&\le d(o, Z_{\Devi_{0}}) - 90E_{0}.
 \end{aligned} 
\end{equation}
Similarly, from the fact that $(Z_{n - \check{\Devi}_{0}} o, \check{\gamma}_{0}, Z_{n} o)$ is $D_{1}$-aligned we deduce \[d(o, Z_{n}^{-1} Z_{n - j(0) + M_{0}} o) \le d(Z_{n-\check{\Devi}_{0}} o, Z_{n} o) - 90E_{0} = d(\check{Z}_{\check{\Devi}_{0}} o, o ) - 90E_{0}.
\]

Next, Display \ref{eqn:assumeV} implies that \[\begin{aligned}
\gamma_{0} \subseteq \big\{Z_{k} o : 0 \le k \le \lfloor n/2 \rfloor - \check{\Devi}_{1}\big\}, \quad \check{\gamma}_{1} \subseteq \big\{Z_{k} o : \Devi_{0} \le k \le \lfloor n/2 \rfloor\big\},\quad\quad \\
\gamma_{1} \subseteq \big\{Z_{k}o : \lfloor n/2 \rfloor \le k \le n - \check{\Devi}_{0}\big\}, \quad \check{\gamma}_{0} \subseteq \big\{Z_{k} o : \lfloor n/2 \rfloor + \Devi_{1} \le k \le n\big\}.
\end{aligned}
\]
Combining this with Observation \ref{obs:proofAObs1}, \ref{obs:proofAObs2} and \ref{obs:proofAObs3}, we deduce that $(\gamma_{0}, \check{\gamma}_{1}, \gamma_{1}, \check{\gamma}_{0}, Z_{n}\gamma_{0})$ is $D_{1}$-aligned, and consequently, \[
\big(o, \gamma_{0}, \check{\gamma}_{1}, \gamma_{1}, \check{\gamma}_{0}, Z_{n}\gamma_{0}, Z_{n}\check{\gamma}_{1}, Z_{n}\gamma_{1}, Z_{n}\check{\gamma}_{0},  \ldots, Z_{n}^{k-1}\gamma_{0}, Z_{n}^{k-1}\check{\gamma}_{1}, Z_{n}^{k-1}\gamma_{1}, Z_{n}^{k-1}\check{\gamma}_{0}, Z_{n}^{k} o\big)
\]
is also $D_{1}$-aligned. By Lemma \ref{lem:BGIPWitness}, there exist points $p_{0}, q_{0}, \ldots, p_{k-1}, q_{k-1}$ on $[o, Z_{n}^{k} o]$, from left to right, so that \[
d(p_{i}, Z_{n}^{i} \cdot Z_{i(0)- M_{0} } o) < 0.1E_{0}, \quad d(q_{i}, Z_{n}^{i} \cdot Z_{n - j(0) +M_{0} } o) < 0.1E_{0}.
\] This implies that \[\begin{aligned}
d(o, Z_{n}^{k} o) &\ge \sum_{i=1}^{k-1} d(p_{i-1}, p_{i}) \\
&\ge \sum_{i=1}^{k-1}  \Big(\begin{array}{c}d(Z_{n}^{i-1}o, Z_{n}^{i}  o) - d(Z_{n}^{i-1} o, Z_{n}^{i-1} Z_{i(0)- M_{0}} o) - d(Z_{n}^{i} o, Z_{n}^{i} Z_{i(0)- M_{0}} o) \\
- d(p_{i-1}, Z_{n}^{i-1} Z_{i(0)- M_{0}} o) - d(p_{i}, Z_{n}^{i} Z_{i(0)- M_{0}} o)\end{array}\Big) \\
&\ge (k-1) \Big(d(o, Z_{n} o) - 2 d(o, Z_{i(0)- M_{0}} o) - E_{0}\Big).
\end{aligned}
\]By taking the limit and applying Inequality \ref{eqn:v0Apply}, we deduce that \[
\tau(Z_{n}) \ge d(o, Z_{n} o) - 2 d(o, Z_{i(0)- M_{0}} o) - 2E_{0} \ge d(o, Z_{n} o) - 2d(o, Z_{\Devi_{0}} o).
\]

By a similar argument using $q_{i}$'s and $Z_{n}^{i-1} Z_{n - j(0) + M_{0}}$'s, we also observe that $\tau(Z_{n}) \ge d(o, Z_{n} o) - 2 d( o, \check{Z}_{\check{\Devi}_{0}} o)$. Claim \ref{claim:proofAClaim1} is now established.

Given the claim, we obtain \[\begin{aligned}
&\Prob(d(o, Z_{n} o) - \tau(Z_{n}) \ge C n^{1/2p})\\
 &= \Prob\left(\left[d(o, Z_{n}o) - \tau(Z_{n}) \right]^{2p}\ge C^{2p} n\right)\\
&\le \Prob(A_{n}) + \Prob \left( 2^{2p} \min \left\{ d\left(o, Z_{\Devi_{0}; o} \, o \right), d\left( o, \check{Z}_{\check{\Devi}_{0}; 0} \, o \right) \right\}^{2p} \ge C^{2p} n \right) \\
&\le F_{p}(C^{2p} n/2^{2p}) + 8Ke^{-\kappa n/10}.
\end{aligned}
\]
When $\mu$ has finite $p$-th moment, $F_{p}(u)$ is integrable and the above probability is summable. The Borel-Cantelli lemma implies Item (1) of Theorem  \ref{thm:discrepancy}.

Now suppose that $\mu$ has finite first moment and let $\lambda$ be its escape rate. We  denote $\Devi_{0}$ by $v(n)$ to clarify its dependence on  $n$. This time, we define \[
A_{n}' := \{v(n) \ge K' \log n \}
\]
for some large $K'$ such that $\sum_{n} K e^{-\kappa K' \log n} < +\infty$. Recall that Lemma \ref{lem:Devi} tells us that \[
\begin{aligned}
\Prob(A_{n}\cup A_{n}') &\le 4Ke^{-\kappa k} + \Prob_{\check{\mu}^{\Z_{>0}} \times \mu^{\Z_{>0}}}(\Devi(\check{\w}, \w) \ge K' \log n) \\
&\le 4Ke^{-\kappa k}+ K e^{-\kappa K' \log n}.
\end{aligned}
\]
The Borel-Cantelli lemma implies that almost every sample path $(\check{\w}, \w)$ eventually lies outside $A_{n}\cup A_{n}'$, which implies that $d(o, Z_{n} o) - \tau(Z_{n}) \le d(o, Z_{v(n)} o)$ for all large enough $n$ (Claim \ref{claim:proofAClaim1}) and $\limsup v(n)/\log n \le K'$. Moreover, by subadditive ergodic theorem, we have $\lim_{n} d(o, Z_{n} o) / n = \lambda$ for almost every sample path.

It remains to show that $\limsup_{n} \frac{d(o, Z_{n} o) - \tau(Z_{n})}{\log n} \le 4\lambda K'$ whenever \[
\limsup_{n} \frac{d(o, Z_{n} o) - \tau(Z_{n})}{d(o, Z_{v(n)} o)} \le 1,\,\,\, \limsup_{n} \frac{v(n)}{\log n} \le K',\,\,\, \lim_{n} \frac{d(o, Z_{n} o)}{n} = \lambda.
\]
To show this, take $N$ large enough such that $d(o, Z_{n} o) / n \le 2\lambda$ for $n \ge 2K'\log N$ and $v(n)/ \log n \le 2K'$ for $n \ge N$. For $n \ge N$, we have \[
d(o, Z_{v(n)} o) \le \left\{ \begin{array}{cc} \max\{ d(o, Z_{i} o) : 0 \le i \le 2K' \log N\}  & \textrm{when}\,v(n) \le 2K' \log N, \\ 2 \lambda v(n) \le 4K' \lambda \log n & \textrm{otherwise}.\end{array}\right.
\]
It is clear that $d(o, Z_{v(n)} o) / \log n \le 4 \lambda K'$ for all large enough $n$. Consequently, we have \[
\limsup_{n}  \frac{d(o, Z_{n} o) - \tau(Z_{n})}{\log n} \le \limsup_{n}  \frac{d(o, Z_{n} o) - \tau(Z_{n})}{d(o, Z_{v(n)} o) } \cdot \frac{d(o, Z_{v(n)} o)}{\log n} \le 4K' \lambda.\qedhere
\]
\end{proof}

When $\mu$ has finite first moment, Theorem \ref{thm:discrepancy} implies that  \begin{equation}\label{eqn:sllnBefore}\begin{aligned}
0 &\le \lim_{n \rightarrow +\infty} \frac{1}{n} \left| d(o, Z_{n} o) - \tau(Z_{n}) \right| \le \lim_{n \rightarrow +\infty} \frac{1}{\sqrt{n \log \log n}} \left| d(o, Z_{n} o) - \tau(Z_{n}) \right| \\
&\le \lim_{n\rightarrow +\infty} \frac{1}{\sqrt{n}} \left| d(o, Z_{n} o) - \tau(Z_{n}) \right|= 0 \quad \textrm{almost surely}.
\end{aligned}
\end{equation}Combining Inequality \ref{eqn:sllnBefore} with the SLLN for displacement (cf. \cite[Theorem A]{choi2022random1}), we conclude:

\begin{cor}[SLLN for finite first moment]\label{cor:SLLNFinite}
Let $(X, G, o)$ be as in Convention \ref{conv:main}, and let $(Z_{n})_{n\ge 0}$ be the random walk generated by a non-elementary measure $\mu$ on $G$ with finite first moment. Then \begin{equation}\label{eqn:SLLN}
\lim_{n} \frac{1}{n} \tau(Z_{n}) = \lambda
\end{equation}
holds almost surely, where $\lambda = \lambda(\mu)$ is the escape rate of $\mu$.
\end{cor}

Also, combining Inequality \ref{eqn:sllnBefore} with the CLT and LIL for displacement (\cite[Theorem 4.13, 4.16]{choi2022random1}), we deduce:

\begin{cor}[CLT]\label{cor:CLT}
Let $(X, G, o)$ be as in Convention \ref{conv:main}, and $\w$ be the random walk generated by a non-elementary measure $\mu$ on $G$. If $\mu$ has finite second moment, then there exists $\sigma(\mu) \ge 0$ such that $\frac{1}{\sqrt{n}}(\tau(Z_{n}) - n \lambda)$ and $\frac{1}{\sqrt{n}}(d(o, Z_{n} o) - n\lambda)$ converge to the same Gaussian distribution $\mathscr{N}(0, \sigma(\mu)^{2})$ in law. We also have \[
\limsup_{n \rightarrow \infty}  \frac{\tau(Z_{n}) - \lambda n}{\sqrt{2n \log \log n}} =\sigma(\mu) \quad \textrm{almost surely}.
\]
\end{cor}

In fact, Theorem \ref{thm:discrepancy} implies Corollary \ref{cor:SLLNFinite} for measures with finite $(1/2)$-th moment, and the converse of CLT for measures with finite $(1/4)$-th moment. For general non-elementary measures, Theorem \ref{thm:discrepancy} is not sufficient for the SLLN and the converse of CLT. Towards a different approach, let us recall: \begin{theorem}[{\cite[Theorem 6.4]{choi2022random1}}] \label{thm:expBdDisp}
Let $(X, G, o)$ be as in Convention \ref{conv:main} let $(Z_{n})_{n \ge 0}$ be the random walk generated by a non-elementary measure $\mu$ on $G$. Let $\lambda(\mu)$ be the escape rate of $\mu$. Then for any $0 < L < \lambda(\mu)$, there exists $K>0$ such that for each $n$ we have  \[
\Prob[d(o, Z_{n} o) \le Ln] \le K e^{-n/K}.
\]
\end{theorem}
Using this theorem, let us prove Theorem \ref{thm:expBd}.

\begin{proof}[Proof of Theorem \ref{thm:expBd}]
Let $\mu$ be a non-elementary probability measure on $G$ and $\lambda(\mu) \in (0, +\infty]$ be its escape rate. Given $0 < L < \lambda(\mu)$, we fix $0<\epsilon < 1/10$ such that $L' = L/(1-2\epsilon)$ is still smaller than $\lambda$. 

In the proof of Theorem \ref{thm:discrepancy}, we defined $\Devi_{0} = v(n)$ and $\check{\Devi}_{0} = \check{v}(n)$ for each $n$. Given $\epsilon>0$, we now define \[
A_{n}'' :=\Big\{\max \{ \Devi(n), \check{\Devi}(n)\} \ge \epsilon n\Big\}.
\]
As explained in Inequality \ref{eqn:exceptionalProbSmall}, with $\epsilon n$ in place of $n/10$, $\Prob(A_{n}'')$ decays exponentially as $n$ increases. In the sequel, we discuss everything outside $A_{n}''$.

Observation \ref{obs:proofAObs1}, \ref{obs:proofAObs2} and \ref{obs:proofAObs3} guarantees 4 Schottky axes $\gamma_{0}, \check{\gamma}_{0}, \gamma_{1}, \check{\gamma}_{1}$, such that $\gamma_{0}$ contains a point $Z_{i(0)} o$ for some $i(0) \in \{M_{0}, M_{0}+1, \ldots, \upsilon(n)\} \subseteq \{0, 1, \ldots, \lfloor \epsilon n \rfloor\}$,  $\check{\gamma}_{0}$ contains a point $Z_{n - j(0)} o$ for some $j(0) \in \{M_{0}, M_{0}+1, \ldots, \check{\upsilon}(n)\} \subseteq \{0,1, \ldots, \lfloor \epsilon n \rfloor\}$, and such that \[
\big(o, \gamma_{0}, \check{\gamma}_{1}, \gamma_{1}, \check{\gamma}_{0}, Z_{n}\gamma_{0}, Z_{n}\check{\gamma}_{1}, Z_{n}\gamma_{1}, Z_{n}\check{\gamma}_{0},  \ldots, Z_{n}^{k-1}\gamma_{0}, Z_{n}^{k-1}\check{\gamma}_{1}, Z_{n}^{k-1}\gamma_{1}, Z_{n}^{k-1}\check{\gamma}_{0},  Z_{n}^{k} o\big)
\]
is $D_{1}$-aligned for each $k$.  By Lemma \ref{lem:BGIPWitness}, there exist points $p_{0}, q_{0}, \ldots, p_{k-1}, q_{k-1}$ on $[o, Z_{n}^{k} o]$, from left to right, so that \[
d(p_{i}, Z_{n}^{i} \cdot Z_{i(0)- M_{0} } o) < 0.1E_{0}, \quad d(q_{i}, Z_{n}^{i} \cdot Z_{n - j(0) +M_{0} } o) < 0.1E_{0}.
\] This implies that \[\begin{aligned}
d(o, Z_{n}^{k} o) &\ge \sum_{i=0}^{k-1} d(p_{i}, q_{i}) \\
&\ge \sum_{i=0}^{k-1} \left( \begin{array}{c}d(Z_{n}^{i} Z_{i(0) - M_{0}} o, Z_{n}^{i} Z_{n-j(0) + M_{0}}o) - d(p_{i}, Z_{n}^{i}  Z_{i(0)- M_{0} } o) \\
-d(q_{i}, Z_{n}^{i}  Z_{n - j(0) +M_{0} } o)\end{array}\right)\\
&\ge k\min \Big\{ d(Z_{i} o, Z_{n-j} o) : 0 \le i, j \le \epsilon n \Big\} - E_{0}.
\end{aligned}
\]
By dividing by $k$ and taking the limit, we conclude that \[
\tau(Z_{n}) \ge \min \Big\{ d(Z_{i} o, Z_{n-j} o) : 0 \le i, j \le \epsilon n \Big\}
\] outside $A_{n}''$. Now, thanks to Theorem \ref{thm:expBdDisp}, there exists $K>0$ such that \[
\Prob[d(o, Z_{m} o) \le L' m] \le K e^{-m/K}
 \]
 holds for all $m$. This implies \begin{equation}\label{eqn:tauBd2}\begin{aligned}
 \Prob \left(\tau(Z_{n}) \le Ln\right) \le & \Prob(A_{n}'') + \Prob[\min \left\{d(Z_{i} o, Z_{n-j} o) : 0 \le i, j \le \epsilon n\right\}  \le Ln]\\
 & \le \Prob(A_{n}'') + \sum_{0 \le i, j \le \epsilon n} \Prob[d(Z_{i} o, Z_{n-j} o) \le Ln] \\
 &\le \Prob(A_{n}'') + (\epsilon n)^{2} \cdot K e^{-(1-2\epsilon) n/K}
 \end{aligned}
 \end{equation}
 for large enough $n$, which decays exponentially. 
 
 Meanwhile, outside $A_{n}''$, \[
 (\ldots, \quad Z_{n}^{k-1} \gamma_{0}, Z_{n}^{k-1} \check{\gamma}_{1}, Z_{n}^{k-1} \gamma_{1}, Z_{n}^{k-1} \check{\gamma}_{0}, Z_{n}^{k} \gamma_{0}, Z_{n}^{k} \check{\gamma}_{1}, \ldots)
 \]
is a $D_{1}$-aligned sequence of Schottky axes. Since $d(\gamma, \gamma')$ is uniformly bounded for $\gamma, \gamma' \in \{\gamma_{0}, \check{\gamma}_{0}, \gamma_{1}, \check{\gamma}_{1}, Z_{n}\gamma_{0}\}$, Lemma \ref{lem:BGIPConcat} tells us that the concatenation of sequences $(Z_{n}^{k}\gamma_{0},  Z_{n}^{k} \check{\gamma}_{1}, Z_{n}^{k} \gamma_{1}, Z_{n}^{k} \check{\gamma}_{0})_{k \in \Z}$ is a contracting axis. Since this axis and the orbit $\{Z_{n}^{k} o\}_{k \in \Z}$ are fellow traveling, $Z_{n}$ is also contracting. Hence we have $\Prob(\textrm{$Z_{n}$ is not contracting}) \le \Prob(A_{n}'')$, which decays exponentially.
\end{proof}

The previous proof did not assume that the initial displacement $d(o, Z_{i} o)$ or the final displacement $d(Z_{n-j} o, Z_{n} o)$ of a random path is shorter than the middle one $d(Z_{i}, Z_{n-j} o)$; indeed, one cannot expect such a phenomenon for high probability if the random walk has no moment condition. Instead, the proof explicitly used the fact that the middle segment will catch up the escape rate regardless of the moment condition, which is proven using the pivoting technique.

In order to discuss the converse of CLT for general measures, one should perform the pivoting more explicitly. For this purpose, we will recall the basics of the pivotal time construction in \cite{choi2022random1}.

\section{The second method: pivoting technique}\label{section:secondMethod}

\subsection{Pivotal times and pivoting}\label{subsection:pivotDiscrete}

This subsection is a summary of results in \cite[Subsection 5.1]{choi2022random1}, which is an adaptation of Gou{\"e}zel's work in \cite[Subsection 5.A]{gouezel2022exponential}; for complete proofs, refer to the explanation there. 

We keep employing Convention \ref{conv:main} for the metric space $X$ with basepoint $o$ and the isometry group $G$. We fix a fairly long $K_{0}$-Schottky set $S \subseteq G^{M_{0}}$ with cardinality $N_{0}$ for some $K_{0}$, $M_{0}$ and $N_{0} > 400$. Recall that we have associated an isometry and a contracting axis for each $\alpha \in S$: \[
\Pi(\alpha) := a_{1} a_{2}\cdots a_{n}, \quad \Gamma(\alpha) :=\big(o, \,a_{1} o,\, a_{1} a_{2} o, \,\ldots, \,\Pi(\alpha) o\big).
\]

We first fix sequences of isometries $(w_{i})_{i=0}^{\infty}$, $(v_{i})_{i=1}^{\infty}$ in $G$. Then we draw a sequence of Schottky sequences \[\begin{aligned}
\mathbf{s} &= (\alpha_{1}, \beta_{1}, \gamma_{1}, \delta_{1}, \ldots, \alpha_{n}, \beta_{n}, \gamma_{n}, \delta_{n}) \in S^{4n},
\end{aligned}
\] with respect to the uniform measure on $S^{4n}$ and construct words \[
W_{k} := w_{0} \Pi(\alpha_{1}) \Pi(\beta_{1}) v_{1} \Pi(\gamma_{1}) \Pi(\delta_{1}) w_{1} \cdots \Pi(\alpha_{k}) \Pi(\beta_{k}) v_{k} \Pi(\gamma_{k}) \Pi(\delta_{k})w_{k}. 
\]

Following \cite{gouezel2022exponential}, we constructed a set  $P_{n} = P_{n}(\mathbf{s}; (w_{i})_{i=0}^{n}, (v_{i})_{i=1}^{n}) \subseteq \{1, \ldots, n\}$, called the \emph{set of pivotal times}, that satisfy the properties described in Lemma \ref{lem:extremal}, \ref{lem:pivotEquiv} and \ref{lem:pivotDistbn}. First, pivotal times are (partial) recording of those Schottky axes that are aligned along the eventual progress.

\begin{lem}[{\cite[Lemma 5.3]{choi2022random1}}]\label{lem:extremal}
Let $\mathbf{s} \in S^{4n}$, $(w_{i})_{i=0}^{n} \in G^{n+1}$ and $(v_{i})_{i=1}^{n} \in G^{n}$, Let us enumerate the elements of $P_{n}(\mathbf{s}; (w_{i})_{i=0}^{n}, (v_{i})_{i=1}^{n})$ as $\{i(1) < \ldots < i(m)\}$ and define \[\begin{aligned}
\gamma_{4k-3} &:= W_{i(k)-1}\cdot  \Gamma(\alpha_{i(k)}),\\ \gamma_{4k-2} &:= W_{i(k)-1} \Pi(\alpha_{i(k)}) \cdot \Gamma(\beta_{i(k)}), \\ \gamma_{4k-1} &:= W_{i(k)-1} \Pi(\alpha_{i(k)}) \Pi(\beta_{i(k)}) v_{i(k)} \cdot \Gamma(\gamma_{i(k)}),\\
\gamma_{4k} &:= W_{i(k)-1} \Pi(\alpha_{i(k)}) \Pi(\beta_{i(k)}) v_{i(k)} \Pi(\gamma_{i(k)})\cdot \Gamma(\delta_{i(k)}).
\end{aligned}
\]
for $k=1, \ldots, m$. Then $\left(o, \gamma_{1}, \gamma_{2}, \ldots, \gamma_{4m}, W_{m}o\right)$ is $D_{0}$-semi-aligned.
\end{lem}

Another important feature of the pivotal times is that we are allowed to change our choices of $\beta_{i}$'s, $\gamma_{i}$'s and $v_{i}$'s at the pivotal times. To formulate, we define a subset $\tilde{S}$ of $S^{2} \times G$:

\begin{definition}[{\cite[Lemma 5.5]{choi2022random1}}]\label{dfn:pivotingChoices}
We define \[
\tilde{S} := \Big\{(\beta, \gamma, v) \in S^{2} \times G : \textrm{$\big(\Gamma(\beta), \Pi(\beta)v \Pi(\gamma) o\big)$ and $\big(v^{-1} o, \Gamma(\gamma)\big)$ are $K_{0}$-aligned sequences}\Big\}.
\]
Furthermore, we define for each $v \in G$ its section: \[
\tilde{S}(v) := \Big\{\big(\beta, \gamma\big) \in S^{2}: \textrm{$\big(\Gamma(\beta), \Pi(\beta)v \Pi(\gamma) o\big)$ and $\big(v^{-1} o, \Gamma(\gamma)\big)$ are $K_{0}$-aligned sequences}\Big\}.
\]
\end{definition}

Motivated by \cite[Lemma 4.6]{gouezel2022exponential}, we observed in \cite{choi2022random1} that $\# \tilde{S}(v) \ge (\#S)^{2} - 2\#S$ for each $v \in G$. The role of $\tilde{S}$ is captured by the following.

\begin{lem}[{\cite[Lemma 5.5]{choi2022random1}}] \label{lem:pivotEquiv}
Let $\mathbf{s} =  (\alpha_{1}, \beta_{1}, \gamma_{1}, \delta_{1}, \ldots, \alpha_{n}, \beta_{n}, \gamma_{n}, \delta_{n})$ be a choice drawn from $S^{4n}$ and $\mathbf{w} = (w_{i})_{i=0}^{n}, \mathbf{v} = (v_{i})_{i=1}^{n}$ be auxiliary sequences in $G$.

Let $i \in P_{n}(\mathbf{s}; \mathbf{w}, \mathbf{v})$ and let ($\bar{\mathbf{s}}$; $\mathbf{w}, \bar{\mathbf{v}})$ be obtained from $(\mathbf{s}; \mathbf{w}, \mathbf{v})$ by replacing $(\beta_{i}, \gamma_{i}, v_{i})$ with some $(\bar{\beta}_{i}, \bar{\gamma}_{i}, \bar{v}_{i})$ such that $(\bar{\beta}_{i}, \bar{\gamma}_{i}, \bar{v}_{i}) \in \tilde{S}$. 

Then  $P_{l}(\mathbf{s}; \mathbf{w}, \mathbf{v}) = P_{l}(\bar{\mathbf{s}}; \mathbf{w}, \bar{\mathbf{v}})$ for any $1 \le l \le n$. 
\end{lem}

Given $\mathbf{s}, \bar{\mathbf{s}}\in S^{4n}$ and $G$-valued sequences $\mathbf{w}, \mathbf{v}$ and $\bar{\mathbf{v}}$, we say that $(\bar{\mathbf{s}}; \mathbf{w}, \bar{\mathbf{v}})$ is \emph{pivoted from} $(\mathbf{s}; \mathbf{w}, \mathbf{v})$ if: \begin{itemize}
\item $\alpha_{i} = \bar{\alpha}_{i}$, $\delta_{i} = \bar{\delta}_{i}$ for all $i \in \{1, \ldots, n\}$;
\item $(\bar{\beta}_{i}, \bar{\gamma}_{i}, \bar{v}_{i}) \in \tilde{S}$ for each $i \in P_{n}(\mathbf{s}; \mathbf{w}, \mathbf{v})$, and 
\item $(\beta_{i}, \gamma_{i}, v_{i})= (\bar{\beta}_{i}, \bar{\gamma}_{i}, \bar{v}_{i})$ for each $i \in \{1, \ldots, n\} \setminus  P_{n}(\mathbf{s}; \mathbf{w}, \mathbf{v})$. 
\end{itemize}
By Lemma \ref{lem:pivotEquiv}, being pivoted from each other is an equivalence relation. Using this fact, we proved in \cite{choi2022random1} that the set of pivotal times grows linearly to the step number in probability.

\begin{lem}[{\cite[Corollary 5.8]{choi2022random1}}]\label{lem:pivotDistbn}
Fix $\mathbf{w} \in G^{n+1}$ and $\mathbf{v} \in G^{n}$. When $\mathbf{s} = (\alpha_{i}, \beta_{i}, \gamma_{i}, \delta_{i})_{i=1}^{n}$ is chosen from $S^{4n}$ with the uniform measure, $\# P_{n}(\mathbf{s})$ is greater in distribution than the sum of $n$ i.i.d. $X_{i}$, whose distribution is\begin{equation}\label{eqn:expRV}
\Prob(X_{i}=j) = \left\{\begin{array}{cc} (N_{0} - 4)/N_{0} & \textrm{if}\,\, j=1,\\ (N_{0} - 4)4^{-j}/N_{0}^{-j+1}& \textrm{if}\,\, j < 0, \\ 0 & \textrm{otherwise.}\end{array}\right.
\end{equation}
\end{lem}

\begin{cor}\label{cor:pivotDistbn} In the setting of Lemma \ref{lem:pivotDistbn}, for each $n$ we have \[
\Prob(\# P_{n}(\mathbf{s}) \le n/2) \le \big(3 \sqrt[4]{4/N_{0}}\big)^{n}.
\]
\end{cor}

\begin{proof}
This is a consequence of Chernoff-Hoeffding type inequality (\cite[Theorem 1]{hoeffding}). We provide here a version of proof using Chebyshev's inequality for convenience.

Let $X_{i}$'s be the RVs in the conclusion of Lemma \ref{lem:pivotDistbn}. By Equation \ref{eqn:expRV}, we have \[\begin{aligned}
\E \left[ \sqrt{4/N_{0}}^{X_{i}} \right] &= \left(1-\frac{4}{N_{0}}\right) \cdot \left[\sqrt{\frac{4}{N_{0}} }+ \sum_{j=1}^{\infty} \sqrt{\frac{N_{0}}{4}}^{j} \cdot \left( \frac{4}{N_{0}} \right)^{j} \right] \\
&= \left(1-\frac{4}{N_{0}}\right) \sqrt{\frac{4}{N_{0}}} \left( 1 + \frac{1}{1-\sqrt{4/N_{0}}} \right) \\
&= 2 \sqrt{4/N_{0}} +  \sqrt{4/N_{0}}^{2} -  \sqrt{4/N_{0}}^{3} \le 3 \sqrt{4/N_{0}}.
\end{aligned}
\]
Here, the last inequality is due to the fact $\sqrt{4/N_{0}} \le 1$. Now using the conclusion of Lemma \ref{lem:pivotDistbn} and the independence of $X_{i}$'s, we deduce \[
\E \left[ \sqrt{4/N_{0}}^{\#P_{n}(\mathbf{s})} \right] \le \E \left[ \sqrt{4/N_{0}}^{\sum_{i=1}^{n} X_{i}}\right] = \prod_{i=1}^{n} \E \left[ \sqrt{4/N_{0}}^{X_{i}} \right] \le \big(3 \sqrt{4/N_{0}}\big)^{n}.
\]
By Chebyshev's inequality, we also have \[
\E \left[ \sqrt{4/N_{0}}^{\#P_{n}(\mathbf{s})} \right] \ge\Prob(\# P_{n}(\mathbf{s}) \le n/2) \cdot \sqrt{4/N_{0}}^{n/2}.
\]
Combining these two inequalities leads to the desired conclusion.
\end{proof}

\subsection{Reduction for random walks}\label{subsection:pivotReduce}

In this subsection, we modify random walks on $G$ into suitable combinatorial model and then apply pivoting technique. As a result, we obtain a random path which is an alternation of fixed subsegments and i.i.d. random subsegments, all aligned on $X$.

\begin{lem}\label{lem:pivotFirstReduce}
Let $\mu$ be a non-elementary probability measure on $G$ and let $S \subseteq G^{M_{0}}$ be a fairly long $K_{0}$-Schottky set for $\mu$. Then for each $n$ there exist an integer $m(n)$, a probability space $\Omega_{n}$, a measurable subset $B_{n} \subseteq \Omega_{n}$, a measurable partition $\mathcal{Q}_{n}$ of $B_{n}$, and random variables \[\begin{aligned}
Z & \in G,\\
\{w_{i}, i=0, \ldots, m(n)\} &\in G^{m(n)+1}, \\
\{v_{i} : i=1, \ldots, m(n)\} &\in G^{m(n)},\\
 \{\alpha_{i}, \beta_{i}, \gamma_{i}, \delta_{i} : i=1, \ldots, m(n)\} &\in S^{4m(n)}
\end{aligned}
\]
such that the following hold: \begin{enumerate}
\item $\lim_{n \rightarrow +\infty} \Prob(B_{n}) = 1$ and $\lim_{n \rightarrow +\infty} m(n)/n>0$.
\item On each equivalence class $\mathcal{F} \in \mathcal{Q}_{n}$, $(w_{i})_{i=0}^{m(n)}$ are constant and $(\alpha_{i}, \beta_{i}, \gamma_{i}, \delta_{i}, v_{i})_{i=1}^{m(n)}$ are i.i.d.s distributed according to $(\textrm{uniform measure on $S^{4}$}) \times \mu$.
\item $Z$ is distributed according to $\mu^{\ast n}$ on $\Omega_{n}$ and \[
Z = w_{0} \Pi(\alpha_{1}) \Pi(\beta_{1}) v_{1} \Pi(\gamma_{1}) \Pi(\delta_{1}) w_{1} \cdots \Pi(\alpha_{m(n)}) \Pi(\beta_{m(n)}) v_{m(n)} \Pi(\gamma_{m(n)}) \Pi(\delta_{m(n)}) w_{m(n)}
\]
holds on $A_{n}$.
\end{enumerate}
\end{lem}

\begin{proof}
Let us denote the uniform measure on $S$ by $\mu_{S}$. Since $S$ is a fairly long Schottky set for $\mu$, $\mu_{S}$ is absolutely continuous with respect to $\mu$. Consequently, there exists $0<p<1$ that admits a decomposition \[
\mu^{4M_{0} + 1} = p \left( \mu_{S}^{2} \times \mu \times \mu_{S}^{2} \right) + (1-p) \nu
\]
for some (nonnegative) probability measure $\nu$. Let $n' = \lfloor n/(4M_{0}+1) \rfloor$.

We consider Bernoulli RVs $(\rho_{i})_{i=0}^{\infty}$ with average $p$, $(\eta_{i})_{i=0}^{\infty}$ with the law $\mu_{S}^{2}\times \mu \times \mu_{S}^{2}$ and $(\nu_{i})_{i=0}^{\infty}$ with the law $\nu$, all independent, and define \[
(g_{(4M_{0}+1)k+1}, \,\ldots, \,g_{(4M_{0}+1)(k+1)}) = \left\{\begin{array}{cc} \nu_{k} & \textrm{when}\,\, \rho_{k} = 0, \\ \eta_{k} & \textrm{when} \,\, \rho_{k} = 1. \end{array}\right. \quad (k=0, \ldots, n'-1).
\]
Let $g_{(4M_{0} + 1) n' + 1}, \ldots, g_{n}$ be i.i.d.s distributed according to $\mu$ that are also independent from $(\rho_{i}, \nu_{i}, \eta_{i})_{i=0}^{n'-1}$. Then $(g_{1}, \ldots, g_{n})$ is distributed according to $\mu^{n}$. We denote by $\Omega_{n}$ the ambient probability space on which $(\rho_{i}, \nu_{i}, \eta_{i})_{i}, (g_{i})_{i}$ are all measurable. We  set $Z = g_{1} \cdots g_{n}$, whose distribution on $\Omega_{n}$ is $\mu^{\ast n}$.

Recall that $\rho_{i}$'s are i.i.d. Bernoulli RVs with average $p$. Chernoff-Hoeffding's inequality implies that $\Prob(\sum_{i=0}^{n'-1} \rho_{i} \le \epsilon n)$ decays exponentially  for each $0 < \epsilon < \frac{p}{4M_{0} + 1}$. Considering this, we define  $m(n) := \lfloor p n / 8M_{0} \rfloor$ and \[
B_{n} := \left\{\w \in \Omega_{n} : \sum_{i=1}^{n'-1} \rho_{i} \ge m(n) \right\}.
\] The first item of the conclusion follows.

For $\w \in B_{n}$, we collect those indices $i$ at which $\rho_{i} = 1$. Then we denote the $m(n)$ smallest ones among them, in the increasing order, by $\vartheta(1), \ldots, \vartheta(m(n))$. In other words, $\{i : \rho_{i}(\w) = 1\} = \{\vartheta(1) < \vartheta(2) < \ldots < \vartheta(m(n)) < \ldots\}$. We now define $\mathcal{Q}_{n}$ of $B_{n}$ to be the partition determined by the values of $\{\rho_{i}, \nu_{i} : i \ge 0\}$, and define \[\begin{aligned}
w_{i-1} &:= g_{(4M_{0}+1)[\stopping(i-1) + 1] + 1} \cdots g_{(4M_{0} +1) \stopping(i)} \quad \left(i=1, \ldots,  m(n)\right),\\
w_{m(n)} &:= g_{(4M_{0}+1)(\stopping (m(n)) + 1) + 1} \cdots g_{n},\\
\alpha_{i} &:= (g_{(4M_{0} +1)\stopping(i) + 1}, \,\ldots, \,g_{(4M_{0} +1) \stopping(i) +M_{0}} ),\\
\beta_{i} &:= (g_{(4M_{0} +1) \stopping(i) + M_{0}+1}, \,\ldots, \,g_{(4M_{0} +1) \stopping(i) +2M_{0}} ),\\
v_{i} &:= g_{(4M_{0} + 1) \stopping (i) + 2M_{0} + 1},\\
\gamma_{i} &:= (g_{(4M_{0} +1) \stopping(i) + 2M_{0}+2}, \,\ldots, \,g_{(4M_{0} +1) \stopping(i) +3M_{0}+1} ),\\
\delta_{i} &:= (g_{(4M_{0} +1)\stopping(i) + 3M_{0}+2}, \,\ldots, \,g_{(4M_{0} +1) \stopping(i) +4M_{0}+1} ).
\end{aligned}
\]

On each equivalence class $\mathcal{F}$ of $\mathcal{Q}_{n}$, $w_{i}$'s are constants (since these are determined by the values of $\{\rho_{i}, \nu_{i} : i \ge 0\}$) and $(\alpha_{i}, \beta_{i}, v_{i}, \gamma_{i}, \delta_{i})$ are i.i.d.s with distribution $\mu_{S}^{2} \times \mu \times \mu_{S}^{2}$. Moreover, on $A_{n}$ we have \[\begin{aligned}
g_{1} \cdots g_{n} &= \prod_{i=1}^{m(n)} (g_{(4M_{0} + 1) [\stopping (i-1) + 1] + 1} \cdots g_{(4M_{0} + 1) \stopping (i)}) \cdot (g_{(4M_{0} + 1) \stopping (i) + 1} \cdots g_{(4M_{0} + 1) (\stopping(i) + 1)} ) \\
&\cdot g_{(4M_{0} + 1) (\stopping (m(n)) + 1 ) + 1} \cdots g_{n} \\
&= 
w_{0} \Pi(\alpha_{1}) \Pi(\beta_{1}) v_{1} \Pi(\gamma_{1}) \Pi(\delta_{1}) w_{1} \cdots \Pi(\alpha_{m(n)}) \Pi(\beta_{m(n)}) v_{m(n)} \Pi(\gamma_{m(n)}) \Pi(\delta_{m(n)}) w_{m(n)}.
\end{aligned}
\]
This is the third item of the conclusion and the proof is completed.
\end{proof}

Before the next reduction, we need the following definition.

\begin{definition}\label{dfn:wForPivot}
Let $n>0$ and $K>0$. We say that a sequence $(w_{i})_{i=0}^{n}$ in $G$ is \emph{$K$-pre-aligned} if, for the isometries $W_{0} := w_{0}$ and \[
V_{k} := W_{k} \Gamma(\beta_{k+1}) v_{k+1}, \,\, W_{k+1} := V_{k} \Gamma(\gamma_{k+1}) w_{k+1} \quad (k=0, \ldots, n-1),
\]
the sequence \[
\Big(o, \,W_{0} \Gamma(\beta_{1}),\, V_{0} \Gamma(\gamma_{1}),\, \ldots, \,W_{n-1}\Gamma(\beta_{n}), \,V_{n-1} \Gamma(\gamma_{n}), \,W_{n} o\Big)
\]
is $K$-semi-aligned for any choices of $(\beta_{i}, \gamma_{i}, v_{i}) \in \tilde{S}$ ($i=1, \ldots, n$).

We say that an isometry $\phi \in G$ is \emph{$K$-pre-aligned} if \[
\Big(\Gamma(\gamma'),\,\, \Pi(\gamma') \phi \cdot \Gamma(\beta) \Big)
\]
is $K$-semi-aligned for any choices of $(\beta, \gamma, v), (\beta', \gamma', v') \in \tilde{S}$.
\end{definition}

We now describe the alignment of Schottky axes at the pivotal times.

\begin{lem}\label{lem:pivotSecondReduce}
For each integer $n$, the following holds for $m(n) := 2 \lfloor 0.25 n \rfloor$.

Let $\mu$ be a probability measure on $G$, let $S \subseteq G^{M_{0}}$ be a fairly long $K_{0}$-Schottky set with cardinality $N_{0} \ge 400$. Fix a sequence $(w_{i})_{i =0}^{n}$ in $G$, and let $(\alpha_{i}, \beta_{i}, \gamma_{i}, \delta_{i}, v_{i})_{i =1}^{n}$ be i.i.d.s on $(S^{4} \times G)^{n}$, distributed according to $(\textrm{uniform measure on $S$})^{4} \times \mu$. Then there exist a measurable subset $B_{n}' \subseteq (S^{4} \times G)^{\Z_{>0}}$, a measurable partition $\mathcal{Q}_{n}'$ of $B_{n}'$ and RVs \[\begin{aligned}
\{w_{i}', i=0, \ldots, m(n)\} &\in G^{m(n)+1}, \\
\{v_{i}' : i=1, \ldots, m(n)\} &\in G^{m(n)},\\
 \{\beta_{i}', \gamma_{i}' : i=1, \ldots, m(n)\} &\in S^{2m(n)}
\end{aligned}
\]
such that the following hold: \begin{enumerate}
\item $\Prob(B_{n}') \ge 1- \big(3 \cdot \sqrt[4]{4/N_{0}}\big)^{n}$.
\item On each $\mathcal{F}' \in \mathcal{Q}_{n}'$, $(w_{i}')_{i=0}^{m(n)}$ is a fixed $D_{0}$-pre-aligned sequence in $G$ and $(\beta_{i}', \gamma_{i}', v_{i}')_{i=1}^{m(n)}$ are i.i.d.s distributed according to $(\textrm{uniform measure on $S^{2}$}) \times \mu$ conditioned on $\tilde{S}$.
\item On $B_{n}'$ we have the equality \[\begin{aligned}
& w_{0}' \Pi(\beta_{1}') v_{1}'\Pi(\gamma_{1}') w_{1}' \cdots \Pi(\beta_{m(n)}') v_{m(n)}'\Pi(\gamma_{m(n)}') w_{m(n)}'\\
&=w_{0} \Pi(\alpha_{1}) \Pi(\beta_{1}) v_{1}\Pi(\gamma_{1}) \Pi(\delta_{1}) w_{1} \cdots \Pi(\alpha_{n}) \Pi(\beta_{n}) v_{n}\Pi(\gamma_{n}) \Pi(\delta_{n}) w_{n}.
\end{aligned}
\]
\end{enumerate}
\end{lem}

\begin{proof}
Following the discussion in Subsection \ref{subsection:pivotDiscrete}, we define \begin{equation}\label{eqn:pivotalTimeModel1}\begin{aligned}
P_{n}(\w) &= P_{n} \big( (\alpha_{i}, \beta_{i}, \gamma_{i}, \delta_{i})_{i=1}^{n}; (w_{i})_{i=0}^{n}, (v_{i})_{i=1}^{n}\big).
\end{aligned}
\end{equation}
We take $m(n) =2\cdot \lfloor 0.25n \rfloor$ and let $B_{n}' := \{ \w\in \Omega : \# P_{n}(\w) \ge n/2\}$. By Corollary \ref{cor:pivotDistbn}, we have \[
\Prob(B_{n}') \ge 1 - \left(3 \cdot \sqrt[4]{4/N_{0}}\right)^{n} \ge 1 - (3 /\sqrt{10})^{n}.
\]

We now define \[
W_{k} := w_{0} \Pi(\alpha_{1}) \Pi(\beta_{1}) v_{1} \Pi(\gamma_{1}) \Pi(\delta_{1}) w_{1} \cdots\Pi(\alpha_{k}) \Pi(\beta_{k}) v_{k} \Pi(\gamma_{k}) \Pi(\delta_{k})w_{k} \quad (k=0, \ldots, n).
\]
Given $\w \in B_{n}'$ with $P_{n}(\w) = \{i(1) < i(2) < \ldots < i(m(n)) < \ldots\}$, we define \[
\begin{aligned}
w_{0}' &:= W_{i(1) - 1} \Pi(\alpha_{i(1)}), \quad w_{m(n)}' := \Pi(\delta_{i(m(n))}) w_{i(m(n))} \cdot W_{i(m(n))}^{-1} \cdot W_{n}, \\
w_{k}' &:=  \Pi(\delta_{i(k)}) w_{i(k)} \cdot W_{i(k)}^{-1} \cdot W_{i(k+1) - 1} \Pi(\alpha_{i(1)}) \quad (k=1, \ldots, m(n)-1).
\end{aligned}
\]
Furthermore, we record the Schottky axes at the pivotal times: we define $\beta_{k}' := \beta_{i(k)}$, $v_{k}' := v_{i(k)}$ and $\gamma_{k}' := \gamma_{i(k)}$ for $k=1, \ldots, m(n)$. We then observe the following relation: \[
W_{i(k+1)-1} \Pi(\alpha_{i(k+1)}) = W_{i(k)-1} \Pi(\alpha_{i(k)})\cdot \Pi(\beta_{i(k)}) v_{i(k)}\cdot \Pi(\gamma_{i(k)}) w_{k}'.
\]
By induction, we observe that \begin{equation}\label{eqn:reparamWk}\begin{aligned}
w_{0}' \Pi(\beta_{1}') v_{1}' \Pi(\gamma_{1}') w_{1}' \cdots \Pi(\beta_{k}') v_{k}' \Pi(\gamma_{k}') w_{k}' = W_{i(k+1) -  1} \Pi(\alpha_{i(k+1)}) \quad (k=0, \ldots, m(n)-1), \\
w_{0}' \Pi(\beta_{1}') v_{1}' \Pi(\gamma_{1}') w_{1}' \cdots \Pi(\beta_{m(n)}') v_{m(n)}' \Pi(\gamma_{m(n)}') w_{m(n)}' = W_{n} \quad \quad \quad\quad\quad\quad\,\,
\end{aligned}
\end{equation}on $B_{n}'$, settling Item 3 of the conclusion. 

We define the measurable partition $\mathcal{Q}_{n}'$ of $B_{n}'$ by pivoting at the first $m(n)$ pivotal times. More explicitly, two elements $\w, \w' \in B_{n}'$ belong to the same equivalence class if and only if:
\[\begin{aligned}
\alpha_{i}'(\w) = \alpha_{i}'(\w'), \,\,\delta_{i}' (\w)&= \delta_{i}'(\w'), w_{i}'(\w) = w_{i}'(\w') & (i \in \{0, \ldots, n\}),\\
P_{n}(\w) = P_{n} (\w') &=: \{i(1) < i(2) < \ldots\}, &\\
(\beta_{i}'(\w), \gamma_{i}'(\w), v_{i}'(\w)) &\in \tilde{S} & \left(i \in \{i(1), \ldots, i(m(n))\}\right), \\
(\beta_{i}'(\w), \gamma_{i}'(\w), v_{i}'(\w)) & = (\beta_{i}'(\w'), \gamma_{i}'(\w'), v_{i}'(\w')) & \left(i \in \{1, \ldots, n\} \setminus\{i(1), \ldots, i(m(n))\} \right).
\end{aligned}
\]
Lemma \ref{lem:pivotEquiv} describes the structure of the equivalence class in $\mathcal{Q}_{n}'$. To elaborate, let us fix an equivalence class $\mathcal{F}' \in \mathcal{Q}_{n}'$. Then $\mathcal{F}'$ is associated with a set of indices $P_{n}(\mathcal{F}')= \{i(1) < \ldots < i(m(n)) < \ldots\}$. Moreover, an element $\w \in (S^{4} \times G)^{n}$ belongs to $\mathcal{F}'$ if and only if the following holds:
\begin{enumerate}[label=(\Roman*)]
\item the following isometries/Schottky sequences coincide with the fixed ones associated to $\mathcal{F}'$.
 \[\begin{array}{c}
\big(w_{0}'(\w), \alpha_{1}'(\w), \beta_{1}'(\w), \ldots, w_{i(1)-1}'(\w), \alpha_{i(1)}'(\w)\big), \\\big(\delta_{i(1)}'(\w), w_{i(1)}' (\w), \ldots, w_{i(2)-1}'(\w), \alpha_{i(2)}'(\w) \big), \\
\vdots, \\
\big(\delta_{i(m(n)-1)}'(\w), w_{i(m(n)-1)}' (\w), \ldots, w_{i(m(n))-1}'(\w), \alpha_{i(m(n))}'(\w) \big), \\
\big(\delta_{i(m(n))}'(\w), w_{i(m(n))}' (\w), \ldots, \delta_{n}'(\w),w_{n}'(\w) \big).
\end{array}
\]
\item $\big(\beta_{i(k)}' (\w), \gamma_{i(k)}'(\w), v_{i(k)}'(\w)\big) \in \tilde{S}$ for $k=1, \ldots, m(n)$.
\end{enumerate}
Moreover, for all $\w \in \mathcal{F}'$ we have $P_{n}(\w) = P_{n}(\mathcal{F}')$.

Item (I) says that $(w_{k}')_{k=0}^{m(n)}$ are constant on $\mathcal{F}'$. Item (II) says that, on each $\mathcal{F}'$, the RVs $(\beta_{k}', \gamma_{k}', v_{k}')_{k=1}^{m(n)}$ are $\big(\beta_{i(k)}, \gamma_{i(k)}, v_{i(k)}\big)$'s restricted on the set $\big\{ \big(\beta_{i(k)}, \gamma_{i(k)}, v_{i(k)}\big) \in \tilde{S}\big\}$. Hence, $\big(\beta_{i(k)}, \gamma_{i(k)}, v_{i(k)}\big)_{k=1}^{m(n)}$ conditioned on $\mathcal{F}'$ are i.i.d.s whose distribution is the restriction of $(\textrm{uniform measure on $S^{2}$})\times \mu$. This establishes Item (3) of the conclusion.

Next, Lemma \ref{lem:extremal} tells us that
\[
\left(\begin{array}{c} o, \,W_{i(1)-1} \Gamma(\alpha_{i(1)}), \,
W_{i(1)-1} \Pi(\alpha_{i(1)}) \Gamma(\beta_{i(1)}), \,W_{i(1)-1} \Pi(\alpha_{i(1)})\Pi(\beta_{i(1)}) v_{i(1)} \Gamma(\gamma_{i(1)}), \\
W_{i(1)-1} \Pi(\alpha_{i(1)})\Pi(\beta_{i(1)}) v_{i(1)}\Pi(\gamma_{i(1)}) \Gamma(\delta_{i(1)}), \\
W_{i(2)} \Gamma(\alpha_{i(2)}), \,
W_{i(2)-1} \Pi(\alpha_{i(2)}) \Gamma(\beta_{i(2)}), \,W_{i(2)-1} \Pi(\alpha_{i(2)})\Pi(\beta_{i(2)}) v_{i(2)} \Gamma(\gamma_{i(2)}), \\
W_{i(2)-1} \Pi(\alpha_{i(2)})\Pi(\beta_{i(2)}) v_{i(2)}\Pi(\gamma_{i(2)}) \Gamma(\delta_{i(2)}), \\
\vdots \\
W_{i(m(n))} \Gamma(\alpha_{i(m(n))}), \,
W_{i(m(n))-1} \Pi(\alpha_{i(m(n))}) \Gamma(\beta_{i(m(n))}), \\
W_{i(m(n))-1} \Pi(\alpha_{i(1)})\Pi(\beta_{i(m(n))}) v_{i(1)} \Gamma(\gamma_{i(m(n))}), \\
W_{i(m(n))-1} \Pi(\alpha_{i(m(n))})\Pi(\beta_{i(m(n))}) v_{i(m(n))}\Pi(\gamma_{i(m(n))}) \Gamma(\delta_{i(m(n))}), \,\ldots, \,W_{n} o\end{array}\right)
\]
is $D_{0}$-semi-aligned on $\mathcal{F}'$. Consequently, for each $\mathcal{F}' \in \mathcal{Q}_{n}'$ with the associated $\big(w_{i}'(\mathcal{F})\big)_{i=0}^{m(n)}$, and for each $\w \in \mathcal{F}'$ that determines $(\beta_{i}'(\w), \gamma_{i}'(\w), v_{i}'(\w))_{i=1}^{m(n)}$, \[
\Big (o, W_{0}' \Gamma(\beta_{1}'), V_{0}'\Gamma(\gamma_{1}'), \ldots, W_{m(n)-1}' \Gamma(\beta_{m(n)}'), V_{m(n)-1}'\Gamma(\gamma_{m(n)}'), W_{m(n)} o\Big)
\]
is $D_{0}$-semi-aligned for $W_{0}' := w_{0}$ and \[
V_{i}' := W_{i}' \Gamma(\beta_{i+1}') v_{i+1}', \quad W_{i+1}':=V_{i}' \Gamma(\gamma_{i+1}') w_{i+1}'  \quad \big(i=0, \ldots, m(n) - 1 \big).
\]
By Item (II), $(\beta_{i}', \gamma_{i}', v_{i}')$ ranges all over $\tilde{S}$ for each $i=1, \ldots, m(n)$. This implies that for each $\mathcal{F}'$,  $\big(w_{i}'(\mathcal{F})\big)_{i=0}^{m(n)}$ is $D_{0}$-pre-aligned. Since $\mathcal{B}_{n}'$ is composed of these equivalence classes, we can conclude Item (2) of the conclusion.\end{proof}

\subsection{Pivoting and self-repulsion} \label{subsection:pivot2}

We discuss the heart of this paper, the pivoting for translation length. The ingredient to begin with is a $D_{0}$-pre-aligned seqeuence $(w_{i})_{i=0}^{n}$ in $G$. This means that for every choice of $(\beta_{i}, \gamma_{i}, v_{i}) \in \tilde{S}$ for $i=1, \ldots, n$ the sequence \[
\big(o, W_{0} \Gamma(\beta_{1}), V_{0} \Gamma(\gamma_{1}), \ldots, W_{n-1} \Gamma(\beta_{n}), V_{n-1} \Gamma(\gamma_{n}), W_{n} o\big)
\]
is $D_{0}$-semi-aligned, where $W_{0} := w_{0}$ and \begin{equation}\label{eqn:vkWkNotat}
V_{k} := W_{k} \Pi(\beta_{k+1})v_{k+1}, \,\, W_{k+1} := V_{k} \Pi(\gamma_{k+1}) w_{k+1} \quad (k=0, \ldots, n-1).
\end{equation}
Our aim is to show that $W_{n}$ is contracting whose axis is fellow traveling with $W_{k} \Gamma(\beta_{k})$ and $V_{k} \Gamma(\gamma_{k})$ for intermediate $k$'s. For this let us define \[
\begin{aligned}
\phi_{k} &= \phi_{k} \big( (w_{i})_{i=0}^{n}, (\beta_{i}, \gamma_{i}, v_{i})_{i=1}^{n} \big) \\
&:= \left(V_{n-k} \Pi(\gamma_{n-k+1}) \right)^{-1} W_{n} W_{k-1} \\
&= w_{n-k+1} \Pi (\beta_{n-k+2}) v_{n-k+2} \Pi(\gamma_{n-k+2})w_{n-k+2} \cdots \Pi(\beta_{n}) v_{n} \Pi(\gamma_{n})w_{n} \\
& \quad \cdot w_{0} \Pi(\beta_{1}) v_{1} \Pi(\gamma_{1})w_{1} \cdots  \Pi(\beta_{k-1}) v_{k-1} \Pi(\gamma_{k-1})w_{k-1}.
\end{aligned}
\]
Then $\phi_{k}$ depends only on $(w_{l}, \beta_{l}, \gamma_{l}, v_{l})$'s for $1 \le l \le k-1$ and $n-k+2 \le l \le n$, together with $w_{0}, w_{n-k+1}$. For $1 \le k \le n/2$ we define\[
\begin{aligned}
S_{k}^{\ast} &:= \left\{ (\beta_{k}, \gamma_{k}) \in \tilde{S}(v_{k}) : \textrm{$\big( W_{n}^{-1} V_{n-k} \Pi(\gamma_{n-k+1}) o, W_{k-1} \Gamma(\beta_{k}) \big)$ is $K_{0}$-aligned}\right\} \\
&= \left\{ (\beta_{k}, \gamma_{k}) \in \tilde{S}(v_{k}) : \textrm{$\big( \phi_{k}^{-1} o, \Gamma(\beta_{k}) \big)$ is $K_{0}$-aligned}\right\}, \\
S_{n-k+1}^{\ast} &:= \left\{ (\beta_{n-k+1}, \gamma_{n-k+1}) \in \tilde{S}(v_{n-k+1}) : \textrm{$\big( W_{n}^{-1} V_{n-k} \Gamma(\gamma_{n-k+1}), W_{k-1} \Pi(\beta_{k})o \big)$ is $K_{0}$-aligned} \right\} \\
&= \left\{ (\beta_{n-k+1}, \gamma_{n-k+1}) \in \tilde{S}(v_{n-k+1}) : \textrm{$\big( \Gamma(\gamma_{n-k+1}), \Pi(\gamma_{n-k+1}) \phi_{k} \Pi(\beta_{k}) o \big)$ is $K_{0}$-aligned} \right\}.
\end{aligned}
\]

These sets are devised to capture the match between displacement and translation length.

\begin{lem}\label{lem:selfRepulse}
Let $1 \le l \le k \le n/2$, let $(w_{i})_{i = 0}^{n}$ be a $D_{0}$-pre-aligned sequence in $G$ and let $(\beta_{i}, \gamma_{i}, v_{i})_{i=1, \ldots, l, n-l+1, \ldots, n} \in \tilde{S}^{2l}$ be choices such that \begin{equation}\label{eqn:selfRepulse}
(\beta_{l}, \gamma_{l}) \in S_{l}^{\ast}\,\,\textrm{and}\,\, (\beta_{n-l+1}, \gamma_{n-l+1}) \in S_{n-l+1}^{\ast}.
\end{equation}
Then for any additional choice $(\beta_{i}, \gamma_{i}, v_{i})_{i=l+1, \ldots, k, n-k+1, \ldots, n-l} \in \tilde{S}^{2(k-l)}$, the isometry \[\begin{aligned}
\phi_{k+1} &:= w_{n-k} \Pi(\beta_{n-k+1}) v_{n-k+1} \Pi(\gamma_{n-k+1}) w_{n-2k+1} \cdots \Pi(\beta_{n}) v_{n} \Pi(\gamma_{n}) w_{n}\\
& \cdot w_{0} \Pi(\beta_{1}) v_{1} \Pi(\gamma_{1}) w_{1}\cdots  \Pi(\beta_{k}) v_{k} \Pi(\gamma_{k})w_{k}
\end{aligned}
\]
is a $D_{0}$-pre-aligned isometry.
\end{lem}

\begin{proof}
Pick an arbitrary further choices $(\beta_{i}, \gamma_{i}, v_{i}) \in \tilde{S}$ for $i=k+1, \ldots, n-k$. Recall again the notation in Display \ref{eqn:vkWkNotat} (with $W_{0} := w_{0}$). For notational purpose, let \[
\kappa_{2i-1} = W_{i-1} \Gamma(\beta_{i}), \kappa_{2i} = V_{i-1} \Gamma(\gamma_{i}) (i=1, \ldots, n).
\]
Since $(w_{i})_{i=0}^{n}$ is $D_{0}$-pre-aligned sequence, any subsequence of \[
(o, \kappa_{1}, \ldots, \kappa_{2n}, W_{n} o)
\]
is $D_{0}$-semi-aligned. 

Meanwhile, note that $W_{n}^{-1} V_{n-l} \Pi(\gamma_{n-l+1}) o$ is the ending point of $W_{n}^{-1} V_{n-l} \Gamma(\gamma_{n-l+1}) = W_{n}^{-1} \kappa_{2(n-l)+2}$, while $W_{l-1} \Pi(\beta_{l}) o$ is the ending point of $W_{l-1} \Gamma(\beta_{l}) o = \gamma_{2l-1}$. Since we are assuming Condition \ref{eqn:selfRepulse}, Lemma \ref{lem:1segment} tells us that $(W_{n}^{-1} \kappa_{2(n-l)+2}, \kappa_{2l-1})$ is $D_{0}$-aligned.

Overall, we observe that \[
(W_{n}^{-1} \kappa_{2(n-k)}, \ldots, W_{n}^{-1} \kappa_{2(n-l)+2},\kappa_{2l-1}, \ldots, \kappa_{2k+1})
\]
is $D_{0}$-semi-aligned. This implies that \[
(W_{n}^{-1} \kappa_{2(n-k)}, \kappa_{2k-1}) = \Big( \big(W_{n}^{-1} V_{n-k-1}\big) \cdot \Gamma(\gamma_{n-k}), \big(W_{n}^{-1} V_{n-1}\big) \cdot \Pi(\gamma_{n-k}) \phi_{k+1} \cdot \Gamma(\beta_{k+1}) \Big)
\] is $D_{0}$-semi-aligned. Also recall that the choices of $(\beta_{k+1}, \gamma_{k+1}, v_{k+1})$ and $(\beta_{n-k}, \gamma_{n-k}, v_{n-k})$ in $\tilde{S}$ were arbitrary. Hence, we conclude that $\phi_{k}$ is $D_{0}$-semi-aligned.
\end{proof}

\begin{lem}\label{lem:selfRepulse2}
Let $1 \le k \le n/2$, let $(w_{i})_{i = 0}^{n}$ be a $D_{0}$-pre-aligned sequence in $G$ and let $(\beta_{i}, \gamma_{i}, v_{i})_{i=1}^{n} \in \tilde{S}^{n}$ be choices such that $\phi_{k}$ is a $D_{0}$-pre-aligned isometry. Then $W_{n}$ is a contracting isometry, and moreover, 
$(\kappa_{i})_{i \in \mathbb{Z}}$ is $D_{0}$-semi-aligned for the Schottky axes \[\begin{aligned}
& \big(\kappa_{2(n - k+1) t + 1}, \, \kappa_{2(n-k+1)t+2}, \, \ldots, \kappa_{2(n-k+1)(t+1) - 1}, \, \kappa_{2(n-k+1)(t+1)}\big)\\
&:=  \big(W_{n}^{t} W_{k-1} \Gamma(\beta_{k}), \,W_{n}^{t} W_{k-1} \Gamma(\gamma_{k}), \,\ldots, \,W_{n}^{t} W_{n-k} \Gamma(\beta_{n-k+1}), \,W_{n}^{t} V_{n-k} \Gamma(\gamma_{n-k+1}) \big) \,\, (t \in \Z).
\end{aligned}
\]
\end{lem}

\begin{proof}
Since $\phi_{k}$ is $D_{0}$-pre-aligned, we have that \[\big( W_{n}^{-1} V_{n-k} \Pi(\gamma_{n-k+1}) o, W_{k-1} \Gamma(\beta_{k}) \big), \,\,\big( W_{n}^{-1} V_{n-k} \Gamma(\gamma_{n-k+1}), W_{k-1} \Pi(\beta_{k})o \big)
\]
is $D_{0}$-semi-aligned, as well as $(\kappa_{2(n-k+1)t - 1}, \kappa_{2(n-k+1) t})$ for $t \in \Z$.

Moreover, since $(w_{i})_{i=0}^{n}$ is also $D_{0}$-pre-aligned, we have that \[
\big(W_{k-1} \Gamma(\beta_{k}), V_{k-1} \Gamma(\gamma_{k}), \ldots, W_{n-k} \Gamma(\beta_{n-k+1}), V_{n-k} \Gamma(\gamma_{n-k+1})\big)
\]
is $D_{0}$-semi-aligned. This implies that $(\kappa_{2(n-k+1) t + 1}, \ldots, \kappa_{2(n-k+1) (t+1)})$ is $D_{0}$-aligned for each $t \in \Z$. Combining these facts, we conclude that $(\kappa_{i})_{i \in \Z}$ is $D_{0}$-semi-aligned.

Since Schottky axes are $K_{0}$-contracting and sufficiently long (as described in Definition \ref{dfn:Schottky}), we can apply Lemma \ref{lem:BGIPConcat} to conclude that the concatenation of $\{\kappa_{2(n-k+1) l}\}_{l \in \Z}$ is a contracting axis. Since the orbit $\{w^{l} o\}_{l\in \Z}$ is fellow traveling with this concatenation, $w$ is contracting as desired.
\end{proof}

We now estimate the probability for the event described in Lemma \ref{lem:selfRepulse}.

\begin{lem}\label{lem:trLengthPivotProb}
Fix $(w_{i})_{i=0}^{n} \in G^{n+1}$ and let $1 \le k \le n/2$. Picking $(\beta_{i}, \gamma_{i}, v_{i})_{i=1}^{n}$ from $\tilde{S}^{n}$, the condition \begin{equation}\label{eqn:trLengthCond}
(\beta_{l}, \gamma_{l}) \in \tilde{S}_{l}^{\ast}\,\, \textrm{and} \,\,(\beta_{n-l+1}, \gamma_{n-l+1}) \in \tilde{S}_{n-l+1}^{\ast} \,\,\textrm{for some $1 \le l \le k$}
\end{equation}
depends only on $\{(\beta_{l}, \gamma_{l}, v_{l}) : l=1, \ldots, k, n-k+1, \ldots, n\}$.

Moreover, fixing any choice of $\{v_{l} : l=1, \ldots, k, n-k+1, \ldots, n\}$, the condition holds for probability at least $1-(6\#S)^{k}$ with respect to the uniform measure on $\tilde{S}(v_{1}) \times \cdots \times \tilde{S}(v_{k}) \times \tilde{S}_{n-k+1}(v_{n-k+1}) \times \cdots \times \tilde{S}(v_{n})$.
\end{lem}

\begin{proof}
Recall that $\phi_{1}, \ldots, \phi_{k}$ depend only on $(\beta_{i}, \gamma_{i}, v_{i})$'s for $i=1, \ldots, k-1, n-k+2, \ldots, n$. Moreover, Condition \ref{eqn:trLengthCond} involves $\phi_{1}, \ldots, \phi_{k}$ and $(\beta_{i}, \gamma_{i})$ for $i=1, \ldots, k, n-k+1, \ldots, n$. Combining these leads to the first assertion.

Before proving the second assertion, we observe the following claim:

\begin{claim}\label{claim:trLengthPivotProb}
\begin{enumerate}
\item For each $1 \le l \le n/2$ and for each choices of \[\begin{aligned}
(\beta_{1}, \gamma_{1}, \ldots, \beta_{l-1}, \gamma_{l-1}, \beta_{n-l+2}, \gamma_{n-l+2}, \ldots, \beta_{n}, \gamma_{n}) &\in S^{4l-4},\\
(v_{1}, \ldots, v_{l}, v_{n-l+1}, \ldots, v_{n}) &\in G^{2l},
\end{aligned}
\] the following holds with respect to the uniform measure on $\tilde{S}(v_{l})$:\[
\Prob\left( (\beta_{l}, \gamma_{l})\in \tilde{S}(v_{l}) : \left(\phi_{l}^{-1} o, \Gamma(\beta_{l})\right) \,\textrm{is $K_{0}$-aligned}\right) \ge 1 - 3/\#S.
\]
\item For each $1 \le l \le n/2$ and for each choice of\[\begin{aligned}
(\beta_{1}, \gamma_{1}, \ldots, \beta_{l}, \gamma_{l}, \beta_{n-l+2}, \gamma_{n-l+2}, \ldots, \beta_{n}, \gamma_{n}) &\in S^{4l-2},\\
(v_{1}, \ldots, v_{l}, v_{n-l+1}, \ldots, v_{n}) &\in G^{2l},
\end{aligned}
\]the following holds with respect to the uniform measure on $\tilde{S}(v_{n-l+1})$: \[
\Prob\left( \begin{array}{c}(\beta_{n-l+1}, \gamma_{n-l+1})\in \tilde{S}(v_{n-l+1}) : \\\big(\Gamma(\gamma_{n-l+1}), \,\Pi(\gamma_{n-l+1}) \phi_{l} \Pi(\beta_{l}) o \big)\,\textrm{is $K_{0}$-aligned}\end{array} \right) \ge 1-3/\#S.
\]
\end{enumerate}
\end{claim}

\begin{proof}[Proof of Claim \ref{claim:trLengthPivotProb}]
Recall that $\phi_{l}$ is determined by $(\beta_{t}, \gamma_{t}, v_{t})$'s for $t =1, \ldots, l-1, n-l+2, \ldots, n$. With this in mind, the property of the Schottky set $S$ (cf. Definition \ref{dfn:Schottky}) implies that the condition in Item (1) is violated by at most 1 choice of $\beta_{l} \in S$; for any other choice of $\beta_{l}$ and for any choice of $\gamma_{l}$, the condition holds. Hence, at most $\#S$ choices in $S^{2} \supseteq \tilde{S}(v_{l})$ violate the condition. Meanwhile, since $\#\tilde{S}(v_{l}) \ge (\#S)^{2} - 2\#S$ holds for any $v_{l}$, we conclude that \[
\Prob\left( (\beta_{l}, \gamma_{l})\in \tilde{S}(v_{l}) : \left(\phi_{l}^{-1} o, \Gamma(\beta_{l})\right) \,\textrm{is not $K_{0}$-aligned}\right) \le \frac{\#S}{(\#S)^{2} - 2\#S} \le \frac{3\#S}{(\#S)^{2}}.
\]
We similarly deduce Item (2), and Claim \ref{claim:trLengthPivotProb} is now established.
\end{proof}

We now prove the second assertion. For this, we first fix $\{v_{l} : l=1, \ldots, k, n-k+1, \ldots, n\}$ and define \[
\mathcal{A}_{k} := \left\{ (\beta_{l}, \gamma_{l})_{l=1}^{n}: \begin{array}{c}(\beta_{l}, \gamma_{l}) \notin S_{l}^{\ast}\,\textrm{or}\, (\beta_{n-l + 1}, \gamma_{n-l+1}) \notin S_{n-l+1}^{\ast}  \\
 \textrm{for each $1 \le l \le k$}\end{array}\right\}.
\]
Then $\mathcal{A}_{k} \subseteq \mathcal{A}_{k+1}$ holds for each $k$. Further, the first assertion implies that the membership of $(\beta_{i}, \gamma_{i})_{i=1}^{n}$ in $\mathcal{A}_{k}$ depends only on $(\beta_{i}, \gamma_{i})$ for $i=1, \ldots, k, n-k+1, \ldots, n$. Finally, for each choice of $(\beta_{i}, \gamma_{i})_{i=1, \ldots, k-1, n-k+2, \ldots, n}$ that makes $(\beta_{i}, \gamma_{i})_{i=1}^{n}$ belong to $\mathcal{A}_{k-1}$, Claim \ref{claim:trLengthPivotProb} implies that  \[
\Prob \left( \mathcal{A}_{k} \, \Big| \, \beta_{1}, \gamma_{1}, \ldots, \beta_{k-1}, \gamma_{k-1}, \beta_{n-k+2},\gamma_{n-k+2}, \ldots, \beta_{n}, \gamma_{n} \right) \le 6/\#S.
\]
Summing up these conditional probabilities, we deduce that $\Prob(\mathcal{A}_{k}) \le (6/\# S) \cdot \Prob(\mathcal{A}_{k-1})$. This leads to the conclusion of the lemma.
\end{proof}

\subsection{Converse of CLT} \label{subsection:CLTConverse}

In this subsection, we prove the converse of CLT. We first assemble Lemma \ref{lem:pivotFirstReduce}, \ref{lem:pivotSecondReduce} \ref{lem:selfRepulse} and \ref{lem:trLengthPivotProb}.

\begin{cor}\label{cor:pivotCorCombined}
Let $\mu$ be a non-elementary probability measure on $G$ and let $S$ be a fairly long $K_{0}$-Schottky set for $\mu$ with cardinality at least $400$. Then for each $n$ there exist $m(n) \in \{2^{s} : s > 0\}$, a probability space $\Omega_{n}$ with a measurable subset $A_{n} \subseteq \Omega_{n}$, a measurable partition $\mathcal{P}_{n}$ of $A_{n}$ and RVs \[
\begin{aligned}
Z_{n} &\in G, \\
\{w_{i} : i = 0, \ldots, m(n)\} &\in G^{m(n) + 1}, \\
\{v_{i} : i = 1, \ldots, m(n)\} &\in G^{m(n)}, \\
\{\beta_{i}, \gamma_{i} : i = 1, \ldots, m(n)\} &\in S^{2m(n)}
\end{aligned}
\]
such that the following hold: \begin{enumerate}
\item $\lim_{n \rightarrow +\infty} \Prob(A_{n}) = 1$ and $\lim_{n \rightarrow +\infty} m(n) / n > 0$.
\item On $A_{n}$, $(w_{i})_{i=0}^{m(n)}$ is a $D_{0}$-pre-aligned sequence in $G$ and $w_{m(n)}^{-1} w_{0}$ is a $D_{0}$-pre-aligned isometry.
\item On each equivalence class $\mathcal{E} \in \mathcal{P}_{n}$, $(w_{i})_{i=0}^{m(n)}$ are constant and $(\beta_{i}, \gamma_{i}, v_{i})_{i=1}^{m(n)}$ are i.i.d.s distributed according to the measure $\left(\textrm{uniform measure on $S^{2}$}\right) \times \mu$ conditioned on $\tilde{S}$.
\item $Z_{n}$ is distributed according to $\mu^{\ast n}$ on $\Omega_{n}$ and \[
Z_{n} = w_{0} \Pi(\beta_{1}) v_{1} \Pi(\gamma_{1}) w_{1} \cdots \Pi(\beta_{m(n)})v_{m(n)} \Pi(\gamma_{m(n)}) w_{m(n)}
\]holds on $A_{n}$.
\end{enumerate}
\end{cor}

\begin{proof}
Given $n>0$, Lemma \ref{lem:pivotFirstReduce} provides an integer $N_{1} = m(n)$, a probability space $(\Omega_{n}, \Prob)$ with a subset $B=B_{n}$, a measurable partition $\mathcal{Q}_{n} = \{\mathcal{F}_{\xi}\}_{\xi}$ of $B_{n}$, random variables $Z_{n} \in G$, $(w_{i})_{i=0}^{N_{1}}\in G^{N_{1}+1}$, $(v_{i}'')_{i=1}^{N_{1}} \in G^{N_{1}}$ and $(\alpha_{i}, \beta_{i}, \gamma_{i}, \delta_{i})_{i=1}^{N_{1}} \in S^{4N_{1}}$ such that the following holds: \begin{enumerate}
\item $\lim_{n} \Prob(B) = 1$ and $\lim_{n \rightarrow +\infty} N_{1}/n > 0$.
\item On each equivalence class $\mathcal{G}_{\gamma} \in \mathcal{R}$, $(w_{i})_{i}$ are constant and $(\alpha_{i}, \beta_{i}, \gamma_{i}, \delta_{i}, v_{i})_{i}$ are i.i.d.s distributed according to $(\textrm{uniform measure on $S^{4}$}) \times \mu$.
\item $Z_{n}$ is distributed according to $\mu^{\ast n}$ on $\Omega_{n}$ and \[
Z_{n} = w_{0} \Pi(\alpha_{1}) \Pi(\beta_{1})v_{1} \Pi(\gamma_{1}) \Pi(\delta_{1}) w_{1} \cdots \Pi(\alpha_{N_{1}}) \Pi(\beta_{N_{1}}) v_{N_{1}} \Pi(\gamma_{N_{1}}) \Pi(\delta_{N_{1}}) w_{N_{1}}
\]
holds on $B = \cup_{\xi} \mathcal{\mathcal{F}}_{\xi}$.
\end{enumerate}

Next, we take $N_{2}(n) :=  2\lfloor 0.25 N_{1} \rfloor$ as in Lemma \ref{lem:pivotSecondReduce}, with input $N_{1}$ in place of $n$. Since $N_{1}$ grows linearly, $\lim_{n} N_{2}/n > 0$ also holds. We now endow each $\mathcal{F}_{\xi} \in \mathcal{Q}_{n}$ with the conditional measure $\Prob|_{\mathcal{F}_{\xi}}$ of $\Prob$. 

We note that $(w_{i}, v_{i}, \alpha_{i}, \beta_{i}, \gamma_{i}, \delta_{i})_{i}$ satisfy the assumption of Lemma \ref{lem:pivotSecondReduce} on each $\mathcal{F}_{\xi} \in \mathcal{Q}_{n}$. For each $\mathcal{F}_{\xi} \in \mathcal{Q}_{n}$, Lemma \ref{lem:pivotSecondReduce} then provides a measurable subset $B_{\xi}' \subseteq \mathcal{F}_{\xi}$, a measurable partition $\mathcal{Q}_{\xi}' = \{\mathcal{F}_{\zeta;\xi}'\}_{\zeta}$ of $B_{\xi}'$, and RVs $(w_{i}')_{i=0}^{N_{2}}, (v_{i}', \beta_{i}', \gamma_{i}')_{i=1}^{N_{2}}$ on $\mathcal{F}_{\xi}$ such that the following holds: \begin{enumerate}
\item  $\Prob\big(B_{\xi}'\big|\mathcal{F}_{\xi}\big) \ge 1-\big(3 \sqrt[4]{4/N_{0}}\big)^{N_{1}} \ge 1-\big(3/\sqrt{10}\big)^{N_{1}}$.
\item On $B_{\xi}$ we have\[\begin{aligned}
Z_{n}&=w_{0} \Pi(\alpha_{1}) \Pi(\beta_{1}) v_{1}\Pi(\gamma_{1}) \Pi(\delta_{1}) w_{1} \cdots \Pi(\alpha_{N_{1}}) \Pi(\beta_{N_{1}}) v_{N_{1}}\Pi(\gamma_{N_{1}}) \Pi(\delta_{N_{1}}) w_{N_{1}}\\
&= w_{0}' \Pi(\beta_{1}') v_{1}'\Pi(\gamma_{1}') w_{1}' \cdots \Pi(\beta_{N_{2}}') v_{N_{2}}'\Pi(\gamma_{N_{2}}') w_{N_{2}}'.
\end{aligned}
\]
\item On each $\mathcal{F}_{\zeta;\xi}' \in \mathcal{Q}_{\xi}'$, $(w_{i}')_{i=0}^{N_{2}}$ is a fixed $D_{0}$-pre-aligned sequence in $G$ and $(\beta_{i}', \gamma_{i}', v_{i}')_{i=1}^{N_{2}}$ are i.i.d.s distributed according to the restriction of $(\textrm{uniform measure on $S^{2}$}) \times \mu$ to $\tilde{S}$.
\end{enumerate}
We then take $B' := \sqcup_{\xi} B_{\xi}'$ and $\mathcal{Q}'_{n} := \sqcup_{\xi} \mathcal{Q}_{\xi}'$. Note that \[
\Prob(B') = \sum_{\xi} \Prob\big(B_{\xi}' \big|\mathcal{F}_{\xi}\big) \cdot \Prob(\mathcal{F}_{\xi}) \ge \big(1-(3/\sqrt{10})^{N_{1}}\big) \cdot \Prob(B),
\]
and the RHS converges to $1$ as $n$ tends to infinity; hence $\lim_{n} \Prob(B') = 1$.

We now set $m(n) = 2^{ \lfloor \log_{2} N_{2} \rfloor - 1}$. Note that $m(n)$ is a power of $2$ and $N_{2} - m(n)$ is even. Moreover, we have $m(n) \ge N_{2}/4$ and $N_{2} - m(n) \ge N_{2}/2$. 

We now consider a refinement $\mathcal{Q}''_{n}$ of $\mathcal{Q}'_{n}$ by first fixing the values of the RVs\[
\Big\{
v_{i}' : i = 1, 2, \ldots, 0.5(N_{2} - m(n)), N_{2} - 0.5(N_{2} - m(n))+ 1, \ldots,N_{2}-1, N_{2}\Big\}.
\]
 We also consider a refinement $\mathcal{P}_{n}$ of $\mathcal{Q}''_{n}$ by further fixing the values of the RVs \[
 \Big\{\beta_{i}', \gamma_{i}': i = 1,2, \ldots, 0.5(N_{2} - m(n)), N_{2} - 0.5(N_{2} - m(n))+ 1, \ldots, N_{2}-1, N_{2} \Big\}.
 \] We consider the following property: \[
P = ``(\beta_{l}', \gamma_{l}') \in \tilde{S}_{l}^{\ast}\,\, \textrm{and} \,\,(\beta_{N_{2} - l+1}', \gamma_{N_{2} -l + 1}') \in \tilde{S}_{N_{2} - l + 1}^{\ast} \,\,\textrm{for some $1 \le l \le 0.5(N_{2} - m(n))$}".
\]
The first part of Lemma \ref{lem:trLengthPivotProb} implies that, for each $\mathcal{E} \in \mathcal{P}_{n}$, either $P$ holds on the entire $\mathcal{E}$ or $P$ does not hold on the entire $\mathcal{E}$. The second part of Lemma \ref{lem:trLengthPivotProb} tells us that $\Prob(P | \mathcal{F}'') \ge 1 - (6/N_{0})^{0.5(N_{2} - m(n))} \ge 1 - (6/N_{0})^{N_{2}/4}$ for each $\mathcal{F}'' \in \mathcal{Q}_{n}''$. If we take \[
A_{n} := \bigcup \{ \mathcal{E}\in\mathcal{P} : \textrm{$P$ holds on $\mathcal{E}$}\},
\]
then we get \[
\Prob(A_{n}) \ge \sum_{\mathcal{F}'' \in \mathcal{Q}_{n}''} \Prob(P | \mathcal{F}'') \cdot \Prob(\mathcal{F}'') \ge (1 - (6/N_{0})^{N_{2}/4}) \cdot \Prob(B_{n}'),
\]
which converges to 1 as $n$ tends to infinity. Item (1) of the conclusion is now proven.
By abuse of notation, from now on, we let $\mathcal{P}_{n}$ be the restriction of $\mathcal{P}_{n}$ on $A_{n}$. 

Now, on $A_{n}$ we define RVs\[
\begin{aligned}
w_{0}'' &:=  w_{0}' \Pi(\beta_{1}') v_{1}' \Pi(\gamma_{1}') w_{1}' \cdots  \Pi(\beta_{k}') v_{k}' \Pi(\gamma_{k}')w_{k}', \\
w_{m(n)}'' &:= w_{N_{2}-k} '\Pi(\beta_{N_{2} - k + 1}') v_{N_{2} - k+1}' \Pi(\gamma_{N_{2} - k+ 1}') w_{N_{2} - k+ 1}' \cdots  \Pi(\beta_{N_{2}}') v_{N_{2}}' \Pi(\gamma_{N_{2}}') w_{N_{2}}', \\
(v_{i}'', \beta_{i}'', \gamma_{i}'')_{i=1}^{m(n)} &:= (v_{i + k}', \beta_{i+k}', \gamma_{i+k}')_{i=1}^{m(n)}, \\
(w_{i}'')_{i=1}^{m(n)-1}& := (w_{i+k}')_{i=1}^{m(n)-1} \quad \left(k := \frac{N_{2} - m(n)}{2} \right).
\end{aligned}
\]
Our goal is to prove that $(w_{i}'')_{i=0}^{m(n)}$ and $(\beta_{i}'', \gamma_{i}'', v_{i}'')_{i=1}^{m(n)}$ (in place of $(w_{i})_{i=0}^{m(n)}$ and $(\beta_{i}, \gamma_{i}, v_{i})_{i=1}^{m(n)}$) satisfy the desired conclusion. First, we have \[\begin{aligned}
&w_{0}'' \Pi(\beta_{1}'') v_{1}''\Pi(\gamma_{1}'') w_{1}'' \cdots \Pi(\beta_{m(n)}'') v_{m(n)}''\Pi(\gamma_{m(n)}'') w_{m(n)}''\\
&=
w_{0}' \Pi(\beta_{1}') v_{1}'\Pi(\gamma_{1}') w_{1}' \cdots \Pi(\beta_{N_{2}}') v_{N_{2}}'\Pi(\gamma_{N_{2}}') w_{N_{2}}'
\end{aligned}
\]
holds on $A_{n}$ by definition, and the RHS equals $Z_{n}$ on $B_{n}$. This establishes Item (4) of the conclusion.

By definition of $\mathcal{P}_{n}$, $(w_{i}'')_{i=0}^{m(n)}$ are constant and $(\beta_{i}'', \gamma_{i}'', v_{i}'')_{i=1}^{m(n)}$ are i.i.d.s distributed according to $(\textrm{uniform measure on $S^{2}$}) \times \mu$ on each $\mathcal{E}_{\alpha} \in \mathcal{P}$. Item (3) of the conclusion follows.

Meanwhile, pick an arbitrary $\w \in A_{n}$, which determines $(w_{i}')_{i=0}^{N_{2}}$ and $(w_{i}'')_{i=0}^{m(n)}$. Since $\w$ belongs to $B_{n}'$, $(w_{i}')_{i=1}^{N_{2}}$ is a $D_{0}$-pre-aligned sequence. This in turn implies that $(w_{i}'')_{i=0}^{m(n)}$ is also $D_{0}$-pre-aligned. Moreover, Lemma \ref{lem:selfRepulse} implies that \[
(w_{m(n)}'')^{-1} \cdot w_{0}'' =\Big(\textrm{$\phi_{k+1}$ for $(w_{i}')_{i=0}^{N_{2}}$ and $(\beta_{i}', \gamma_{i}', \delta_{i}')_{i=1}^{N_{2}}$}\Big)
\] is a $D_{0}$-pre-aligned isometry. Item (2) of the conclusion is now established and the proof is complete.
\end{proof}

We now investigate the displacement and the translation length of $Z_{n}$ in Corollary \ref{cor:pivotCorCombined}.

\begin{lem}\label{lem:alignedGromov}
Let $S$ be a fairly long $K_{0}$-Schottky set for $\mu$ with cardinality at least $400$, and fix a $D_{0}$-pre-aligned sequence $(w_{i})_{i=0}^{n}$ in $G$. Given $(\beta_{i}, \gamma_{i}, v_{i}) \in \tilde{S}$ for $i=1, \ldots, n$, define $W_{0}:= w_{0}$ and \[
\begin{aligned}
V_{k} &:= W_{k} \Gamma(\beta_{k+1}) v_{k+1}, \,\, W_{k+1} := V_{k} \Gamma(\gamma_{k+1}) w_{k+1}& (k=0, \ldots, n-1), \\
(x_{2k-1}, x_{2k}) &:= \big(W_{k-1} o, \,\,V_{k-1} \Pi(\gamma_{k}) o\big) & (k=1, \ldots, n), \\
x_{0} &:= 0, \,\,x_{2n+1} := W_{n} o.&
\end{aligned}
\]
Let $\mu$ be a probability measure on $G$ with infinite second moment, and let $(\beta_{i}, \gamma_{i}, v_{i})_{i=1}^{n}$ be i.i.d.s distributed according to $(\textrm{uniform measure on $S^{2}$}) \times \mu$ conditioned on $\tilde{S}$. Then the following holds:\begin{enumerate}
\item $\{d(x_{2l-1}, x_{2l}):l=1, \ldots, n\}$ are i.i.d.s with infinite second moment, whose distribution depends on $S$ and $\mu$ but is independent of $w_{i}$'s.
\item $\{d(x_{2l}, x_{2l+1}) : l = 0, \ldots, n\}$ are constant RVs;
\item for each $0 \le i \le j \le k \le 2n+ 1$, $(x_{i}, x_{k})_{x_{j}}$ is bounded by $E_{0}$.
\item for each $0 \le i \le j \le k \le i' \le j' \le k' \le 2n+1$, $(x_{i}, x_{k})_{x_{j}}$ and $(x_{i'}, x_{k'})_{x_{j'}}$ are independent.
\end{enumerate}
\end{lem}

\begin{proof}
Note that $[x_{2l-1}, x_{2l}]$'s are translates of $[o, \Pi(\beta_{l}) v_{l} \Pi(\gamma_{l}) o]$'s, which are i.i.d.s. 

We denote by $\mu_{S}$ the uniform measure on $S$. Recall that $\#\tilde{S}(v) \ge (\#S)^{2} - 2\#S$ holds for each $v\in G$. As a consequence, for each $g \in G$ we have\[
0.995\cdot\mu(g) \le \frac{(\#S)^{2} - 2\#S}{(\#S)^{2}} \mu(g) \le (\mu_{S}^{2} \times \mu) \big(\{(\beta, \gamma, v) \in \tilde{S} : v= g\} \big) \le \mu(g).
\]
Also, we have \[
(\mu_{S}^{2} \times \mu) (\tilde{S}) \ge \frac{(\#S)^{2} - 2\#S}{(\#S)^{2}} \le0.995.
\]
Hence, if we denote by $\mu'$ the common distribution of $v_{i}$'s, then  \[
0.99 \mu(g) \le \mu'(g) \le 1.01 \mu(g)
\]
holds for each $g \in G$. Since $\mu$ is assumed to have infinite second moment, $\mu'$ also has infinite second moment. Now note that \[\begin{aligned}
|d(x_{2l-1}, x_{2l}) - d(o, v_{l} o)| &= \left|d\big(\Pi(\beta_{l})^{-1} o, v_{l} \Pi(\gamma_{l}) o\big) - d(o, v_{l} o) \right|\\
& \le 2\max\{d(o, \Pi(\mathbf{s}) o) : \mathbf{s} \in S\}.
\end{aligned}
\]
This implies that $d(x_{2l-1}, x_{2l})$ also has infinite second moment. 

Next, recall that $(w_{i})_{i=0}^{n}$ is a fixed $D_{0}$-pre-aligned sequence. Hence,  $d(x_{2l}, x_{2l+1}) = d(o, w_{l})$ is constant. Moreover, the sequence \begin{equation}\label{eqn:alignedGromov}
(o, W_{0} \Gamma(\beta_{1}), V_{0} \Gamma(\gamma_{1}), \ldots, W_{n-1} \Gamma(\beta_{n}), V_{n-1} \Gamma(\gamma_{n-1}), W_{n} o)
\end{equation}
is $D_{0}$-semi-aligned for any choices of $(\beta_{i}, \gamma_{i}, v_{i})$'s. Also, note that $(x_{2l-1}, x_{2l})$ are the beginning and ending points of the $i$-th Schottky axes in Display \ref{eqn:alignedGromov}. By Lemma \ref{lem:BGIPWitness}, there exist points $y_{1}, \ldots, y_{2n}$ on $[x_{0}, x_{2n+1}]$, from left to right, so that $d(x_{i}, y_{i}) \le 0.1E_{0}$ for $i=1, \ldots, n$. It follows that \[
(x_{i}, x_{k})_{x_{j}} \le (y_{i}, y_{k})_{y_{j}} + 0.15E_{0} < E_{0} \quad (i\le j \le k).
\]

Finally, for any $k \in \{0, 1,\ldots, 2n+1\}$, the RVs $\{d(x_{i}, x_{k}) : i \le k\}$ depend only on the choices of $\{(\beta_{l}, \gamma_{l}, v_{l}) : l \le k/2\}$, while  $\{d(x_{i}, x_{k}) : i \ge k\}$ depend on the choices of $\{(\beta_{l}, \gamma_{l}, v_{l}) : l > k/2\}$. It follows that the two collections of RVs are independent, and this settles item (4) of the claim.
\end{proof}

\begin{lem}
\label{lem:alignedGromovTrans}
Let $S$ be a fairly long $K_{0}$-Schottky set for $\mu$ with cardinality at least $400$, and fix a $D_{0}$-pre-aligned sequence $(w_{i})_{i=0}^{n}$ in $G$ such that $w_{n}^{-1} w_{0}$ is a $D_{0}$-pre-aligned isometry. Given $(\beta_{i}, \gamma_{i}, v_{i}) \in \tilde{S}$ for $i=1, \ldots, n$, define $x_{0}:= w_{n}^{-1} o$, $W_{0}:= w_{0}$ and \[
\begin{aligned}
V_{k} &:= W_{k} \Gamma(\beta_{k+1}) v_{k+1}, \,\, W_{k+1} := V_{k} \Gamma(\gamma_{k+1}) w_{k+1}& (k=0, \ldots, n-1), \\
(x_{2k-1}, x_{2k}) &:= \big(W_{k-1} o, \,\,V_{k-1} \Pi(\gamma_{k}) o\big) & (k=1, \ldots, n).
\end{aligned}
\]
Let $\mu$ be a probability measure on $G$ with infinite second moment, and let $(\beta_{i}, \gamma_{i}, v_{i})_{i=1}^{n}$ be i.i.d.s distributed according to $(\textrm{uniform measure on $S^{2}$}) \times \mu$ conditioned on $\tilde{S}$. Then, in addition to all the conclusions of Lemma \ref{lem:alignedGromov}, the following also holds: \begin{enumerate}\setcounter{enumi}{4}
\item the translation length $\tau(W_{n})$ and $d(x_{0}, x_{2n})$ differ by at most $E_{0}$.
\end{enumerate}
\end{lem}

\begin{proof}
The single change is that $x_{0}$ is now $w_{n}^{-1} o$ instead of $o$. Since $w_{n}^{-1} w_{0}$ is assumed to be $D_{0}$-pre-aligned, we have that \[
\big( w_{n}^{-1} \Pi(\gamma_{n})^{-1} \Gamma(\gamma_{n}), w_{0} \Gamma(\beta_{1}) \big) =\big( W_{n}^{-1} V_{n-1} \Gamma(\gamma_{n}), W_{0} \Gamma(\beta_{1}) \big)
\]
is $D_{0}$-semi-aligned for any choice of $(\beta_{1}, \gamma_{1}, v_{1}), (\beta_{n}, \gamma_{n}, v_{n}) \in \tilde{S}$. Moreover, since $(w_{i})_{i=0}^{n}$ is assumed to be $D_{0}$-pre-aligned, we still have that \[
(o, W_{0} \Gamma(\beta_{1}), V_{0} \Gamma(\gamma_{1}), \ldots, W_{n-1} \Gamma(\beta_{n}), V_{n-1} \Gamma(\gamma_{n}), W_{n} o)
\]
is $D_{0}$-semi-aligned for any choices of $(\beta_{i}, \gamma_{i}, v_{i}) \in \tilde{S}$. Combining these, we observe that \[
(x_{0}, W_{0} \Gamma(\beta_{1}), V_{0} \Gamma(\gamma_{1}), \ldots, W_{n-1} \Gamma(\beta_{n}), V_{n-1} \Gamma(\gamma_{n}), W_{n} o)
\]
is always $D_{0}$-semi-aligned. Item (1), (2), (3), (4) of the conclusion follows.

It remains to prove the additional conclusion. For this, define \begin{equation}\label{eqn:alignedGromovTransAssu}
(\kappa_{1}, \kappa_{2}, \ldots, \kappa_{2n-1}, \gamma_{2n}) := \big(W_{0}\Gamma(\beta_{1}), V_{0} \Gamma(\gamma_{1}), \ldots, W_{n-1}\Gamma(\beta_{n}), V_{n-1} \Gamma(\gamma_{n}) \big)
\end{equation}
and $\kappa_{2nt + k} := W_{n}^{t} \cdot \kappa_{k}$ for each $t \in \Z$ and $k = 1, \ldots, 2n$. Since $(\kappa_{0}, \kappa_{1})$ and $(\kappa_{1}, \ldots, \kappa_{2n})$ are always $D_{0}$-semi-aligned, we conclude that $(\ldots, \kappa_{-1}, \kappa_{0}, \kappa_{1}, \kappa_{2}, \ldots)$ is $D_{0}$-semi-aligned for any $t>0$. At this point, note that $W_{n}^{t} x_{0} = W_{n}^{t-1} V_{n-1} \Pi(\gamma_{n}) o$ is the ending point of $\kappa_{2nt}$ for each $t$. By applying Lemma \ref{lem:BGIPWitness} to the $D_{0}$-semi-aligned sequence \[
(\kappa_{0} \ni x_{0}, \kappa_{1}, \kappa_{2}, \ldots, \kappa_{2nt-1}, W_{n}^{t} x_{0} \in \kappa_{2nt}),
\]
we deduce that there exist points $y_{1}, \ldots, y_{t-1}$ on $[x_{0}, W_{n}^{t} x_{0}]$, in order from closest to farthest from $x_{0}$, such that $d(W_{n}^{k} x_{0}, y_{k}) \le 0.1 E_{0}$ for each $k$.  This implies that \[\begin{aligned}
d(x_{0}, W_{n}^{t} x_{0}) &= d(x_{0}, y_{1}) + \sum_{k=1}^{t-2} d(y_{k}, y_{k+1}) + d(y_{t-1}, W_{n}^{t} x_{0}) \\
&\ge d(x_{0}, W_{n} x_{0}) + \sum_{k=1}^{t-2} d(W_{n}^{k} x_{0}, W_{n}^{k+1} x_{0}) + d(W_{n}^{t-1} x_{0}, W_{n}^{t} x_{0}) - t E_{0} \\
&\ge t \big(d(x_{0}, x_{2n}) - E_{0}\big).
\end{aligned}
\]
After dividing by $t$ and sending $t$ to infinity, we obtain $\tau(W_{n}) \ge d(x_{0}, x_{2n}) - E_{0}$. Meanwhile, the triangle inequality implies that  \[
d(x_{0}, W_{n}^{t} x_{0}) \le \sum_{k=1}^{t} d(W_{n}^{k-1} x_{0}, W_{n}^{k} x_{0}) \le t d(x_{0}, x_{2n})
\]
and hence $\tau(W_{n}) \le d(x_{0}, x_{2n})$. The conclusion is now established.
\end{proof}

We are now ready to prove:

\begin{prop}\label{prop:CLTConverse}
Let $(X, G, o)$ be as in Convention \ref{conv:main} and let $(Z_{n})_{n \ge 0}$ be the random walk on $G$ generated by a non-elementary measure $\mu$ with infinite second moment. Then for any sequence $(c_{n})_{n}$ of real numbers,  neither of $\frac{1}{\sqrt{n}} (d(o, Z_{n} o) - c_{n})$ and $\frac{1}{\sqrt{n}} (\tau(Z_{n}) - c_{n})$ converges in law.
\end{prop}

\begin{proof}
We follow the proof given in \cite[Section 6.1]{choi2021clt}. Let $S$ be a fairly long $K_{0}$-Schottky set for $\mu$ with cardinality at least $400$. Since we are concerned with the distribution of $Z_{n}$, we may replace $Z_{n}$ with the one in Corollary \ref{cor:pivotCorCombined} that is defined on a probability space $\Omega_{n}$, constructed together with the integer $m(n)$, the subset $A_{n}\subseteq \Omega_{n}$, partition $\mathcal{P}_{n}$ of $A_{n}$ and the RVs $(w_{i}, v_{i}, \beta_{i}, \gamma_{i})_{i}$.

The idea is to construct the product space $\Omega_{n} \times \dot{\Omega}_{n}$ with the product measure $\Prob_{n}$, where $\dot{\Omega}_{n}$ carries an i.i.d. copy $\dot{Z}_{n}$ of $Z_{n}$, as well as the corresponding notions $\dot{A}_{n}$, $\dot{\mathcal{P}}_{n}$ and $(\dot{w}_{i}, \dot{v}_{i}, \dot{\beta}_{i}, \dot{\gamma}_{i})_{i}$. If $\frac{1}{\sqrt{n}}d(o, Z_{n} o) -c_{n}$ converges in law for some sequence $(c_{n})_{n>0}$, the difference with its i.i.d. copy \[
\frac{1}{\sqrt{n}}[d(o, Z_{n} o)  - d(o, \dot{Z}_{n}o)]
\]
also converges in law. We will deduce a contradiction to this convergence.

On $A_{n} \subseteq \Omega_{n}$, we define $W_{i}$'s and $V_{i}$'s as in the setting of Lemma \ref{lem:alignedGromov}: $W_{0} := w_{0}$ and \[
V_{k} := W_{k} \Pi(\beta_{k+1}) v_{k+1}, \,\, W_{k+1} := V_{k} \Pi(\gamma_{k+1}) w_{k+1} \quad (k=0, \ldots, m(n) - 1).
\]
We also define \[
(x_{2k-1}, x_{2k}) := (W_{k-1} o, V_{k-1} \Pi(\gamma_{k}) o)
\]
for $k=1, \ldots, m(n)$ and let $x_{0} := o$, $x_{2m(n)+1} := W_{m(n)} o = Z_{n} o$.

Let $\mathcal{E} \in \mathcal{P}_{n}$ be an arbitrary equivalence class. By Item (2) of the conclusion of Corollary \ref{cor:pivotCorCombined}, conditioned on $\mathcal{E}$, $(w_{i})_{i}$ are fixed and $(\beta_{i}, \gamma_{i}, v_{i}))_{i}$ are i.i.d.s distributed according to  $(\textrm{uniform measure on $S^{2}$})\times \mu$ conditioned on $\tilde{S}$. Moreover, by Item (3) of the conclusion of Corollary \ref{cor:pivotCorCombined},
\[
(o, W_{0} \Gamma(\beta_{1}), V_{0} \Gamma(\gamma_{1}), \ldots, W_{m(n)-1} \Gamma(\beta_{m(n)}), V_{m(n)-1} \Gamma(\gamma_{m(n)}), W_{m(n)} o)
\]
is $D_{0}$-semi-aligned on $\mathcal{E}$. We then apply Lemma \ref{lem:alignedGromov} to the equality \[\begin{aligned}
d(o, Z_{n} o) &= \underbrace{\sum_{i=1}^{m(n) + 1} d(x_{2i-2}, x_{2i-1})}_{I_{1}} + \underbrace{\sum_{i=1}^{m(n)} d(x_{2i-1}, x_{2i})}_{I_{2}}    \\
&- 2\underbrace{\sum_{l=0}^{\log_{2} m(n)} \sum_{k=1}^{m(n) /2^{l}}(x_{2^{l} (2k-2)}, x_{2^{l} \cdot 2k})_{x_{2^{l} (2k-1)}}}_{I_{3}} - 2\underbrace{\cdot(x_{0}, x_{2 \cdot m(n)+1})_{x_{2m(n)}}}_{I_{4}}.
\end{aligned}
\]
Let $J_{l} := \sum_{k=1}^{m(n)/2^{l}} (x_{2^{l} (2k-2)}, x_{2^{l} \cdot 2k} )_{x_{2^{l} (2k-1)}}$ be the $(l-1)$-th summand of $I_{3}$. Conditioned on $\mathcal{E}$, $J_{l}$ is the sum of $m(n)/2^{l}$ independent RVs bounded by $E_{0}$. By Chebyshev's inequality, we have \[\begin{aligned}
\Prob_{n} \left( \Big|J_{l} - \E[J_{l} | \mathcal{E} ] \Big| \ge 2^{-l/4} \cdot 100E_{0} \sqrt{m(n)}  \, \, \Big| \, \mathcal{E}\right) &\le \frac{1}{\left(2^{-l/4} \cdot 100E_{0}\cdot \sqrt{m(n)}\right)^{2}} Var\left(J_{l} \, \big| \, \mathcal{E} \right) \\
&\le\frac{2^{l/2}}{ 10000E_{0}^{2}m(n)} \cdot \frac{m(n)}{2^{l}} \cdot E_{0}^{2} \le 10^{-4} \cdot 2^{-l/2}.
\end{aligned}
\]
By summing this probability, we have \begin{equation}\label{eqn:I3Prob}
\Prob_{n}\left(\big|I_{3} - \E[I_{3} \, | \, \mathcal{E}]\big|> 800 E_{0} \cdot \sqrt{m(n)} \,\, \Big| \,A_{n}\right) \le 1/2000.
\end{equation}
Note also $I_{1}$ is constant on each $\mathcal{E} \in \mathcal{P}_{n}$ and $I_{4}$ is bounded by $E_{0}$ on $A_{n}$.

Let us now bring an arbitrary equivalence class $\dot{\mathcal{E}}$ from the partition $\dot{\mathcal{P}}_{n}$. Similarly, we get \[
\begin{aligned}
d(o, \dot{Z}_{n} o) &= \underbrace{\sum_{i=1}^{2^{s} + 1} d(\dot{x}_{2i-2}, \dot{x}_{2i-1})}_{\dot{I}_{1}} + \underbrace{\sum_{i=1}^{2^{s}} d(\dot{x}_{2i-1}, \dot{x}_{2i})}_{I_{2}}    \\
&- 2\underbrace{\sum_{l=0}^{s} \sum_{k=1}^{2^{s-l}}(\dot{x}_{2^{l} (2k-2)}, \dot{x}_{2^{l} \cdot 2k})_{\dot{x}_{2^{l} (2k-1)}}}_{\dot{I}_{3}} - 2\underbrace{\cdot(\dot{x}_{0}, \dot{x}_{2 \cdot 2^{s}+1})_{\dot{x}_{2\cdot2^{s}}}}_{\dot{I}_{4}}.
\end{aligned}
\]
We have the following ($\ast$):
\begin{enumerate}
\item Since the situation is symmetric, for large enough $n$ we have \[\begin{aligned}
\Prob_{n}\left( \bigcup \left\{(\mathcal{E}, \dot{\mathcal{E}}) \in \mathcal{P}_{n} \times \dot{\mathcal{P}}_{n} : I_{1} + \E[I_{3}\, |\, \mathcal{E}] \ge \dot{I}_{1} + \E[\dot{I}_{3} \, | \, \dot{\mathcal{E}}] \right\}\right) &\ge \frac{1}{2} \Prob_{n}(A_{n} \times \dot{A}_{n}) \\
&\ge 0.5(1 - \Prob_{n}(A_{n}) - \Prob_{n}(\dot{A}_{n})) \ge  0.4.
\end{aligned}
\]
\item By Inequality \ref{eqn:I3Prob}, we have \[
\Prob_{n}\left( \max \bigg( \Big|I_{3} - \E[I_{3} \, | \, \mathcal{E}]\Big|,  \Big|\dot{I}_{3} - \E[\dot{I}_{3} \, | \, \mathcal{\dot{E}}] \Big| \bigg) \ge 800 E_{0} \cdot \sqrt{m(n)} \, \bigg| \, \mathcal{E} \times \dot{\mathcal{E}}\right) \le 1/1000
\]
for each $\mathcal{E} \in \mathcal{P}_{n}$ and $\dot{\mathcal{E}} \in \dot{\mathcal{P}}_{n}$ for all $n$.
\item Since $I_{2}-\dot{I}_{2}$ is a sum of $m(n)$ i.i.d.s of \emph{symmetric} distribution with infinite second moment, for any $K'>0$ we have \[
\Prob_{n}\left(I_{2} - \dot{I}_{2} \ge K' \sqrt{m(n)}\, \Big| \, \mathcal{E} \times \dot{\mathcal{E}} \right) \ge 1/5
\]
for each $\mathcal{E} \in \mathcal{P}$ and $\dot{\mathcal{E}} \in \dot{\mathcal{P}}$ for sufficiently large $n$.
\item $I_{4}$ and $\dot{I}_{4}$ are bounded by $E_{0}$ on $A_{n} \times \dot{A}_{n}$.
\end{enumerate}
Let $C := \lim_{n} m(n)/n > 0$. The 4 items in $(\ast$) imply that for each $K'>10000E_{0}$, for large enough $n$, we have \[
d(o, Z_{n} o) - d(o, \dot{Z}_{n} o) \ge K' \sqrt{m(n)}  - 1600 E_{0} \sqrt{m(n)} - 2E_{0} \ge 0.25 K' \sqrt{C n}
\] with probability at least $0.4 \cdot (1/5 - 1/1000) \ge 0.05$. Hence we have \[
\liminf_{n} \Prob_{n}\left( \frac{1}{\sqrt{n}} \big(d(o, Z_{n} o) - d(o, \dot{Z}_{n} o) \big) \ge 0.25 K' \sqrt{C} \right) \ge 0.05 \quad (\forall \, K' > 10000E_{0}).
\]
This contradicts to the convergence of $\frac{1}{\sqrt{n}}[d(o, Z_{n} o) - d(o, \dot{Z}_{n} o)]$ in law.

We now derive a contradiction to the convergence in law of $\frac{1}{\sqrt{n}} [\tau(Z_{n}) - \tau(\dot{Z}_{n})]$. We again prepare the outputs of Corollary \ref{cor:pivotCorCombined}, namely, $m(n), \Omega_{n}, A_{n}$, $\mathcal{P}_{n}$, $(w_{i})_{i=0}^{m(n)}$,  $(\beta_{i}, \gamma_{i}, v_{i})_{i=1}^{m(n)}$, and their independent copies. We define $(W_{i}, V_{i})_{i=0}^{m(n)-1}$ and $(x_{i})_{i=1}^{2m(n)}$ as before, and let $x_{0} := w_{m(n)}^{-1} o$.

This time, Lemma \ref{lem:alignedGromovTrans} tells us that \[
\begin{aligned}
\tau(Z_{n}) &= d(x_{0}, x_{2n}) + [\tau(Z_{n}) - d(x_{0}, x_{2n}) ] \\
&= \underbrace{\sum_{i=1}^{m(n) + 1} d(x_{2i-2}, x_{2i-1})}_{I_{1}} + \underbrace{\sum_{i=1}^{m(n)} d(x_{2i-1}, x_{2i})}_{I_{2}}    \\
&- 2\underbrace{\sum_{l=0}^{\log_{2} m(n)} \sum_{k=1}^{m(n) /2^{l}}(x_{2^{l} (2k-2)}, x_{2^{l} \cdot 2k})_{x_{2^{l} (2k-1)}}}_{I_{3}} +\underbrace{\big( \tau(Z_{n}) - d(x_{0}, x_{2m(n)} \big)}_{I_{4}}.
\end{aligned}
\]
Here, $I_{1}, I_{2}, I_{3}$ and $I_{4}$ satisfy the exact same properties as before. Hence, we observe the same consequences as in $(\ast$). This leads to the same contradiction to the convergence in law of  $\frac{1}{\sqrt{n}} [\tau(Z_{n}) - \tau(\dot{Z}_{n})]$. Hence, for any choice of the sequence $(c_{n})_{n>0}$,  $\frac{1}{\sqrt{n}} (\tau(Z_{n}) - c_{n})$ cannot converge in law.
\end{proof}

\subsection{Pivoting and repulsion among independent random walks} \label{subsection:pivotIndep}

In this subsection, we consider two random walks $(Z_{n}^{(1)})_{n}$ and $(Z_{n}^{(2)})_{n}$ generated by non-elementary measures $\mu^{(1)}$ and $\mu^{(2)}$ on $G$. Let $S^{(1)}$ and $S^{(2)}$ be fairly long $K_{0}$-Schottky sets for $\mu^{(1)}$ and $\mu^{(2)}$. We define the subset $\tilde{S}^{(t)}$ following Definition \ref{dfn:pivotingChoices}, with $S^{(t)}$ in place of $S$.

We first fix $D_{0}$-semi-aligned sequences $\big(w_{i}^{(t)}\big)_{i=0}^{n}$ in $G$ for $t =0, 1$. Given the choices of $\big(\beta_{i}^{(t)}, \gamma_{i}^{(t)}, v_{i}^{(t)}\big) \in \tilde{S}^{(t)}$ for $i=1, \ldots, n$, we define $W_{i}^{(t)}$'s, $V_{i}^{(t)}$'s as before: $W_{0}^{(t)} := w_{0}^{(t)}$ and \[
V_{k}^{(t)} := W_{k}^{(t)} \Pi\big(\beta_{k+1}^{(t)}\big) v_{k+1}^{(t)}, \quad W_{k+1}^{(t)} := V_{k}^{(t)} \Pi\big(\gamma_{k+1}^{(t)}\big) w_{k+1}^{(t)}.
\]

This time, we want not only that $W_{n}^{(t)}$ and $(W_{n}^{(t)})^{-1}$ head in different directions for $t=1, 2$, but also that the directions made by $W_{n}^{(1)}$, $W_{n}^{(2)}$, $(W_{n}^{(1)})^{-1}$, $(W_{n}^{(2)})^{-1}$ are all distinct. For this purpose, we define the following: \begin{enumerate}
\item (front-front repulsion) For each $1 \le k \le n/2$ we let \[\begin{aligned}
\phi_{k}^{(1, 2), front} &:= \big(W_{k-1}^{(2)} \big)^{-1} \cdot W_{k-1}^{(1)} \\
&= \big(w_{k-1}^{(2)}\big)^{-1} \big(\Pi\big(\gamma_{k-1}^{(2)}\big) \big)^{-1} \cdots \big(w_{0}^{(2)}\big)^{-1} \cdot w_{0}^{(1)} \Pi(\beta_{1}^{(1)}) \cdots w_{k-1}^{(1)}, \\
\tilde{S}_{k}^{(1, 2), front} &:= \left\{ (\beta_{k}^{(1)}, \gamma_{k}^{(1)}) \in \tilde{S}^{(1)}(v_{k}^{(1)}) : \textrm{$\left(\big(\phi_{k}^{(1, 2), front} \big)^{-1} o, \, \Gamma\big(\beta_{k}^{(1)}\big) \right)$ is $K_{0}$-aligned} \right\},\\
\tilde{S}_{k}^{(2, 1), front} &:= \left\{ (\beta_{k}^{(2)}, \gamma_{k}^{(2)}) \in \tilde{S}^{(2)}(v_{k}^{(1)}) : \textrm{$\left(\phi_{k}^{(1, 2), front} \Pi\big(\beta_{k}^{(1)}\big)o, \, \Gamma\big(\beta_{k}^{(2)}\big) \right)$ is $K_{0}$-aligned} \right\}.
\end{aligned}
\]
\item (back-back repulsion) For each $1 \le k \le n/2$ we let \[\begin{aligned}
\phi_{k}^{(1, 2), back} &:= \Big(V_{n-k}^{(2)} \Pi\big(\gamma_{n-k+1}^{(2)}\big)\Big)^{-1} W_{n}^{(2)} \cdot (W_{n}^{(1)})^{-1}V_{n-k}^{(1)} \Pi\big(\gamma_{n-k+1}^{(1)}\big) \\
& = w_{n-k+1}^{(2)} \Pi(\beta_{n-k+2}^{(2)}) \cdots w_{n}^{(2)} \cdot \big(w_{n}^{(1)}\big)^{-1} \big(\Pi(\gamma_{n}^{(1)})\big)^{-1} \cdots \big(w_{n-k+1}^{(1)}\big)^{-1},\\
\tilde{S}_{k}^{(1, 2), back} &:= \left\{ \begin{array}{c} (\beta_{n-k+1}^{(1)}, \gamma_{n-k+1}^{(1)}) \in \tilde{S}^{(1)}(v_{n-k+1}^{(1)}) : \\
 \textrm{$\left( \Gamma\big(\gamma_{n-k+1}^{(1)}\big),\, \Pi(\gamma_{n-k+1}^{(1)}) \cdot \big(\phi_{k}^{(1, 2), back}\big)^{-1} o\right)$ is $K_{0}$-aligned} \end{array} \right\},\\
\tilde{S}_{k}^{(2, 1), back} &:= \left\{ \begin{array}{c}  (\beta_{n-k+1}^{(2)}, \gamma_{n-k+1}^{(2)}) \in \tilde{S}^{(2)}(v_{n-k+1}^{(2)}) : \\\textrm{$\left(\Gamma\big(\gamma_{n-k+1}^{(2)}\big), \,\Pi\big(\gamma_{n-k+1}^{(2)} \big) \phi_{k}^{(1, 2), back} \left(\Pi\big(\gamma_{n-k+1}^{(1)}\big) \right)^{-1} \cdot o \right)$ is $K_{0}$-aligned}\end{array}\right\}.
\end{aligned}
\]
\item (front-back repulsion) For each $s, t \in \{1, 2\}$ and $1 \le k \le n/2$ we let \[\begin{aligned}
\phi_{k}^{(s, t), mixed} &:= \left( V_{n-k}^{(t)} \Pi\big(\gamma_{n-k+1}^{(t)} \big) \right)^{-1} W_{n}^{(t)} \cdot W_{k-1}^{(s)} \\
&= w_{n-k+1}^{(t)} \Pi\big(\beta_{n-k+2}^{(t)}\big) \cdots w_{n}^{(t)} \cdot w_{0}^{(s)} \Pi\big(\beta_{1}^{(s)}\big) \dots w_{k-1}^{(s)},\\
\tilde{S}_{k}^{(s, t), \rightarrow} &:= \left\{ \begin{array}{c} (\beta_{k}^{(s)}, \gamma_{k}^{(s)}) \in \tilde{S}^{(s)}(v_{k}^{(s)}) : \\
 \textrm{$\left( \big(\phi_{k}^{(s, t), mixed}\big)^{-1} o,\, \Gamma\big(\beta_{k}^{(s)} \big) \right)$ is $K_{0}$-aligned} \end{array} \right\},\\
\tilde{S}_{k}^{(s, t), \leftarrow} &:= \left\{ \begin{array}{c}  (\beta_{n-k+1}^{(t)}, \gamma_{n-k+1}^{(t)}) \in \tilde{S}^{(t)}(v_{n-k+1}^{(t)}) : \\\textrm{$\left(\Gamma\big(\gamma_{n-k+1}^{(t)}\big), \,\Pi\big(\gamma_{n-k+1}^{(t)} \big) \phi_{k}^{(s, t), mixed} \Pi\big(\beta_{k}^{(s)}\big)  \cdot o \right)$ is $K_{0}$-aligned}\end{array}\right\}.
\end{aligned}
\]
\end{enumerate}

We now establish an analogue of Lemma \ref{lem:selfRepulse}.

\begin{lem}\label{lem:multiRepulse}
Let $1 \le k\le  n$. Let $\big(w_{i}^{(0)}\big)_{i=0}^{n}$ and $\big(w_{i}^{(1)}\big)_{i=0}^{n}$ be $D_{0}$-pre-aligned sequences in $G$. Let $\big(\beta_{i}^{(t)}, \gamma_{i}^{(t)}, v_{i}^{(t)}\big)_{i=1, \ldots, k, n-k+1, \ldots, n} \in \big(\tilde{S}^{(t)} \big)^{2k}$ for $t=1, 2$ be choices such that \begin{equation}\label{eqn:multiCondition}\begin{aligned}
(\beta_{k}^{(1)}, \gamma_{k}^{(1)}) &\in \tilde{S}_{k}^{(1, 2), front}\cap \tilde{S}_{k}^{(1, 1), \rightarrow} \cap \tilde{S}_{k}^{(1, 2), \rightarrow}, \\
(\beta_{k}^{(2)}, \gamma_{k}^{(2)}) &\in \tilde{S}_{k}^{(2,1), front}\cap \tilde{S}_{k}^{(2, 1), \rightarrow} \cap \tilde{S}_{k}^{(2, 2), \rightarrow}, \\
(\beta_{n-k+1}^{(1)}, \gamma_{n-k+1}^{(1)}) &\in \tilde{S}_{k}^{(1, 2), back}\cap \tilde{S}_{k}^{(1, 1), \leftarrow} \cap \tilde{S}_{k}^{(2, 1), \leftarrow},\\
(\beta_{n-k+1}^{(2)}, \gamma_{n-k+1}^{(2)}) &\in \tilde{S}_{k}^{(1, 2), back}\cap \tilde{S}_{k}^{(2, 1), \leftarrow} \cap \tilde{S}_{k}^{(2, 2), \leftarrow}.
\end{aligned}
\end{equation}
Then for any further choices $\big(\beta_{i}^{(t)}, \gamma_{i}^{(t)}, v_{i}^{(t)}\big)_{i=k+1, \ldots, n-k} \in \big(\tilde{S}^{(t)} \big)^{2(n-k)}$ for $t=1, 2$, $W_{n}^{(1)}$ and $W_{n}^{(2)}$ are contracting isometries that generate a free group of order 2. Moreover, the orbit map is a quasi-isometric embedding of $\langle W_{n}^{(1)}, W_{n}^{(2)} \rangle$ into a quasi-convex subset of $X$.
\end{lem}

\begin{proof}
Recall the following notation: given a path $\Gamma = (x_{0}, x_{1} \ldots, x_{m})$, we defined its \emph{reversal} by  \[
\bar{\Gamma} := (x_{m}, x_{m-1}, \ldots, x_{0}).
\]
For convenience, let us define $\kappa^{(t)} := W_{k-1}^{(t)} \Gamma(\beta_{k}^{(t)})$ and let $\bar{\kappa}^{(t)}$ be its reversal for $s = 1, 2$. Also, let $\eta^{(t)} := (W_{n}^{(t)})^{-1} V_{n-k}\Gamma(\gamma_{n-k+1}^{(t)})$ and let $\bar{\eta}^{(t)}$ be its reversal.

By Lemma \ref{lem:1segment}, the assumption of the lemma implies that the following sequences of Schottky axes are $K_{0}$-aligned: \[\begin{aligned}
(\bar{\kappa}_{(1)}, \kappa_{(2)}), \quad (\eta_{(1)}, \bar{\eta}_{(1)}), \quad (\eta_{(1)}, \kappa_{(1)}), \\
 (\eta_{(2)}, \kappa_{(1)}),\quad (\eta_{(1)}, \kappa_{(2)}),\quad (\eta_{(2)}, \kappa_{(2)}).
 \end{aligned}
\]
Moreover, since $(w_{i}^{(t)})_{i=0}^{n}$'s are $D_{0}$-pre-aligned, we have that \[
(o, \kappa_{(t)}, W_{n}^{(t)} \eta_{(t)}, W_{n}^{(t)} o) \quad (t\in \{1, 2\})
\]is $D_{0}$-semi-aligned.

We now prove that the orbit map from $\langle W_{n}^{(1)}, W_{n}^{(2)}\rangle$ to $X$ is bi-Lipshictz and the orbit set is quasiconvex. It is clear that the map is $\max(d(o, W_{n}^{(1)} o), d(o, W_{n}^{(2)} o))$-Lipschitz. We define \[
\begin{aligned}
K_{1} &:=  \min \{ \diam(\kappa_{(1)} \cup W_{n}^{(1)}\eta_{(1)}), \diam(\kappa_{(2)} \cup W_{n}^{(2)} \eta_{(2)})\} - E_{0}, \\
K_{2}& := 4\sum_{t \in \{1, 2\}} \Big(\diam(\eta_{(t)}, o) + \diam(\kappa_{(t)} o)  + \diam(\kappa_{(t)} \cup W_{n}^{(t)} \eta_{(t)})  + E_{0} \big).
\end{aligned}
\]
Since $\kappa_{(1)}$ and $\kappa_{(2)}$ are both fairly long Schottky axes, we have $\diam(\kappa_{(1)}), \diam(\kappa_{(2)}) \ge 10E_{0}$ and consequently $K_{1}$ is positive.

The conclusion is established once we show that \begin{equation}\begin{aligned}\label{eqn:lowerBdQIMu}
d(o, a_{1} \cdots a_{m} o) &\ge 
m K_{1}, \\
[o, a_{1} \cdots a_{m} o] &\subseteq N_{K_{2}}\left( \{o, a_{1}o, a_{1}a_{2} o, \ldots, a_{1} \cdots a_{m} o\}\right)
\end{aligned}
\end{equation}
for each $m$ and each choice of $a_{1}, \ldots, a_{m} \in \{W_{n}^{(1)}, W_{n}^{(2)}, (W_{n}^{(1)})^{-1}, (W_{n}^{(2)})^{-1}\}$ such that $a_{i} \neq a_{i+1}^{-1}$ for each $i$. Let us show this for the choice $a_{1}a_{2}a_{3} = W_{n}^{(1)} (W_{n}^{(2)})^{-1} W_{n}^{(1)}$ as an example; the same argument applies to all other cases. Combining the alignment mentioned above, we deduce that \[
\Big(\begin{array}{c} o, \kappa_{(1)}, W_{n}^{(1)} \eta_{(1)}, W_{n}^{(1)} \bar{\eta}_{(2)}, W_{n}^{(1)} (W_{n}^{(2)})^{-1}\bar{\kappa}_{(2)}, W_{n}^{(1)} (W_{n}^{(2)})^{-1}\kappa_{(1)}, \\W_{n}^{(1)} (W_{n}^{(2)})^{-1}W_{n}^{(1)} \eta_{(1)}, W_{n}^{(1)} (W_{n}^{(2)})^{-1}W_{n}^{(1)} o \end{array}\Big)
\]
is $D_{0}$-semi-aligned. By Lemma \ref{lem:BGIPWitness}, there exist disjoint subsegments $[x_{1}, y_{1}]$, $\ldots$, $[x_{6}, y_{6}]$ of $[o, W_{n}^{(1)} (W_{n}^{(2)})^{-1} W_{n}^{(1)} o]$, in order from closest to farthest from $o$, that $0.1E_{0}$-fellow travel with $\kappa_{(1)}$, $\ldots$, $W_{n}^{(1)} (W_{n}^{(2)})^{-1}W_{n}^{(1)} \eta_{(1)}$, respectively. In particular, we deduce that \[\begin{aligned}
d(x_{1}, y_{2}) &\ge \diam(\kappa_{(1)}\cup  W_{n}^{(1)} \eta_{(1)}) - E_{0}, \\
d(x_{3}, y_{4}) &\ge  \diam(W_{n}^{(1)} \bar{\eta}_{(2)} \cup  W_{n}^{(1)}(W_{n}^{(2)})^{-1}  \bar{\kappa}_{(2)}) - E_{0} \\
&= \diam(\kappa_{(2)}\cup  W_{n}^{(2)} \eta_{(2)}) - E_{0}, \\
d(x_{5}, y_{6}) &\ge  \diam(W_{n}^{(1)} (W_{n}^{(2)})^{-1} \kappa_{(1)} \cup  W_{n}^{(1)}(W_{n}^{(2)})^{-1}  W_{n}^{(1)} \eta_{(1))}) - E_{0} \\
&= \diam(\kappa_{(1)}\cup  W_{n}^{(1)} \eta_{(1)}) - E_{0}.
\end{aligned}
\]
Since $d(o,  W_{n}^{(1)} (W_{n}^{(2)})^{-1}W_{n}^{(1)} o)$ is greater than $d(x_{1}, y_{2}) + d(x_{3}, y_{4}) + d(x_{5}, y_{6})$, we can conclude the first inequality in Display \ref{eqn:lowerBdQIMu}. Moreover, note that \[
\begin{aligned}
d(x_{2}, W_{n}^{(1)} o) \le \diam (\eta_{(1)} o \cup o) + E_{0},\\
d(x_{4}, W_{n}^{(1)} ( W_{n}^{(2)})^{-1} o) \le \diam (\kappa_{(1)} o \cup o) + E_{0}.
\end{aligned}
\]
Now for each point $y$ on $[o, x_{2}]$ we have \[
\begin{aligned}
d(y, W_{n}^{(1)} o) &\le d(o, x_{2}) + d(x_{2}, W_{n}^{(1)} o)\\
& \le \Big(\diam(o \cup \kappa_{(1)}) + \diam(\kappa_{(1)} \cup W_{n}^{(1)} \eta_{(1)}) +E_{0} \Big)+\Big(\diam (\eta_{(1)} o \cup o) + E_{0}\Big) \le K_{2}.
\end{aligned}
\]
Similarly, for any point $y$ on $[x_{2}, x_{4}]$ we have \[
\begin{aligned}
d(y, W_{n}^{(1)}(W_{n}^{(2)})^{-1} o) &\le d(x_{2}, x_{4}) + d(x_{4}, W_{n}^{(1)} (W_{n}^{(2)})^{-1} o)\\
& \le \diam(o\cup  \eta_{(1)}) + \diam(o \cup \eta_{(2)}) + \diam(\kappa_{(2)} \cup W_{n}^{(2)} \eta_{(2)}) + 2E_{0} \\
& +\diam (\eta_{(2)} o \cup o) + E_{0} \le K_{2}.
\end{aligned}
\]
For a similar reason, $[x_{4}, W_{n}^{(1)} (W_{n}^{(2)})^{-1} W_{n}^{(1)} o]$ is contained in the $K_{2}$-neighborhood of $W_{n}^{(1)} (W_{n}^{(2)})^{-1} W_{n}^{(1)} o$. This settles the second line of Display \ref{eqn:lowerBdQIMu}.
\end{proof}

We now establish an analogue of Lemma \ref{lem:trLengthPivotProb}.

\begin{lem}\label{lem:multiPivotProb}
Let $1 \le k \le n/2$, let $(w_{i}^{(t)})_{i=1}^{n} \in G^{n+1}$ and let $(\beta_{i}^{(t)}, \gamma_{i}^{(t)}, v_{i}^{(t)})_{i=1}^{n} \in \tilde{S}_{n}^{(t)}$ for $t = 1, 2$. Then Condition \ref{eqn:multiCondition} depends only on $\{(\beta_{l}^{(t)}, \gamma_{l}^{(t)}, v_{l}^{(t)}) : l = 1, \ldots, k, n-k+1, \ldots, n, t = 1, 2\}$.

Moreover, fixing any choice of $\{v_{l}^{(t)} : l=1, \ldots, k, n-k+1, \ldots, n, t=1, 2\}$, the condition holds for probability at least $1 - (10\#S^{(1)})^{k} - (10\#S^{(2)})^{k}$ with respect to the uniform measures on $\tilde{S}^{(t)}(v_{l})$ for $t=1, 2$ and for $l=1, \ldots, k, n-k+1, \ldots, n$, all independent.
\end{lem}

\begin{proof}
Note that the statement \begin{equation}\label{eqn:multiCondition1}
(\beta_{k}^{(1)}, \gamma_{k}^{(1)}) \in S_{l}^{(1, 2), front} \cap \tilde{S}_{k}^{(1, 1), \rightarrow} \cap \tilde{S}_{k}^{(1, 2), \rightarrow}
\end{equation} is really a condition for $\beta_{k}^{(1)}$ that involves $\phi_{k}^{(1, 2), front}$, $\phi_{k}^{(1, 1), mixed}$ and $\phi_{k}^{(1, 2), mixed}$, which depend only on $\beta_{l}^{(t)}$ and $\gamma_{l}^{(t)}$ for $l=1, \ldots, k-1, n-k+2, \ldots, n$ and for $t = 1, 2$. Fixing a choice of $(\beta_{k}^{(1)}, \gamma_{k}^{(1)})$ satisfying condition \ref{eqn:multiCondition1}, the statement \begin{equation}\label{eqn:multiCondition2}
(\beta_{k}^{(2)}, \gamma_{k}^{(2)}) \in S_{l}^{(2,1), front} \cap \tilde{S}_{k}^{(2,1), \rightarrow} \cap \tilde{S}_{k}^{(2, 2), \rightarrow}
\end{equation} is now a condition for $\beta_{k}^{(2)}$ that is determined only by $\beta_{l}^{(t)}$ and $\gamma_{l}^{(t)}$ for $l=1, \ldots, k-1, n-k+2, \ldots, n$ and for $t = 1, 2$, and $\beta_{k}^{(1)}$. Fixing a $(\beta_{k}^{(2)}, \gamma_{k}^{(2)})$ further, the condition \begin{equation}\label{eqn:multiCondition3}
(\beta_{n-k+1}^{(1)}, \gamma_{n-k+1}^{(1)}) \in \tilde{S}_{k}^{(1, 2), back}\cap \tilde{S}_{k}^{(1, 1), \leftarrow} \cap \tilde{S}_{k}^{(2, 1), \leftarrow}
\end{equation}
then depends on $\beta_{l}^{(t)}$ and $\gamma_{l}^{(t)}$ for $l=1, \ldots, k-1, n-k+2, \ldots, n$ and for $t = 1, 2$, and the choice of $\beta_{k}^{(1)}$, $\beta_{k}^{(2)}$. Fixing a $(\beta_{n-k+1}^{(1)}, \gamma_{n-k+1}^{(1)})$, finally, the condition 
\begin{equation}\label{eqn:multiCondition4}
(\beta_{n-k+1}^{(2)}, \gamma_{n-k+1}^{(2)}) \in \tilde{S}_{k}^{(1, 2), back}\cap \tilde{S}_{k}^{(2, 1), \leftarrow} \cap \tilde{S}_{k}^{(2, 2), \leftarrow}.
\end{equation}
 depends on $\beta_{l}^{(t)}$ and $\gamma_{l}^{(t)}$ for $l=1, \ldots, k-1, n-k+2, \ldots, n$, $t = 1, 2$, and $\beta_{k}^{(1)}$, $\beta_{k}^{(2)}$, $\gamma_{n-k+1}^{(1)}$. Combining these four statements leads to the first assertion.

The second assertion is proven analogously to the proof of Lemma \ref{lem:trLengthPivotProb}; we sketch the necessary ingredients here. Fixing an arbitrary choice of \[
\Big\{ \beta_{i}^{(t)}, \gamma_{i}^{(t)}, v_{i}^{(t)} : i= 1, \ldots, k-1, t=1,2\Big\},
\]
note that Condition \ref{eqn:multiCondition1} is asking the alignment of $\beta_{k}^{(1)}$ with 3 other fixed isometries. The property of the Schottky set $S^{(1)}$ (cf. Definition \ref{dfn:Schottky}) tells us that such alignments are realized by all but at most 3 choices of $\beta_{k}^{(1)} \in S^{(1)}$. This implies that at most $3\#S^{(1)}$ choices in $(S^{(1)})^{2} \supseteq \tilde{S}^{(1)}(v_{k}^{(1)})$ violate the condition. This implies that \[\begin{aligned}
&\Prob\left( (\beta_{k}^{(1)}, \gamma_{k}^{(1)}) \in \tilde{S}^{(1)}(v_{k}^{(1)})\,\, \textrm{violates Condition \ref{eqn:multiCondition1}} \, \Big|\, \beta_{i}^{(t)}, \gamma_{i}^{(t)}, v_{i}^{(t)} : i= 1, \ldots, k-1, t=1,2 \right) \\
&\le \frac{3\#S^{(1)}}{(\#S^{(1)})^{2} - 2\#S^{(1)}} \le \frac{5\#S^{(1)}}{(\#S^{(1)})^{2}}.
\end{aligned}
\]
For a similar reason, we can prove that the probabilities of violating Condition \ref{eqn:multiCondition2}, \ref{eqn:multiCondition3} and \ref{eqn:multiCondition4} are at most $5/\#S^{(2)}$, $5/\#S^{(1)}$ and $5/\#S^{(2)}$, respectively. Combining these, we can deduce \[
\Prob\left( \textrm{Condition \ref{eqn:multiCondition} violated at step $k$} \, \Big| \, \textrm{Condition \ref{eqn:multiCondition} violated till $k-1$} \right) \le \frac{10}{\#S^{(1)}} + \frac{10}{\#S^{(2)}}.
\]
By induction, we can conclude the second assertion.
\end{proof}

Recall that we have implemented pivotal times for random walks in Corollary \ref{cor:pivotCorCombined}. Following this idea, and also by combining Lemma \ref{lem:multiRepulse} and Lemma \ref{lem:trLengthPivotProb}, we deduce the following corollary: 

\begin{cor}\label{cor:multiRepulse}
Let $(X, G, o)$ be as in Convention \ref{conv:main} and $(Z_{n}^{(1)})_{n>0}, (Z_{n}^{(2)})_{n>0}$ be two independent random walks generated by a non-elementary measure $\mu$ on $G$. Then there exists $K > 0$ such that the following holds outside a set of probability $K e^{-n/K}$. 

The $n$-th step isometries $Z_{n}^{(1)}$, $Z_{n}^{(2)}$ arising from two random walks generate a free group of order 2. Moreover, the orbit map is a quasi-isometric embedding of $\langle Z_{n}^{(1)}, Z_{n}^{(2)} \rangle$ into a quasi-convex subset of $X$.
\end{cor}

It is not difficult to consider $k$ independent random walks; the arguments are identical. Hence, we conclude Theorem \ref{thm:qi}.

\section{Effective estimates for translation length} \label{section:effective}

In this section, we establish an effective version of Corollary \ref{cor:pivotCorCombined} that will be used for the counting problem. 

We will employ Schottky sets to construct the generating set $S''$. For this we introduce a notation. Given a set of sequences $S \subseteq G^{n}$, we define \[
\Phi(S^{4}):= \{\Pi(s_{1}) \Pi(s_{2})\Pi(s_{3})\Pi(s_{4}): s_{1}, s_{2}, s_{3}, s_{4} \in S\}.
\]
Applying the map $\Phi$ to $S^{4}$ does not erase any information because:

\begin{lem}\label{lem:SchottkyResult}
Let $S$ be a fairly long Schottky set. Then the map $\Phi$ is a bijection between $S^{4}$ and $\Phi(S^{4})$.
\end{lem}

\begin{proof}
This is a consequence of Lemma \ref{lem:concatSchottky} below.
\end{proof}

One technicality is that $\Phi(S^{4})$ cannot be a symmetric subset of $G$ for any Schottky set $S$. (This is due to our choice of Property (4) of Schottky sets.) To get around this technicality, we introduce another notation. Given a sequence $\alpha = (a_{1}, \ldots, a_{n}) \in G^{n}$, let us denote the sequence $(a_{n}^{-1}, \ldots, a_{1}^{-1})$ by $\alpha^{-1}$. 
For a set $S \subseteq G^{n}$, we define \[
\check{S} := \{\alpha^{-1} : \alpha \in S\} = \left\{ \big(a_{n}^{-1}, \ldots, a_{1}^{-1}\big) : \big(a_{1}, \ldots, a_{n}\big) \in S\right\}.
\]

\begin{lem}\label{lem:concatSchottky}
Let $k \in \mathbb{Z}_{>0}$ and let $S$ be a fairly long Schottky set. Let $\alpha_{i} \in S$ and $\epsilon_{i} \in \{\pm 1\}$ for $i = 1, \ldots, k$. Suppose that there does not exist $i$ such that $\alpha_{i} = \alpha_{i+1}$ and $\epsilon_{i}\epsilon_{i+1} = -1$ simultaneously holds. Then: \begin{enumerate}
\item the sequence of Schottky axes\[
\big(\Gamma(\alpha_{1}^{\epsilon_{1}}),\quad\Pi(\alpha_{1}^{\epsilon_{1}}) \Gamma(\alpha_{2}^{\epsilon_{2}}),\quad\ldots, \quad\Pi(\alpha_{1}^{\epsilon_{1}}) \cdots \Pi(\alpha_{k-1}^{\epsilon_{k}}) \Gamma(\alpha_{k}^{\epsilon_{k}}) \big)
\]is $D_{0}$-aligned, and 
\item $\Pi(\alpha_{1}^{\epsilon_{1}}) \cdots \Pi(\alpha_{k}^{\epsilon_{k}})$ is a nontrivial isometry.
\end{enumerate}
\end{lem}

\begin{proof}
Let $\alpha \in S$ and $\epsilon \in \{\pm 1\}$. Then \[
\diam\left(\pi_{\Gamma(\alpha^{\epsilon})}( \prodSeq(\alpha^{\epsilon})o) \cup o\right) = \diam\left(\prodSeq(\alpha^{\epsilon}) o \cup o\right) \ge 10E_{0} > K_{0}
\]
holds, which means that \[\begin{aligned}
\alpha &\notin \Big\{ \beta \in S : \textrm{$\big(\Pi(\alpha) o, \Gamma( \beta) \big)$ is $K_{0}$-aligned} \Big\},\\
\alpha&\notin \Big\{ \beta \in S : \textrm{$\big(\Gamma( \beta), \Pi(\beta) \Pi(\alpha^{-1}) o \big)$ is $K_{0}$-aligned} \Big\}.\\
\end{aligned}
\] By Property 4 of Schottky sets (cf. Definition \ref{dfn:Schottky}), we deduce that  $\big(\Gamma(\beta), \Pi(\beta) \Pi(\alpha^{\mathbf{\epsilon}}) o\big)$ and $\big(\Pi(\alpha^{\mathbf{\epsilon}} )o, \Gamma(\beta)\big)$ are $K_{0}$-aligned sequences for each $\beta \in S \setminus \{\alpha\}$. The latter statement means that  $\big(\Gamma(\beta^{-1}), \Pi(\beta^{-1}) \Pi(\alpha^{\mathbf{\epsilon}}) o\big)$ is $K_{0}$-aligned. ($\ast$)

Now for each $i$, we have the following cases. \begin{enumerate}[label=(\Roman*)]
\item $\alpha_{i} \neq \alpha_{i+1}$. In this case,  $\big(\Gamma(\alpha_{i}^{\epsilon_{i}}),   \Pi(\alpha_{i}^{\epsilon_{i}}) \Pi(\alpha_{i+1}^{\mathbf{\epsilon}_{i+1}}) o\big)$ is $K_{0}$-aligned due to ($\ast$). Moreover, $\big( \Pi(\alpha_{i}^{\epsilon_{i}}) o,   \Pi(\alpha_{i}^{\epsilon_{i}}) \Gamma(\alpha_{i+1}^{\mathbf{\epsilon}_{i+1}})  \big)$ is $0$-aligned. This is because $\Pi(\alpha_{i}^{\epsilon_{i}}) o$ projects onto $\Pi(\alpha_{i}^{\epsilon_{i}}) \Gamma(\alpha_{i+1}^{\mathbf{\epsilon}_{i+1}}) \ni \Pi(\alpha_{i}^{\epsilon_{i}}) o$ at itself.

Now by Lemma \ref{lem:1segment},$\big(\Gamma(\alpha_{i}^{\epsilon_{i}}),   \Pi(\alpha_{i}^{\epsilon_{i}}) \Gamma(\alpha_{i+1}^{\mathbf{\epsilon}_{i+1}}) \big)$ is $D_{0}$-aligned.
\item $\alpha_{i} = \alpha_{i+1}$: then $\epsilon_{i} = \epsilon_{i+1}$, and the sequence \[
(\Gamma(\alpha_{i}^{\epsilon_{i}}), \Pi(\alpha_{i}^{\epsilon_{i}}) \Gamma(\alpha_{i}^{\epsilon_{i}}))
\]
is $D_{0}$-aligned thanks to Property 5 of Schottky sets.
\end{enumerate}
Combining these stepwise $D_{0}$-alignment, we conclude Item (1).

For convenience, we denote $\Pi(\alpha_{1}^{\epsilon_{1}}) \cdots \Pi(\alpha_{k}^{\epsilon_{k}})$ by $w$. By Lemma \ref{lem:BGIPWitness}, $[o, wo]$ contains disjoint subsegments that $0.1E_{0}$-fellow travel with suitable translates of $\Gamma(\alpha_{1}^{\epsilon_{1}})$, $\ldots$, $\Gamma(\alpha_{k}^{\epsilon_{k}})$, respectively. Recall that these Schottky axes have diameter at least $10E_{0}$. Hence we have  \[
d(o, wo) \ge \left[\sum_{i=1}^{k} d\left(o, \Pi(\alpha_{k}^{\epsilon_{k}}) o\right) \right] - 0.2(k-1) E_{0} \ge 5E_{0} k. 
\]
In particular, $w$ is nontrivial and Item (2) follows.
\end{proof}

\begin{lem}\label{lem:triviaSchottkyInv}
Let $S \subseteq G^{n}$ be a set of sequences in $G$ and let $K_{0}>0$. Then the following hold.
\begin{enumerate}
\item $\Phi(\check{S}^{4})$ consists of the inverses of elements of $\Phi(S^{4})$.
\item $S$ is a fairly long $K_{0}$-Schottky set if and only if $\check{S}$ is.
\item If $S$ is a fairly long Schottky set, then $\Phi(\check{S}^{4})$ and $\Phi(S^{4})$ are disjoint.
\end{enumerate}
\end{lem}

\begin{proof}
\begin{enumerate}
\item This follows from the definition of $\check{S}$ and $\check{S}^{4}$. 
\item Let $S$ be a fairly long $K_{0}$-Schottky set. Note the following equivalence for each $\alpha \in G^{n}$ and $x \in X$: \[\begin{aligned}
\textrm{$\Gamma(\alpha)$ is $K_{0}$-contracting} &\Leftrightarrow 
\textrm{$\Gamma(\alpha^{-1})$ is $K_{0}$-contracting},\\
d(o, \Pi(\alpha)o) > 10E_{0} &\Leftrightarrow d(o, \Pi(\alpha^{-1})o) > 10E_{0},\\
\textrm{$\big(x, \Gamma(\alpha)\big)$, $\big(\Gamma(\alpha), \Pi(\alpha) x\big)$ are $K_{0}$-aligned} &\Leftrightarrow 
\textrm{$\big(x, \Gamma(\alpha^{-1})\big)$, $\big(\Gamma(\alpha^{-1}), \Pi(\alpha^{-1}) x\big)$ are $K_{0}$-aligned}  \\
\textrm{$\big(\Gamma(\alpha), \Pi(\alpha) \Gamma(\alpha)\big)$ is $K_{0}$-aligned} &\Leftrightarrow 
\textrm{$\big(\Gamma(\alpha^{-1}), \Pi(\alpha^{-1}) \Gamma(\alpha^{-1})\big)$ is $K_{0}$-aligned}.
\end{aligned}
\]
This implies that changing an element $\alpha \in S$ with $\alpha^{-1}$ does not affect the Schottky-ness of $S$. By converting all the elements into their inverses, we conclude that $\check{S}$ is also $K_{0}$-Schottky.
\item If $\mathbf{s} = (\alpha_{1}, \alpha_{2}, \alpha_{3}, \alpha_{4}) \in S^{4}$ and $\mathbf{s}' = (\check{\alpha}_{1}, \check{\alpha}_{2}, \check{\alpha}_{3}, \check{\alpha}_{4}) \in \check{S}^{4}$ have the same $\Phi$-values, then we have \begin{equation}\label{eqn:triviaSchottkyInv}
id = \Phi(\mathbf{s}) \cdot \Phi(\mathbf{s}')^{-1} = \Pi(\alpha_{1}) \cdots \Pi(\alpha_{4}) \cdot \Pi(\check{\alpha}_{4}^{-1}) \cdots \Pi(\check{\alpha}_{1}^{-1}).
\end{equation}
Here, $\alpha_{i}$'s and $\check{\alpha}_{i}^{-1}$'s are all elements of $S$. With this, Equation \ref{eqn:triviaSchottkyInv} contradicts Lemma \ref{lem:concatSchottky}. \qedhere
\end{enumerate}
\end{proof}

We are now ready to state and proved modified versions of Lemma \ref{lem:pivotFirstReduce} and Corollary \ref{cor:pivotCorCombined}.

\begin{lem}\label{lem:pivotFirstReduceCount}
For each $0 < q < p < 1$ there exists a constant $\epsilon = \epsilon(p, q) > 0$ such that the following holds.

Let $S\subseteq G^{M_{0}}$ be a fairly long $K_{0}$-Schottky set and let $\mu$ be a probability measure on $G$ such that \[
\mu \ge p\cdot  \big(\textrm{uniform measure on $\Phi(S^{4}) \cup \Phi(\check{S}^{(4)})$}\big).
\]
Then for each $n$ there exist a probability space $\Omega_{n}$, a measurable subset $B_{n} \subseteq \Omega_{n}$, a measurable partition $\mathcal{Q}_{n}$ of $B_{n}$ and RVs \[\begin{aligned}
\{w_{i} : i=0, \ldots, m(n)\} &\in G^{m(n) + 1}, \\
\mathcal{S} &\in \{S, \tilde{S}\}, \quad \quad \big(m(n) := \lfloor \epsilon n \rfloor\big)\\
\{\alpha_{i}, \beta_{i}, \gamma_{i}, \delta_{i}: i=1, \ldots, m(n)\}& \in \mathcal{S}^{4m(n)}\\
\end{aligned}
\]
such that the following hold: \begin{enumerate}
\item $\Prob(B_{n}) \ge 1 - (1-q)^{n}$ holds for each $n$.
\item On each equivalence class $\mathcal{F} \in \mathcal{Q}_{n}$, the isometries $(w_{i})_{i=0}^{m(n)}$ and the Schottky set $\mathcal{S}$ are constant, and $(\alpha_{i}, \beta_{i}, \gamma_{i}, \delta_{i})_{i=1}^{m(n)}$ are i.i.d.s distributed according to the uniform measure on $\mathcal{S}^{4}$.
\item The RV \[
w_{0} \Pi(\alpha_{1}) \Pi(\beta_{1})\Pi(\gamma_{1})\Pi(\delta_{1})w_{1} \cdots \Pi(\alpha_{m(n)}) \Pi(\beta_{m(n)}) \Pi(\gamma_{m(n)}) \Pi(\delta_{m(n)}) w_{m(n)}
\]
is distributed according to $\mu^{\ast n}$ on $\Omega_{n}$.
\end{enumerate}
\end{lem}

\begin{proof} 
Let $\mu_{1}$ be the uniform measure on $\Phi(S^{(4)})$ and $\mu_{2}$ be the uniform measure on $\Phi(\check{S}^{(4)})$. Recall that $\Phi(S^{4})$ and $\Phi(\check{S}^{4})$ are disjoint by Lemma \ref{lem:triviaSchottkyInv}.(3). Hence we have a decomposition 
\[
\mu = p \left( \frac{1}{2}\mu_{1} + \frac{1}{2} \mu_{2}\right)+ (1-p) \nu,
\]
Consider: \begin{itemize}
\item Bernoulli random variables $(\rho_{i})_{i \ge 0}$ with average $p$, 
\item Bernoulli random variables $(\sigma_{i})_{i \ge 0}$ with average $1/2$, 
\item $G$-valued random variables $(\eta_{i})_{i\ge 0}$ with the law $\mu_{1}$,
\item $G$-valued random variables $(\check{\eta}_{i})_{i\ge 0}$ with the law $\mu_{2}$, and 
\item $G$-valued random variables $(\nu_{i})_{i\ge 0}$ with the law $\nu$, 
\end{itemize}
all independent. We then define \[
g_{k}  = \left\{ \begin{array}{cc} \nu_{k} & \textrm{when}\,\, \rho_{k} = 0, \\
\eta_{k} & \textrm{when} \,\, \textrm{$\rho_{k}= 1$ and $\sigma_{k} =0$}, \\
\check{\eta}_{k} &  \textrm{when} \,\, \textrm{$\rho_{k}= 1$ and $\sigma_{k} =1$}.
\end{array}\right.
\]
Then $g_{k}$'s are i.i.d. distributed according to $\mu$. Setting $m(n)$ suitably, we define \[\begin{aligned}
C_{n} = \left\{\w : \sum_{k=1}^{n} \rho_{k} (1-\sigma_{k}) \ge m(n)\right\}, \quad \check{C}_{n} = \left\{\w : \sum_{k=1}^{n} \rho_{k} \sigma_{k} \ge m(n)\right\}.
\end{aligned}
\]
Note that $C_{n}$ and $\check{C}_{n}$ are determined by the values of $\{\rho_{k}, \sigma_{k}\}_{k}$. Moreover, the values of $\{\rho_{k}, \sigma_{k}\}_{k>0}$ determine the indices $i \in \{1, \ldots, n\}$ such that $\rho_{i} =1$ and $\sigma_{i} = 0$, and for each $\w \in C_{n}$ we can gather the $m(n)$ smallest such positive indices and label them by $\vartheta(1)< \ldots <\vartheta(m(n))$. For each $\w \in \check{C}_{n}$, we analogously define $\check{\vartheta}(1)< \ldots < \check{\vartheta}(m(n))$ to be the $m(n)$ smallest indices $i$ such that $\rho_{i} =1$ and $\sigma_{i} = 1$.

Now on $C_{n}$, we declare $\mathcal{S} := S$ and let $\mathcal{Q}_{n}^{(1)}$ be the measurable partition determined by the values of $\{\rho_{k}, \sigma_{k}, \nu_{k}\}_{k>0}$ as well as $\{\eta_{k} : k > \vartheta(m(n)) \}$. We then define
 \[\begin{aligned}
w_{i-1} &:= g_{\stopping(i-1) + 1} \cdots g_{\stopping(i)-1} \quad (i=1, \ldots, m(n)), \\
w_{m(n)} &:= g_{\stopping(m(n)) + 1} \cdots g_{n}, \\
(\alpha_{i}, \beta_{i}, \gamma_{i}, \delta_{i})&:= \Phi^{-1} (g_{\stopping(i)}) \quad (i=1, \ldots, m(n)).
\end{aligned}
\]
Then on each equivalence class in $\mathcal{Q}_{n}^{(1)}$, $w_{i}$'s are fixed and $(\alpha_{i}, \beta_{i}, \gamma_{i}, \delta_{i})$ are uniformly distributed on $S^{4m(n)}$. Moreover, on $B_{n}$ we have\begin{equation}\label{eqn:rwMatch}
\begin{aligned}
&g_{1} \cdots g_{n} \\
&= 
w_{0} \Pi(\alpha_{1}) \Pi(\beta_{1})\Pi(\gamma_{1})\Pi(\delta_{1})w_{1} \cdots \Pi(\alpha_{m(n)}) \Pi(\beta_{m(n)}) \Pi(\gamma_{m(n)}) \Pi(\delta_{m(n)}) w_{m(n)}.
\end{aligned}\end{equation}
Similarly, on $\check{C}_{n} \setminus C_{n}$, we first declare $\mathcal{S} := \check{S}$ and let $\mathcal{Q}_{n}^{(2)}$ be the measurable partition determined by the values of $\{\rho_{k}, \sigma_{k}, \nu_{k}\}_{k>0}$ as well as $\{\check{\eta}_{k} : k > \check{\vartheta}(m(n)) \}$. We then define
 \[\begin{aligned}
w_{i-1} &:= g_{\check{\stopping}(i-1) + 1} \cdots g_{\check{\stopping}(i)-1} \quad (i=1, \ldots, m(n)), \\
w_{m(n)} &:= g_{\check{\stopping}(m(n)) + 1} \cdots g_{n}, \\
(\alpha_{i}, \beta_{i}, \gamma_{i}, \delta_{i})&:= \Phi^{-1} (g_{\check{\stopping}(i)}) \quad (i=1, \ldots, m(n)).
\end{aligned}
\]
Then on each equivalence class in $\mathcal{Q}_{n}^{(2)}$, $w_{i}$'s are fixed and $(\alpha_{i}, \beta_{i}, \gamma_{i}, \delta_{i})$ become i.i.d. distributed according to the uniform measure on $\check{S}^{4}$. Moreover, Equation \ref{eqn:rwMatch} still holds on $\check{C}_{n}$. These constructions realize Item (1) and Item (2) of the conclusion on the set $B_{n} := C_{n} \cup \check{C}_{n}$.

Finally, we define $w_{i}$'s and $\{\alpha_{i}, \beta_{i}, \gamma_{i}, \delta_{i}\}$ arbitrarily on $(C_{n} \cup \check{C}_{n})^{c}$ so that Equation \ref{eqn:rwMatch} still holds. Since $g_{1} \cdots g_{n}$ is distributed according to $\mu^{n}$, we conclude Item (3) of the conclusion.

It remains  to precisely choose $\epsilon$ and $m(n)$. Note that $\sum_{k=1}^{n} \rho_{k} < 2m(n)$ holds  outside $C_{n} \cup \check{C}_{n}$. Chernoff-Hoeffding's inequality as in \cite[Theorem 1]{hoeffding} tells us that \[
\Prob\left( \sum_{k=1}^{n} \rho_{k} \le 2\epsilon n \right) \le\left(f(\epsilon) :=\left( \frac{p}{2\epsilon}\right)^{2\epsilon} \cdot \left(\frac{1-p}{1-2\epsilon} \right)^{1-2\epsilon} \right)^{n}.
\]
Since $\lim_{\epsilon \rightarrow 0} f(\epsilon) = 1-p$ is smaller than $1-q$, we can pick small $\epsilon$ so that $\Prob \left( \sum_{k=1}^{n} \rho_{k} \le 2\epsilon n \right) \le (1-q)^{n}$ holds. By setting $m(n) = \lfloor \epsilon n \rfloor$, we arrive at the desired conclusion.
\end{proof}

Recall that we obtained Corollary \ref{cor:pivotCorCombined} by combining Lemma \ref{lem:pivotFirstReduce}, \ref{lem:pivotSecondReduce}, \ref{lem:selfRepulse} and \ref{lem:trLengthPivotProb}. By replacing Lemma \ref{lem:pivotFirstReduce} with Lemma \ref{lem:pivotFirstReduceCount}, we obtain the following:

\begin{cor}\label{cor:pivotCorCombinedCount}
For each $0 < q < p < 1$ there exists a constant $\epsilon = \epsilon(p, q) > 0$ and $M = M(p, q) > 0$ such that the following holds.

Let $S\subseteq G^{M_{0}}$ be a fairly long $K_{0}$-Schottky set with cardinality $N_{0} > M$. Let $\mu$ be a probability measure on $G$ such that \[
\mu \ge p\cdot \big(\textrm{uniform measure on $\Phi(S^{4}) \cup \Phi(\check{S}^{(4)})$}\big).
\]
Then for each $n$ there exist an integer $m(n)$,  a probability space $\Omega_{n}$ with a measurable subset $A_{n} \subseteq \Omega_{n}$, a measurable partition $\mathcal{P}_{n}$ of $A_{n}$ and RVs \[
\begin{aligned}
Z_{n} &\in G, \\
\mathcal{S} &\in \{S, \check{S}\},\\
\{w_{i} : i = 0, \ldots, m(n)\} &\in G^{m(n) + 1}, \\
\{\beta_{i}, \gamma_{i} : i = 1, \ldots, m(n)\} &\in S^{2m(n)}
\end{aligned}
\]
such that the following hold: \begin{enumerate}
\item $\Prob(A_{n}) > 1-3\cdot(1-q)^{n}$ and $m(n) > \epsilon n$ for eventual $n$.
\item On $A_{n}$, $(w_{i})_{i=0}^{m(n)}$ is a $D_{0}$-pre-aligned sequence in $G$ and $w_{m(n)}^{-1} w_{0}$ is a $D_{0}$-pre-aligned isometry.
\item On each $\mathcal{E} \in \mathcal{P}_{n}$, $(w_{i})_{i=0}^{m(n)}$ and $\mathcal{S}$ are constant and $(\beta_{i}, \gamma_{i})_{i=1}^{m(n)}$ are i.i.d.s distributed according to the uniform measure on $\tilde{\mathcal{S}}(id)$.
\item $Z_{n}$ is distributed according to $\mu^{\ast n}$ on $\Omega_{n}$ and \[
Z_{n} = w_{0} \Pi(\beta_{1}) v_{1} \Pi(\gamma_{1}) w_{1} \cdots \Pi(\beta_{m(n)})v_{m(n)} \Pi(\gamma_{m(n)}) w_{m(n)}
\]holds on $A_{n}$.
\end{enumerate}
\end{cor}

\begin{proof}
First, given the inputs $p$ and $q$, Lemma \ref{lem:pivotFirstReduceCount} provides the constant $\epsilon_{1} = \epsilon(p, q)$, and we set $\epsilon = 0.01 \epsilon_{1}$. Now for each $n$ with $N_{1} :=\lfloor \epsilon_{1} n \rfloor$, Lemma \ref{lem:pivotFirstReduceCount} also provides  a probability space $\Omega_{n}$ with a subset $B=B_{n}$, a measurable partition $\mathcal{Q}_{n}$ of $B_{n}$, a random fairly long $K_{0}$-Schottky set $\mathcal{S} \in \{S, \check{S}\}$ of cardinality $N_{0}$, random variables $Z_{n} \in G$, $(w_{i})_{i=0}^{N_{1}} \in G^{N_{1} + 1}$ and $(\alpha_{i}, \beta_{i}, \gamma_{i}, \delta_{i})_{i=1}^{N_{1}} \in \mathcal{S}^{4N_{1}}$ such that the following holds: \begin{enumerate}
\item $\Prob_{n}(B) \ge 1 - (1-q)^{n}$.
\item On each $\mathcal{F} \in \mathcal{Q}_{n}$, the sequence $(w_{i})_{i=0}^{N_{1}}$ and the Schottky set $\mathcal{S}$ are constant and $(\alpha_{i}, \beta_{i}, \gamma_{i}, \delta_{i})_{i}$ are i.i.d.s distributed according to the uniform measure on $\mathcal{S}^{4}$.
\item $Z_{n}$ is distributed according to $\mu^{\ast n}$ on $\Omega_{n}$ and \[
Z_{n} = w_{0} \Pi(\alpha_{1}) \Pi(\beta_{1}) \Pi(\gamma_{1}) \Pi(\delta_{1}) w_{1} \cdots \Pi(\alpha_{N_{1}}) \Pi(\beta_{N_{1}}) \Pi(\gamma_{N_{1}}) \Pi(\delta_{N_{1}}) w_{N_{1}}
\]
holds on $B$.
\end{enumerate}
Now, as we did in the proof of Corollary \ref{cor:pivotCorCombined}, we apply Lemma \ref{lem:pivotSecondReduce}; this time, it is with the input probability $\operatorname{Dirac}_{id}$ in place of $\mu$, with the input Schottky set $\mathcal{S}$ in place of $S$, and with the input integer $N_{1}$ in place of $n$. As a result we obtain $N_{2} = 2\lfloor 0.25 N_{1} \rfloor$. Just as in the proof of Corollary \ref{cor:pivotCorCombined}, we obtain a measurable set $B'$ with measurable partition $\mathcal{Q}'$ and RVs $(w_{i}', \beta_{i}', \gamma_{i}')_{i}$ such that the following holds: \begin{enumerate}
\item $B'$ has probability at least $\left(1 - (3 \sqrt[4]{4/N_{0}})^{N_{1}} \right) \cdot \Prob(B) \ge \left(1 - (3 \sqrt[4]{4/N_{0}})^{N_{1}} \right) \cdot (1-(1-q)^{n})$,
\item On $B'$ we have \[
Z_{n} = w_{0}'  \Pi(\beta_{1}') \Pi(\gamma_{1}') w_{1}' \cdots \Pi(\beta_{N_{1}}') \Pi(\gamma_{N_{2}}') w_{N_{2}}',
\]
\item On each $\mathcal{F}' \in \mathcal{Q}'$, $(w_{i}')_{i=0}^{N_{2}}$ is a fixed $D_{0}$-pre-aligned sequence in $G$ for $\mathcal{S}$ and $(\beta_{i}', \gamma_{i}')_{i=1}^{N_{2}}$ are i.i.d.s distributed according to the uniform measure on $\tilde{\mathcal{S}}(id)$.
\end{enumerate}

Finally, by setting $m(n) := 2^{\lfloor \log_{2} N_{2} \rfloor - 1} \ge N_{2}/4$ and by considering the refinement of $\mathcal{Q}'$ as in the proof of Corollary \ref{cor:pivotCorCombined}, we can realize Item (2), (3), (4) of the conclusion on a subset $A_{n}$ of $B$ such that \[
\Prob(A_{n}) \ge \big(1- (6/N_{0})^{N_{2}/4}\big) \cdot \Prob(B').
\]
Moreover, note that $m(n) \ge N_{2}/4 \ge N_{1}/5 \ge \epsilon n$ for large enough $n$, which establishes the latter half of Item (1) of the conclusion.

At this point, by requiring \[
N_{0} > M:= \max\left(\left(\frac{3^{4} \cdot 4}{ (1-q)^{4}}\right)^{4/\epsilon_{1}}, \left(\frac{6}{1-q}\right)^{8/\epsilon_{1}}\right)+1,
\]
we conclude \[
\Prob(A_{n}) \ge 1-(1-q)^{n} - (3 \sqrt[4]{4/N_{0}})^{N_{1}}-( 6/N_{0})^{N_{2}/4} \ge 1-3\cdot(1-q)^{n}
\]
for large enough $n$. This settles Item (1) of the conclusion.
\end{proof}

We now combine Corollary \ref{cor:pivotCorCombinedCount} with Lemma \ref{lem:selfRepulse2}. Note that on $A_{n}$, $(w_{i})_{i=0}^{m(n)}$ is always a $D_{0}$-pre-aligned sequence and $\phi_{1} := w_{m(n)}^{-1} w_{0}$ is a $D_{0}$-pre-aligned isometry. Lemma \ref{lem:selfRepulse2} then implies the following on $A_{n}$: \[
Z_{n} = w_{0} \Pi(\beta_{1}) \Pi(\gamma_{1}) w_{1} \cdots \Pi(\beta_{m(n)}) \Pi(\gamma_{m(n)}) w_{m(n)}
\]
is a contracting isometry, and moreover, there exist Schottky axes $\kappa_{1}, \ldots, \kappa_{2m(n)}$ such that \[
\big( \ldots, Z_{n}^{-1} \kappa_{1}, \ldots, Z_{n}^{-1} \kappa_{2m(n)}, \kappa_{1}, \ldots, \kappa_{2m(n)}, Z_{n} \kappa_{1}, \ldots, Z_{n} \kappa_{2m(n)},\ldots \big)
\]
is $D_{0}$-semi-aligned, and that $(o, \kappa_{1}, \kappa_{2m(n)}, Z_{n} o)$ is also $D_{0}$-semi-aligned. Then Lemma \ref{lem:BGIPWitness} implies that for each $j$, $[o, Z_{n}^{j} o]$ contains $2m(n) \cdot j$ disjoint subsegments, each $0.1E_{0}$-fellow traveling with some $Z_{n}^{l} \kappa_{i}$ $(l=0, \ldots, j-1, i=1, \ldots, 2m(n)$). This implies that \[
d(o, W_{m(n)}^{j} o) \ge \sum_{i=1}^{2m(n)} \sum_{l=0}^{j-1} \Big( \diam(Z_{n}^{l} \kappa_{i} ) - 0.2E_{0}\Big) \ge 5 E_{0} m(n).
\] As a consequence, we obtain the following corollary: 

\begin{cor}\label{cor:countTr}
For each $0 < q < p < 1$ there exists a constant $\epsilon = \epsilon(p, q) > 0$ and $M = M(p, q) > 0$ such that the following holds.

Let $S\subseteq G^{M_{0}}$ be a fairly long $K_{0}$-Schottky set with cardinality $N_{0} > M$. Let $\mu$ be a probability measure on $G$ such that \[
\mu \ge p\cdot \big(\textrm{uniform measure on $\Phi(S^{4}) \cup \Phi(\check{S}^{(4)})$}\big),
\]
and let $(Z_{n})_{n>0}$ be the random walk generated by $\mu$. Then  \[
\Prob\Big(Z_{n}\textrm{is a contracting element with $\tau(Z_{n}) \ge 5 E_{0} \epsilon n$}\Big) \ge 1-3 (1-q)^{n}
\]
holds for large enough $n$.
\end{cor}

\section{Counting problem}\label{section:counting}

We now present a quantitative version of the main theorem in \cite{choi2021generic}.

\begin{thm}[Translation length grows linearly]\label{thm:genericityGen}
Let $(X, G)$ be as in Convention \ref{conv:main}. Then for each $\lambda>1$, there exists $\lambda_{0}>0$ satisfying the following. Let $G$ be a finitely generated non-elementary subgroup of $\Isom(X)$ and $S' \subseteq G$ be a finite symmetric generating set. 

Then there exist a subset $S''$ of $G$ containing $S'$ such that $\#S'' \le (1+\lambda) \#S' + \lambda_{0}$, and a constant $K>0$, such that for each large $n$ we have\[
\frac{\#\{g \in B_{S''}(n) : g \,\,\textrm{is not contracting or}\,\,\tau_{X}(g) \le Kn\}}{\#B_{S''}(n)} \le K e^{-n/K}.
\]
\end{thm}

\begin{proof}
First note that we can replace any finite symmetric generating set $S'$ with $S' \cup \{id\}$ and then find $S''$, at the cost of increasing $\lambda_{0}$ by 1. For this reason, we from now on assume that $S'$ contains $id$.

The assumption $\lambda > 1$ implies that $\frac{\lambda}{1+\lambda} > 1/2$, so we can take parameters $q$ and $p$ such that $1/2 < q < p < \frac{\lambda}{1+\lambda}$. Take $\epsilon = \epsilon(p, q)$ and $M = M(p, q)$ as in Corollary \ref{cor:countTr}. 

Note that the function $f(x) := \frac{x}{x+1}$ is strictly increasing on the positive reals and $\lim_{x \rightarrow +\infty} f(x) = 1$. Meanwhile, note the equivalence \[
p(1+\lambda) < \lambda \Leftrightarrow p < \lambda(1-p) \Leftrightarrow \frac{p}{\lambda(1-p)} < 1.
\]
Since we took $p < \frac{\lambda}{1+\lambda}$, there exists some $\lambda_{0} > M^{4}$ such that $f(\sqrt[4]{\lambda_{0}}) > \sqrt[4]{p/\lambda(1-p)}$. Furthermore, since $1/2p$ is also smaller than $1$, we can further require that $f(\sqrt[4]{\lambda_{0} }-1) > \sqrt[4]{1/2p}$.

Using Lemma \ref{lem:Schottky}, we take a fairly long Schottky set $S$ in $G$ with cardinality \[
N_{0} = \max \left( \left\lfloor\sqrt[4]{\frac{1}{2} \lambda \#S'} \right\rfloor, \sqrt[4]{\lambda_{0}}\right).
\]
Let $K_{0}$ be the Schottky parameter of $S$ and $D_{0}, D_{1}, E_{0}$ be the associated constants.

Note that \[
2N_{0}^{4} \le 2 \max \left( \frac{1}{2} \lambda \#S', \lambda_{0}\right) \le \lambda \#S' + \lambda_{0}.
\] Meanwhile, using the monotonicity of $f(x) = \frac{x}{x+1}$ we deduce \begin{align}\label{eqn:monotonF1}
\frac{N_{0}^{4}}{( N_{0} + 1)^{4}} &= \big(f(N_{0})\big)^{4} \ge \Big(f \left(\sqrt[4]{\lambda_{0}}\right)\Big)^{4} \ge \frac{p}{\lambda (1-p)}, \\ \label{eqn:monotonF2}
\frac{(N_{0}-1)^{4}}{N_{0} ^{4}} &= \big(f(N_{0} - 1)\big)^{4}\ge \Big(f \left( \sqrt[4]{\lambda_{0}} - 1 \right) \Big)^{4} \ge 1/2p.
\end{align}
Note that Inequality \ref{eqn:monotonF1} implies \[
2N_{0}^{4} = 2(N_{0}+1)^{4} \cdot \frac{N_{0}^{4}}{( N_{0} + 1)^{4}}  \ge \lambda\#S' \cdot \frac{p}{\lambda (1-p)}\ge \frac{p}{1-p} \#S'.
\]

With these estimates, we set $S'' = S' \cup \Phi(S^{4}) \cup \Phi(\check{S}^{4})$, which is still a symmetric set. Moreover, we have \[
\#S'' \le \#S' + 2 (\#S)^{4} = \#S' + 2N_{0}^{4} \le (1+\lambda)\#S'^{2} + \lambda_{0}
\]
and \[
\frac{\# \left( \Phi(S^{4}) \cup \Phi(\check{S}^{4}) \right)}{\#S''} = 
\frac{2N_{0}^{4} }{\#S''} \ge  \frac{2N_{0}^{4}}{2N_{0}^{4} + \#S'} \ge \frac{1}{1 + \frac{1-p}{p}} = p.
\]
This implies that the uniform measure $\mu$ on $S''$ satisfies the assumption of Corollary \ref{cor:countTr}. At this moment, let \[
\mathcal{A}_{n} = \{ g \in G : \textrm{$g$ is contracting and $\tau(g) \ge 5E_{0} \epsilon n$}\}.
\]
Then Corollary \ref{cor:countTr} tells us that, for large enough $n$, the $n$-th step $Z_{n}$ of the $\mu$-random walk is in $\mathcal{A}_{n}$ except for probability at most $3(1-q)^{n}$. 

Meanwhile, each element $g \in B_{S''}(id, n)$ arises as an $n$-th step $Z_{n}$ of the $\mu$-random walk, with probability at least $(\#S'')^{-n}$. Hence, we have \[
\# \left( B_{S''}(id, n) \setminus \mathcal{A}_{n} \right) \le \frac{\Prob(Z_{n} \notin \mathcal{A}_{n})}{(\#S'')^{-n}} \le 3(\#S'')^{n} (1-q)^{n}.
\]

Meanwhile, let $\alpha_{1}, \ldots, \alpha_{4n}$ of $S$ be such that $\alpha_{i} \neq \alpha_{i+1}$ for $i=1, \ldots, 4n-1$, and let $\epsilon_{1}, \ldots, \epsilon_{n} \in \{1, -1\}$. Lemma \ref{lem:concatSchottky}.(2) guarantees that the isometries \[
\prod_{i=1}^{4n} \Pi(\alpha_{i}^{\epsilon_{\lceil i/4 \rceil}}) = \Pi(\alpha_{1}^{\epsilon_{1}}) \cdots \Pi\big(\alpha_{4}^{\epsilon_{1}}\big) \cdot \Pi(\alpha_{5}^{\epsilon_{2}}) \cdots \Pi(\alpha_{8}^{\epsilon_{2}}) \cdots \Pi(\alpha_{4n-3}^{\epsilon_{n}}) \cdots \Pi(\alpha_{4n}^{\epsilon_{n}})
\]
are distinct for different choices of such $\big((\alpha_{i})_{i=1}^{4n}, (\epsilon_{j})_{j=1}^{n} \big)$'s. All these isometries are contained in $B_{S''}(id, n)$. The number of such elements is at least \[\begin{aligned}
(\# S - 1)^{4n} \cdot 2^{n} &= 2^{n} (N_{0} - 1)^{4n} = \left(\frac{2N_{0}^{4}}{\#S''} \cdot (\#S'') \cdot \left( \frac{N_{0}-1}{N_{0}} \right)^{4} \right)^{n}&\\
 &\ge \left(p \cdot \frac{1}{2p} \cdot (\# S'') \right)^{n}. & (\because \textrm{Inequality \ref{eqn:monotonF2}})
\end{aligned}
\]
Hence, for large enough $n$ we have \[
\frac{\# (B_{S''}(id, n) \setminus \mathcal{A}_{n})}{\#B_{S''}(id, n)} \le \frac{3 (1-q)^{n} (\# S'')^{n}}{(1/2)^{n} (\#S'')^{n} } = 3 \cdot (2(1-q))^{n}
\]
This decays exponentially since $1-q < 1/2$ as desired.
\end{proof}

Using a similar argument that involves pivoting for quasi-isometric embedding of $k$ independent random walks (Lemma \ref{lem:multiRepulse} and \ref{lem:multiPivotProb}), we can deduce the following version of Theorem \ref{thm:qiCount}:

\begin{thm}\label{thm:qiCountingStrong}
For each $k \in \Z_{>0}$ and $\lambda>1$, there exists $\lambda_{0}>0$ satisfying the following. Let $G$ be a finitely generated non-elementary subgroup of $\Isom(X)$ and $S' \subseteq G$ be a finite symmetric generating set. 

Then there exists a set $S'' \supseteq S'$ of $G$ with $\#S' \le (1+\lambda) \# S' + \lambda_{0}$ such that for all $k$-tuples $(g_{1}, \ldots, g_{k})$ of elements in $B_{S''}(n)$ except an exponentially decaying proportion, $\langle g_{1}, \ldots, g_{k} \rangle$ is q.i. embedded into a quasi-convex subset of $X$.
\end{thm}

\medskip
\bibliographystyle{alpha}
\bibliography{Combined_v3}

\end{document}